\documentclass[EJP]{ejpecp} 

\usepackage[english]{babel}
\usepackage[utf8x]{inputenc}
\usepackage[T1]{fontenc}
\usepackage{dsfont}
\usepackage{xspace}
\usepackage{amscd}
\usepackage{todonotes}
\usepackage{enumerate}
\usepackage{indentfirst}
\usepackage{url}
\usepackage{array}
\usepackage{multirow}
\usepackage{float}
\usepackage{color}
\usepackage{layout}
\usepackage{mathabx}

\SHORTTITLE{Eq. fluct. for the disordered harm. chain perturbed by an energy conserving noise} 
\TITLE{Equilibrium fluctuations for the disordered harmonic chain perturbed by an energy conserving noise}

\AUTHORS{%
  Cl\'ement~Erignoux\footnote{Dipartimento di Matematica e Fisica, Univ. Roma Tre, Largo San Leonardo Murialdo, 1, 00146, Roma -- \texttt{clement.erignoux@mat.uniroma3.it}}
  \and 
  Marielle~Simon\footnote{INRIA Univ. Lille, CNRS, UMR 8524 - Laboratoire Paul Painlev\'e, F-59000 Lille  -- \texttt{marielle.simon@inria.fr}}
  }


\AMSSUBJ{Primary 60K35 ; Secondary 82C22} 

\SUBMITTED{December 7, 2018} 
\ACCEPTED{November 27, 2019} 

\VOLUME{0}
\YEAR{2012}
\PAPERNUM{0}
\DOI{vVOL-PID}

\ABSTRACT{We investigate the macroscopic energy diffusion of a disordered harmonic chain of oscillators, 
whose hamiltonian dynamics  is perturbed by a degenerate conservative noise. 
After rescaling space and time diffusively, we prove that the equilibrium energy fluctuations evolve according to a linear heat equation. 
The diffusion coefficient is obtained from Varadhan's non-gradient approach, and is equivalently defined through the Green-Kubo formula. 
Since the perturbation is very degenerate and the symmetric part of the generator does not have a spectral gap,  
the standard non-gradient method is reviewed under new perspectives.}

\newcommand{\N}{\mathbb{N}}                                              
\newcommand{\Z}{\mathbb{Z}}                                              
                                              
\newcommand{\Q}{\mathcal{Q}}                                              
\newcommand{\R}{\mathbb{R}}                                              

\newcommand{\T}{\mathbb{T}}
\renewcommand{\P}{\mathbb{P}}
\newcommand{\E}{\mathbb{E}}
\newcommand{\Y}{\mathcal Y}

\newcommand{\A}{\mathcal{A}}

\renewcommand{\S}{\mathcal{S}}
\newcommand{\cL}{\mathcal{L}}

\newcommand{\cQ}{\mathcal{Q}}
\newcommand{\cH}{\mathcal{H}}

\newcommand{\cD}{\mathcal{D}}
\newcommand{\cC}{\mathcal{C}}
\renewcommand{\le}{\leqslant}
\renewcommand{\ge}{\geqslant}
\renewcommand{\tilde}{\widetilde}
\newcommand{\m}{\mathbf{m}}

\renewcommand{\bar}{\overline}

\newcommand{\benum}{\begin{enumerate}}
\newcommand{\eenum}{\end{enumerate}}
\newcommand{\bitem}{\begin{itemize}}
\newcommand{\eitem}{\end{itemize}}

\newcommand{\barray}{\begin{array}}
\newcommand{\earray}{\end{array}}

\newcommand{\ps}[1]{\langle #1\rangle}
\newcommand{\bps}[1]{\big\langle #1\big\rangle}

\newcommand{\vertiii}[1]{{\big\vert\kern-0.25ex\big\vert\kern-0.25ex\big\vert #1 
    \big\vert\kern-0.25ex\big\vert\kern-0.25ex\big\vert}}

\newtheorem{theo}{\textsc{Theorem}}[section]
\newtheorem{cor}[theo]{\textsc{Corollary}}
\newtheorem{lem}[theo]{\textsc{Lemma}}
\newtheorem{prop}[theo]{\textsc{Proposition}}

\newtheorem{de}{\textsc{Definition}}[section]

\theoremstyle{remark}

\newtheorem{rem}{\textsc{Remark}}[section]
\newtheorem{ex}{\textsc{Example}}[section]
\newtheorem{exe}{\textsc{Exercice}}[section]

\newcommand{\btheo}[1]{\begin{theo} #1\end{theo}}
\newcommand{\bprop}[1]{\begin{prop} #1\end{prop}}
\newcommand{\bcor}[1]{\begin{cor} #1\end{cor}}
\newcommand{\bprf}[1]{\begin{proof} #1\end{proof}}
\newcommand{\brem}[1]{\begin{rem} #1\end{rem}}

\newcommand{\blem}[1]{\begin{lem} #1\end{lem}}
\newcommand{\bde}[1]{\begin{de} #1\end{de}}

\newcommand{\ccl}[1]{{#1}}

\begin{document}

\thanks{\footnotesize \textbf{Acknowledgements.} This problem was suggested by C\'edric Bernardin, and we are grateful to him for helpful and valuable 
remarks. We warmly thank Makiko Sasada and Stefano Olla for their interest and constructive discussions on this work. 
M.S. also would like to express her great appreciation to the anonymous referees of the first form of this article, who significantly helped improve the final version.\\
This work has been partially supported by the fellowship L'Or\'eal France-UNESCO \emph{For Women in Science}. The authors thank Labex CEMPI (ANR-11-LABX-0007-01) for its support. C.E. acknowledges support
from the European Research Council (ERC) under the European Union's
Horizon 2020 research and innovation programme (ERC CoG UniCoSM, grant
agreement n\textsuperscript{o}724939). M.S. acknowledges support from the ERC under  the European Union's Horizon 2020 research and innovative programme (grant agreement  n\textsuperscript{o}715734), and from the project EDNHS ANR-14-CE25-0011 of the French National Research Agency (ANR).
}

\tableofcontents

\section{Introduction}

This paper deals with diffusive behaviour in heterogeneous media for interacting particle systems. More precisely, we address the problem of 
energy fluctuations for chains of oscillators with random defects. In the last fifty years, it has been recognized that introducing disorder 
in interacting particle systems has a drastic effect on the conduction properties of the material \cite{CL}. The most mathematically 
tractable model of oscillators is the one-dimensional system with harmonic interactions \cite{MR2784282}. The anharmonic case is poorly understood from 
a mathematical point of view, but since the works of Peierls \cite{P1,P2}, it is admitted that non-linear interactions between atoms should play a crucial 
role in the derivation of the Fourier law. In \cite{MR2480742, MR2185330,MR2082192} (among many others) the authors propose to model anharmonicity by 
stochastic perturbations, in order to recover the expected macroscopic behavior: in some sense, the noise  simulates the effect of non-linearities.  Being inspired by all these previous works, the aim of this paper is to prove 
the diffusive energy behavior of disordered harmonic chains perturbed by an energy conserving noise.  Moreover we prove that all the disorder effects  are, on a 
sufficiently large scale, contained in a diffusion coefficient, which depends on the statistics of the field, but not on the randomness itself.

On the one hand, the disorder effect has already been investigated for \emph{lattice gas dynamics}: the first article dealing with scaling limits of particle systems in random 
environment is the remarkable work of Fritz \cite{F}, and since then the subject has attracted a lot of interest, see for example 
\cite{MR2021195,MR2446327,MR2540160,MR2271489}. These papers share one main feature: the models are \emph{non-gradient}\footnote{Roughly speaking, the \emph{gradient property} states that the microscopic current (of density, or energy, depending on the conservation law under consideration) can be decomposed as a local gradient. We refer to Section \ref{ssec:current} for more details.} due to the presence of the disorder. 
Except in \cite{F}, non-gradient systems are usually solved by establishing a microscopic Fourier law up to a small fluctuating term, following the 
sophisticated method initially developed by Varadhan in \cite{MR1354152}, and generalized to non-reversible dynamics in \cite{MR1653397}.  
These previous works mostly deal with systems of particles that evolve according to an exclusion process in random environment: the particles are attempting 
jumps to nearest neighbour sites at rates which depend on both their position and the objective site, and the rates themselves come from a quenched 
random field. Different approaches are adopted to tackle the non-gradient feature: whereas the standard method of Varadhan is helpful  in dimension 
$d \geqslant 3$ only (see \cite{MR2021195}), the ``long jump'' variation developed by Quastel in  \cite{MR2271489} is valid in any dimension. 
The study of disordered chains of oscillators perturbed by a conservative noise has appeared more recently, see for instance 
\cite{MR2448630, MR3101849, DLK}. In all these papers, the thermal conductivity is defined by the Green-Kubo formula only.  
Here, we also define the diffusion coefficient  through hydrodynamics and we prove that both definitions are equivalent. 

On the other hand, the study of \emph{one-dimensional chains of oscillators} is an active field of research. In \cite{MR3810840}, the authors derive the diffusive scaling limit for a \emph{homogeneous} (without disorder) chain of coupled harmonic oscillators, perturbed by a noise which 
randomly flips the sign of the velocities (called \emph{velocity-flip noise}), so that the energy is conserved but not the momentum. 
We want to investigate here the scaling limit of equilibrium fluctuations for the same chain of harmonic oscillators, still perturbed by the velocity-flip noise, but now provided with i.i.d.~random masses. In \cite{Sim}, for the same model, an exact \textit{fluctuation-dissipation relation} 
(see for example \cite{MR1465163})  reproduces the Fourier law at the microscopic level. With random masses, however, the fluctuation-dissipation equations are no longer directly 
 solvable. We therefore adapt Varadhan's non-gradient approach, which allows one to show that an approximate fluctuation-dissipation decomposition holds. The main ingredients of the usual non-reversible non-gradient method are: first, a 
 \emph{spectral gap estimate} for the symmetric part of the dynamics, and second, a \emph{sector condition} for the total generator. 

The rigorous study of the disordered harmonic chain perturbed by the velocity-flip noise contains three major obstacles: (i) first, the symmetric part of 
the generator (which, in our case, comes only from the stochastic noise) is poorly ergodic, and does not have a spectral gap when restricted to micro-canonical manifolds. 
This issue is usually critical to apply Varadhan's  method ;
(ii) second, the degeneracy of the perturbation implies that the asymmetric part of the generator cannot be controlled by its symmetric part (in technical terms, 
the sector condition does not hold) ; (iii) finally, the energy current depends on the disorder, and has to be approximated 
by a fluctuation-dissipation equation which takes into account the fluctuations of the disorder itself.

\bigskip

To overcome the second obstacle (ii), namely the lack of sector condition due to the high degeneracy of the velocity-flip noise, we add a second stochastic perturbation, that exchanges velocities (divided by the square 
root of mass) and positions at random independent Poissonian times, so that a kind of sector condition can be proved (see Proposition \ref{prop:sector}: 
we call it the \emph{weak sector condition}). However, the spectral gap estimate and the usual sector condition still do not hold when adding the exchange 
noise, meaning that the stochastic perturbation remains very degenerate; in other words, the noises are still far from ergodic. 
To sum up, the final model that we rigorously investigate here is: the coupled harmonic chain with  random masses, perturbed by two degenerate stochastic noises, one which exchanges velocities and positions, the other one which flips the sign of velocities. 
The main results and contributions of this article are \begin{itemize}
\item an adaptation of the non-gradient method to a microscopic model for which neither the spectral gap inequality nor the sector condition hold (Proposition \ref{prop:macro2} and Lemma \ref{lem:def_d}), which makes no use of the closed forms theory. In particular, the main novelty is Lemma \ref{lem:firststep};
 \item the macroscopic behavior of the equilibrium fluctuations of the energy, which is diffusive, with a diffusion coefficient $D$ depending only on the statistics of the random masses field (Theorem \ref{theo:fluctuations});
 \item the equivalence between two definitions of the diffusion coefficient: the one obtained via Varadhan's approach, and the one obtained from the Green-Kubo formula (Theorem \ref{theo:GKF}). 
\end{itemize}
Our model has one crucial feature, that makes an adaptation of Varadhan's approach possible despite the lack of spectral gap of the symmetric part of the generator: 
thanks to the harmonicity of the chain, the  generator of the dynamics preserves homogeneous polynomials together 
with their degree. In particular, the derivation of the  sector condition and the non-gradient decomposition of closed discrete differential forms at the center of the non-gradient method (see 
\cite{MR2895558, MR1707314} for more details) can be carried out rather explicitly in a suitable space of quadratic functions, without need for a spectral gap. 
Although some further complications appear in the study of our model, this is one of its clear advantages. It allows us to avoid significant 
technical difficulties usually inherent to Varadhan's approach, and to adapt the latter to a model with a very degenerate noise. 
In particular, if the chain is not assumed to be harmonic, a stronger noise than ours is generally needed: the one proposed by Olla and Sasada in \cite{sasolla} 
is strong enough so that the spectral gap and the sector 
condition  hold, and they were able to use ideas from  Varadhan's approach to determine the scaling limit of the equilibrium fluctuations. 
Our purpose here is to show, using elements of the non-gradient method as well, that in the presence of i.i.d.~random masses, the annealed (i.e. averaged out over the masses' randomness) equilibrium fluctuations of the energy evolve following 
an infinite Ornstein-Uhlenbeck process. The covariances characterizing this linearized heat equation are given in terms 
of the diffusion coefficient, which is defined through a variational formula.
We opted for a rather detailed redaction, even if some proofs may look standard to expert readers. We hope that this 
choice will be beneficial for the reader not already familiar with the non-gradient  method. 
 
Finally, we also show that the diffusion coefficient can be equivalently given by the Green-Kubo formula. The latter is defined as the space-time 
variance of the current at equilibrium, which is only formal in the sense that a double limit (in space and time) has to be taken. As in \cite{MR2448630}, 
where the disordered harmonic chain is perturbed by a stronger energy conserving noise, we prove here that the limit exists, and that the homogenization 
effect occurs for the Green-Kubo formula: for almost every realization of the disorder, the thermal conductivity exists, is independent of the disorder, 
is positive and finite. This allows us to prove that the diffusion coefficient $D$ obtained through the variational formula in Varadhan's method, and the coefficient
$\overline D$ defined through the Green-Kubo formula, are actually equal: $ D=\overline D$. 


\bigskip

To conclude this introduction, we introduce in more details the model on which this article focuses. As explained earlier, we consider here an infinite harmonic hamiltonian system 
described by the sequence $\{p_x,r_x\}_{x \in \Z}$, where $p_x$ stands for the momentum of the oscillator at site $x$, and $r_x$ represents 
the distance between oscillator $x$ and oscillator $x+1$.  Each atom $x \in \Z$ has a mass $M_x >0$, thus the velocity of atom $x$ is given by $p_x/{M_x}$.
We assume the disorder $\mathbf M:=\{M_x\}_{x \in \Z}$ to be a collection of real i.i.d.~positive random variables such that 
\begin{equation} 
\forall \ x \in \Z,\quad \frac1C \leqslant M_x \leqslant C, \label{eq:bb}
\end{equation} 
for some finite constant $C>0$.   The equations of motions are given by 
\begin{equation}
\left\{ \begin{aligned} \frac{\text d p_x}{\text d t}&=r_{x}-r_{x-1}, \\
 \frac{\text d r_x}{\text d t}&=\frac{p_{x+1}}{M_{x+1}}-\frac{p_{x}}{M_{x}}, \end{aligned}\right. \label{eq:syste}
\end{equation}
so that the dynamics conserves the total energy \[ \mathcal{E}:=\sum_{x\in\Z} \bigg\{\frac{p_x^2}{2M_x} + \frac{r_x^2}{2}\bigg\}.\]
To overcome the  lack of ergodicity of deterministic chains\footnote{For the deterministic system of harmonic oscillators, it is well known that the energy is ballistic, destroying the validity of the Fourier law. For more details, see the remarkable work of Lebowitz, Lieb and Rieder \cite{LLR}, which is the standard reference.}, we add a stochastic perturbation to \eqref{eq:syste}.  The noise can be easily described: at independently distributed random Poissonian times, the quantity $p_x/\sqrt{M_x}$ and the interdistance $r_x$ are exchanged, or the momentum $p_x$ is flipped into $-p_x$. This noise still conserves the total energy $\mathcal{E}$, and is very degenerate. The main goal of this paper is to prove that the energy fluctuations in equilibrium converge in a suitable space-time scaling limit (Theorem \ref{theo:fluctuations}).

Even if Theorem \ref{theo:fluctuations} could be proved \emph{mutatis mutandis} for this harmonic chain described by $\{p_x,r_x\}$,  for pedagogical reasons we now  focus on a simplified model, which has  the same features and involves simplified computations\footnote{We invite the reader to see \cite{MR2904271} for the origin of this new particle system.}.  From now on, we study the dynamics on new configurations $\{\eta_x\}_{x\in\Z} \in \R^{\Z}$ written as 
\begin{equation} 
{m}_x \text d \eta_x = (\eta_{x+1}-\eta_{x-1}) \text d t, \label{eq:motion2} 
\end{equation}
where $\m:=\{m_x\}_{x \in \Z}$ is the new disorder with the same characteristics as in \eqref{eq:bb}. 
It is notationally convenient to change the variable $\eta_x$ into $\omega_x:=\sqrt{ m_x} \eta_x$, so that the total energy reads \[ \mathcal{E}=\sum_{x\in\Z} \omega_x^2.\] Let us now introduce the corresponding stochastic energy conserving dynamics: the evolution is described by \eqref{eq:motion2} between random exponential times, and at each ring one of the following interactions can happen:

\emph{a. Exchange noise -- } the two nearest neighbour variables $\omega_x$ and $\omega_{x+1}$ are exchanged; 

\emph{b. Flip noise -- }  the variable $\omega_x$ at site $x$ is flipped into $-\omega_{x}$.

\noindent As a consequence of these two perturbations, the dynamics only  conserves the total energy, the other important conservation laws of the hamiltonian part being destroyed by the stochastic noises\footnote{It is now well understood that the ballisticity of the harmonic chain is due to the infinite number of conserved quantities. In 1994, Fritz, Funaki and Lebowitz \cite{FFL} propose  different stochastic noises that are sufficient to destroy the ballisticity of the chain: the \emph{velocity-flip} noise is one of them.}. It is not difficult to check that the following family  $\{\mu_\beta\}_{\beta > 0}$ of \textit{grand-canonical Gibbs measures} on $\R^\Z$ is invariant for the  resulting process $\{\omega_x(t) \; ; \; x \in \Z, t\geq 0\}$:  
\begin{equation} 
\text d \mu_\beta(\omega):=\prod_{x \in \Z} \sqrt{\frac{\beta}{2\pi}}\exp\left(-\frac\beta 2 \omega^2_x\right)\text d \omega_x. 
\label{eq:mesinva}
\end{equation}
The index $\beta$ stands for the inverse temperature. Note that with our notational convenience, $\mu_\beta$ does not depend on the disorder $\m$. Observe also that the dynamics is not reversible with respect to the measure $\mu_\beta$. We define $\mathbf e_\beta:=\beta^{-1}$ as the thermodynamical energy associated to $\beta$, namely the expectation of $\omega^2_0$ with respect to $\mu_\beta$, and $\chi(\beta)=2\beta^{-2}$ as the static compressibility, namely the variance of $\omega^2_0$ with respect to $\mu_\beta$.  

To state the convergence result, let us define the distribution-valued energy fluctuation field, as follows: at time 0,  it is given by
\[ \mathcal{Y}_0^N:=\frac{1}{\sqrt N} \sum_{x\in\Z} \delta_{x/N} \big\{ \omega^2_x(0)-\mathbf{e}_\beta \big\}, \] 
where $\delta_u$ is the Dirac measure at point $u\in\mathbb{R}$. We assume that the dynamics is at equilibrium, namely that $\{\omega_x(0)\}_{x\in\Z}$ is distributed according to 
the Gibbs measure $\mu_\beta$. It is well known that $\mathcal Y_0^N$ converges in distribution as $N\to\infty$ towards a  centered Gaussian field $\mathcal Y$, 
which satisfies \[\E_{\mathcal{Y}}\big[ \mathcal{Y}(F) \mathcal{Y}(G)\big] = \chi(\beta) \int_\R F(y) G(y) \text dy,\] for continuous test functions $F,G$. 
One of the main results of this article, Theorem \ref{theo:fluctuations} below,  states that the energy fluctuations evolve diffusively in time: starting from  $\mu_\beta$, and averaging over the disorder $\m$, 
the energy  field 
\[ \mathcal Y_t^N=\frac{1}{\sqrt N}\sum_{x\in\Z} \delta_{x/N} \big\{\omega^2_x(tN^2)-\mathbf e_\beta\big\} \] 
converges in distribution as $N \to \infty$ to the solution of the linear Stochastic Partial Differential Equation (SPDE) 
\[ \partial_t \mathcal Y=D \partial_y^2 \mathcal Y +\sqrt{2D \chi(\beta)} \partial_y W, \qquad t>0, y\in \R. \] 
where $D $ is the diffusion coefficient, defined by variational formula (see Definition \ref{def:diffusion} below), and $W$ is the standard normalized space-time white noise. Let us note here that the coefficient $D$ does not depend on $\beta$. This claim is proved below in Section \ref{ssec:products} (cf. Definition \ref{def:diffusion}). 

Finally, one could think of using the well-known \textit{entropy method} \cite{GPV} to further derive the \textit{hydrodynamic equation}: in that case, the initial law is not assumed to be the equilibrium measure $\mu_\beta$, but a \emph{local equilibrium measure}. We conjecture that this property of local equilibrium propagates in time, and that an \textit{hydrodynamic limit} result holds. Let $\mathbf e_0: \T\to\R$ be a bounded function, where $\T$ denotes the torus $[0,1)$. One would like  to show that the empirical energy profile $\frac1N\sum_x\delta_{x/N}\; \omega_x^2(tN^2)$ converges  as $N\to\infty$ to the macroscopic profile $\mathbf e(t,\cdot):\T\to\R$ solution to
\[\left\{\begin{aligned} \frac{\partial \mathbf e}{\partial t}(t,u)  & = D \frac{\partial^2 \mathbf e}{\partial u^2}(t,u),  \qquad \qquad t>0, \ u \in \T, \\
\mathbf e(0,u)& =\mathbf e_0(u). \end{aligned}\right.\]
Unfortunately, even if the diffusion coefficient $D$ is well defined through the non-gradient approach, this does not straightforwardly provide a method to prove such a result\footnote{More details about the main obstacles to prove hydrodynamic limits are available in the \texttt{ArXiv} version of this paper: \href{https://arxiv.org/abs/1402.3617}{https://arxiv.org/abs/1402.3617}}.

\bigskip

Let us now give the outline of the article.   Section \ref{sec:model} is devoted to properly introducing the model and all  definitions that are needed, as well as stating our main results, namely Theorem \ref{theo:fluctuations} (macroscopic equilibrium fluctuations) and Theorem \ref{theo:GKF} (Green-Kubo formula) below.  
The first theorem, namely the convergence of the energy fluctuations field  (in the sense of finite dimensional distributions) is proved in Section \ref{ssec:martingale}, up to technical results which are postponed in further sections. Indeed, the main point is to identify the diffusion coefficient $D$  by adapting the non-gradient method of Varadhan \cite{MR1354152}. 
This is done in several steps: in Section \ref{sec:CLTvariances_equ}, we derive the so-called \textit{Boltzmann-Gibbs principle} stated in Proposition \ref{prop:macro2}. In Section \ref{sec:hilbert} we obtain  the diffusion coefficient as resulting from a projection of the current in some suitable Hilbert space, 
and finally Section \ref{sec:diffusion} improves the description of the diffusion coefficient through several variational formulas (see Remark \ref{rem:DDt}). 
In Section \ref{sec:gk} we prove the second theorem in two steps: first, we show the convergence of the Green-Kubo formula (Proposition \ref{prop:conv_GK}), and  second, we  demonstrate rigorously that both definitions of the diffusion coefficient are equivalent (Proposition \ref{prop:equivalence}). 
Finally, in Section \ref{sec:anharmonic}, we present a second disordered model, where the interaction is described by a potential $V$ which is not assumed to be harmonic anymore. 
For this anharmonic chain, we need a very strong stochastic perturbation, which has a spectral gap, and satisfies the sector condition.  
In Appendices, technical points are detailed: in Appendix \ref{sec:hermite}, the space of square-integrable functions w.r.t.~the standard Gaussian law is studied through  
its orthonormal basis of Hermite polynomials.   The sector condition (Proposition \ref{prop:seccond}) is proved in Appendix \ref{app:sector} for a specific class of functions suitable for our needs.   
In Appendix \ref{sec:tightness}, tightness for the energy fluctuation field (Proposition \ref{prop:tight}) is investigated for the sake of completeness. 

\section{The harmonic chain perturbed by stochastic  noises\label{sec:model}}

\subsection{Generator of the Markov process \label{sec:generator}} 

Let us define the dynamics on the finite torus $\T_N:=\Z/N\Z$, meaning that boundary conditions are periodic. 
The space of configurations is given by $\Omega_N=\R^{\T_N}$. The configuration $\{\omega_x\}_{x\in\T_N}$ evolves according to a dynamics which can be divided into two parts, 
a deterministic one and a stochastic one.  The disorder is an i.i.d.~sequence ${\mathbf m}=\{ m_x\}_{x\in \T_N}$ which satisfies:  
\begin{equation*} \forall \ x \in \T_N,\quad \frac1C \leqslant m_x \leqslant C, \end{equation*} 
for some finite constant $C>1$. 
The corresponding product and translation invariant measure on the space $\Omega_N^{\cD}=[C^{-1},C]^{\T_N}$ is denoted by $\P^N$ and its expectation is denoted by $\E^N$.  Their infinite volume counterparts on $\Z$ are denoted by $\Omega^{\cD}$,  $\P$ and $\E$.

For a fixed disorder field $\mathbf m=\{m_x\}_{x\in\T_N}$, we consider the system of equations
\begin{equation*} \sqrt{{m}_x} \text d \omega_x = \bigg(\frac{\omega_{x+1}}{\sqrt{m_{x+1}}}-\frac{\omega_{x-1}}{\sqrt{m_{x-1}}}\bigg) \text d t, \qquad t \geqslant 0, \ x \in \T_N,  \end{equation*}
 and we superpose to this deterministic dynamics a stochastic perturbation described as follows: 
 with each atom $x \in \T_N$ (respectively each bond $\{x,x+1\}, x \in \T_N$) is associated an exponential clock of rate $\gamma >0$ (resp. $\lambda>0$), 
and all clocks are independent one from another. When the clock attached to the atom $x$ rings, $\omega_x$ is flipped into $-\omega_x$. 
 When the clock attached to the bond $\{x,x+1\}$ rings, the values $\omega_x$ and $\omega_{x+1}$ are exchanged. 
 This dynamics can be equivalently defined by the generator $\cL_N^{\mathbf{m}}$ of the Markov process $\{\omega_x(t)\, ; \, x\in\T_N\}_{t\geqslant 0}$, which is written as
\begin{equation*}{\cL_N^{\mathbf m}= \A_N^{\mathbf m}+\gamma \S_N^{\text{flip}}+\lambda \S_N^{\text{exch}}},\end{equation*}
where 
\[
\A_N^\m  = \sum_{x \in \T_N} \bigg(\frac{\omega_{x+1}}{\sqrt{m_xm_{x+1}}}-\frac{\omega_{x-1}}{\sqrt{m_{x-1}m_x}}\bigg) \frac{\partial}{\partial \omega_x}, \]
and, for all functions $f:\Omega^{\cD}_N\times \Omega_N \to \R,$
\begin{align*}
\S_N^{\text{flip}} f(\m,\omega)&= \sum_{x\in\T_N}\Big(f(\m,\omega^x)-f(\m,\omega)\Big),\\
\S_N^{\text{exch}} f(\m,\omega) &=\sum_{x \in \T_N} \Big(f(\m,\omega^{x,x+1})-f(\m,\omega)\Big).
.\end{align*}
We  denote  $\cL^\m$, $\A^{\m}$, $\S^{\text{flip}}$ and $\S^{\text{exch}}$ their infinite volume counterparts for which $\T_N$ has been replaced by $\Z$. In the formulas above, the configuration $\omega^{x}$ is the configuration obtained from $\omega$ by flipping the value at site $x$:  
\[(\omega^x)_z = \begin{cases} \omega_z &\text{ if } z \neq x,\\ -\omega_x & \text{ if } z=x, \end{cases}\] 
and the configuration $\omega^{x,x+1}$ is obtained from $\omega$ by exchanging the values at sites $x$ and $x+1$:  
\[ (\omega^{x,x+1})_z= \begin{cases} \omega_z & \text{ if } z \neq x,x+1, \\ \omega_{x+1} & \text{ if } z=x, \\ \omega_x &\text{ if } z=x+1. \end{cases} \] 
We denote the total generator of the noise by $\S_N:=\gamma \S_N^{\text{flip}}+\lambda \S_N^{\text{exch}}$ (and similarly $\S=\gamma \S^{\text{flip}}+\lambda \S^{\text{exch}}$ in infinite volume). 

It is straightforward to see that the total energy $\sum_{x\in \T_N}\omega_x^2$ is conserved by the dynamics and that the following translation 
invariant product Gibbs measures $\mu^N_\beta$ on $\Omega_N$ are invariant for the Markov process:
\[\text d\mu_\beta^N(\omega):=\prod_{x \in \T_N} \sqrt{\frac{\beta}{2\pi}} \exp\left(-\frac{\beta}{2}\omega_x^2\right) \text d\omega_x.\] 
The index $\beta$ \ccl{is a constant that} stands for the inverse temperature, \ccl{and satisfies} $\int \omega_0^2 \text{d}\mu_\beta^N=\beta^{-1}$. Let us note that the Gibbs measures 
do not depend on the disorder $\mathbf m$.
From the definition, our model is not reversible with respect to the measure $\mu_\beta^N$. More precisely, $\A_N^{\mathbf m}$ is an antisymmetric operator in 
$\mathbf{L}^2(\mu_\beta^N)$, whereas $\S_N$ is symmetric.

\ccl{
We denote by $\mu_\beta$ the product Gibbs measure on the infinite configuration space $\Omega:=\R^\Z$.} The scalar product in $\mathbf{L}^2(\mu_\beta)$ is denoted by $\langle \cdot, \cdot \rangle_\beta$. 
Moreover, we denote by $\P^\star_\beta$ \ccl{(resp. $\P^N_\beta$)} the probability measure on $\ccl{\Omega^{\cD}}\times\Omega$ \ccl{(resp. $\Omega^{\cD}_N\times\Omega_N$)} defined by 
\begin{equation}
\label{eq:DefPbstar}
\P^\star_\beta:=\P\otimes\mu_{\beta}\ccl{\quad (\mbox{resp.}\quad \P^N_\beta:=\P^N\otimes\mu^N_{\beta}\;)}.
\end{equation}
Throughout this article we will frequently use the fact  that $\P^\star_\beta$ \ccl{and $\P^N_\beta$} are translation invariant. We write $\E^\star_\beta$ \ccl{(resp. $\E^N_\beta$)} for the corresponding expectation, and $\E^\star_\beta[\cdot,\cdot]$ for the scalar product in $\mathbf{L}^2(\P^\star_\beta)$.  We also define the {static compressibility} which is equal to the variance of the one-site energy $\omega_0^2$ with respect to $\mu_\beta$, namely
\begin{equation} 
\label{eq:Defchi}
\chi(\beta):=\langle \omega_0^4 \rangle_\beta - \langle \omega_0^2\rangle_\beta^2=\frac{2}{\beta^2}. 
\end{equation}

\subsection{Energy current} \label{ssec:current}
Since the dynamics \ccl{locally conserves the energy, one can write 
\begin{equation}
\label{eq:Lj}
\cL_N^{\mathbf m}(\omega^2_x)=j_{x-1,x}(\mathbf m,\omega)-j_{x,x+1}(\mathbf m,\omega).
\end{equation} 
where $j_{x,x+1}$ is} the instantaneous \ccl{net  amount of  energy flowing from the particle $x$ to the particle $x+1$}, and is equal to  
\begin{equation*} 
j_{x,x+1}(\mathbf m,\omega)=-\frac{2\omega_x\omega_{x+1}}{\sqrt{m_xm_{x+1}}} + \lambda(\omega_{x}^2-\omega_{x+1}^2). 
\end{equation*} 
We write $j_{x,x+1}=j_{x,x+1}^A+j_{x,x+1}^S$ where $j_{x,x+1}^A$ (resp. $j_{x,x+1}^S$) is the current associated to the antisymmetric 
(resp. symmetric) part of the generator: 
\begin{align}
j_{x,x+1}^A(\mathbf m,\omega)&=-\frac{2\omega_x\omega_{x+1}}{\sqrt{m_xm_{x+1}}} \label{eq:DefjA}\\
j_{x,x+1}^S(\mathbf m,\omega)&=j_{x,x+1}^S(\omega)=\lambda(\omega_{x}^2-\omega_{x+1}^2)\label{eq:DefjS}.
\end{align}
\ccl{Note that \eqref{eq:Lj} \eqref{eq:DefjA} and \eqref{eq:DefjS} still hold if $\cL_N^{\mathbf m}$ is replaced by its infinite volume counterpart $\cL^{\mathbf m}$ and $x\in \T_N$ by $x\in \Z$, so that we will, without further mention, identify the instantaneous currents in a periodic setting and in infinite volume.}
As mentioned in the introduction, this model is \emph{non-gradient}, i.e.~the current  cannot be directly written  as the gradient of a local function. 
Moreover, there is not an exact \textit{fluctuation-dissipation equation}, as in \cite{Sim}.%

\subsection{Main results}
\label{sec:fluctuation}

Let us state our main results on the fluctuations of the empirical energy around equilibrium, and the Green-Kubo formula for the diffusion coefficient. 

Recall \ccl{that we introduced after \eqref{eq:mesinva} the thermodynamical energy $\mathbf{e}_\beta=\beta^{-1}$ associated to $\beta>0$. }  
We define the energy empirical distribution $\pi_{t,\m}^N$ on the continuous torus $ \T=\R/\Z$ as
\begin{equation*} 
\pi_{t,\m}^N(\text du)=\frac 1 N \sum_{x \in \T_N} \omega^2_x(t) \delta_{x / N} (\text du), \quad t \in [0,T], \quad u \in \T, 
\end{equation*}
where $\delta_u$ stands for the Dirac measure at point $u$, and where   $\{\omega(t)\}_{t\geqslant 0}$ is the Markov process generated 
by $N^2\cL_N^{\m}$. If the initial state of the dynamics is  the equilibrium Gibbs measure $\mu_\beta^N$, then, 
for any fixed $t\geqslant 0$, and any disorder $\m \in \ccl{\Omega_N^{\cD}}$, the measure $\pi_{t,\m}^N$ weakly converges  
towards the measure $\{\mathbf e_\beta \text du\}$ on $\T$, which is deterministic and with constant density w.r.t. the Lebesgue measure on $\T$. 
Here we investigate the fluctuations of the empirical measure $\pi_{t,\m}^N$ with respect to this limit.

\bde{[Energy fluctuation field]
We denote by $\mathcal Y_{t,\m}^N$ the empirical energy fluctuation field associated with the Markov process $\{\omega(t)\}_{t\geq 0}$ 
generated by $N^2\cL_N^{\m}$ \ccl{(note the diffusive time acceleration)} and starting from $\ccl{\P_\beta^N=\P^N \otimes \mu_\beta^N}$, defined by its action over  test functions $H\in C^2(\T)$,} 
\begin{equation*}
\Y_{t,\m}^N(H)=\frac 1{\sqrt N}\sum_{x\in \T_N} H\left(\frac x N \right) \left(\omega_x^2(t)-\mathbf e_\beta\right).
\end{equation*}
We will prove that the annealed distribution of $\Y_{t,\m}^N$ converges in distribution towards the solution to the linear SPDE: \begin{equation} 
\partial_t\Y=D\,\partial^2_y\Y  + \sqrt{2D\chi(\beta)}\, \partial_y W, \label{eq:spde}
\end{equation}  where $W$ is a standard normalized space-time white noise, and $D$ is the diffusion coefficient defined in Definition \ref{def:diffusion} below.  More precisely, the solution to \eqref{eq:spde} is the stationary generalized Ornstein-Uhlenbeck process \ccl{(cf. \cite{HS})} with zero mean and covariances given by \begin{equation*}
\E_{\Y}\left[\Y_t(H)\Y_0(G)\right]=\frac{\chi(\beta)}{\sqrt{4\pi t D}}\int_{\R^2}  \overline{H}(u)\overline G(v) \exp\left(-\frac{(u-v)^2}{4tD}\right)\, \text{d}u \,\text{d}v,
\end{equation*} for all $t \geqslant 0$ and test functions $H,G \in C^2(\T)$. Here and after, $\overline H$ (resp. $\overline G$) is the periodic extension to the real line of $H$ (resp. $G$).
Let us fix a time horizon $T>0$. The probability measure on the Skorokhod space $\cD([0,T],\ccl{\Omega_N})$ induced by the Markov process $\{\omega(t)\}_{t\geq 0}$ 
generated by $N^2\cL_N^{\m}$ and starting from \ccl{$\P_\beta^N=\P^N \otimes \mu_\beta^N$ is  denoted by $\ccl{\mathbb Q_{\mu_\beta^N}}$. In order to avoid burdening notations, we also denote by $\ccl{\mathbb Q_{\mu_\beta^N}}[f]$ the expectation of the function $f$ w.r.t.~$\ccl{\mathbb Q_{\mu_\beta^N}}$}.

\medskip

Consider for $\ccl{k>0}$ the Sobolev space  $\mathfrak{H}_{-k}$ \ccl{(e.g. see  \cite{Brezis} for an exhaustive reference)}  of distributions $\Y$ on $\T$ with finite norm 
\[
\Vert\Y \Vert^2_{-k}=\sum_{n\geqslant 1} (\pi n)^{-2k} \big| \Y({e}_n) \big|^2,
\] 
where ${e}_n$ is the function $x \mapsto \sqrt 2 \sin(\pi n x)$. 
We denote by $\mathfrak Y^N$ the annealed probability measure on the space $\cD([0,T],\mathfrak{H}_{-k})$ 
of \ccl{c\`adl\`ag} trajectories on the Sobolev space,  induced by the Markov process $\{\omega(t)\}_{t\geq 0}$ and the mapping $\Y^N:(\m,\omega)\mapsto \{\Y_{t,\m}^N\}_{0\leq t\leq T}$. In other words, we define 
\[\mathfrak Y^N(\Y\in \cdot)=\ccl{\mathbb{Q}_{\mu_\beta^N}} \circ (\Y^N)^{-1}(\,\cdot\,).\]
\ccl{Note that for any fixed $N$, $\P_\beta^N$--a.s., the total initial energy is finite. Therefore,    $\mathbb{Q}_{\mu_\beta^N}$--a.s. the quantities $\omega^2_x(t)$ are bounded uniformly in $x\in \T_N$ and $t\geq 0$. In particular, $\mathbb{Q}_{\mu_\beta^N}$--a.s., the trajectory $\{\Y_{t,\m}^N\}_{t\geq 0}$ is in $ \cD([0,T],\mathfrak{H}_{-k})$.}

We let $\mathfrak Y$ be the probability measure on the space $\cC([0,T],\mathfrak{H}_{-k})$ 
corresponding to the generalized Ornstein-Uhlenbeck process $\Y_t$ solution to \eqref{eq:spde}. The \ccl{first} main result of \ccl{this paper} is the following.

\btheo{[\ccl{Equilibrium fluctuations}]\label{theo:fluctuations} Fix $k > 5/2$ and $T>0$. The sequence  $\{\mathfrak Y^N\}_{N\geqslant 1}$ weakly converges in $\cD([0,T],\mathfrak{H}_{-k})$  to the probability measure $\mathfrak Y$.} 

\ccl{
The second main result is the so-called \emph{Green-Kubo formula}, which gives an alternative formula for the diffusion coefficient $D$.
\btheo{[\ccl{Green-Kubo formula}]\label{theo:GKF}The diffusion coefficient $D$ introduced in Definition \ref{def:diffusion} below satisfies 
\[D=\lim_{z\downarrow 0}\left\{\lambda+\frac{1}{2} \int_0^{+\infty} \text dt \ e^{-z t} \sum_{x \in \Z} \ccl{\mathbb Q_{\mu_1^N}}\Big[j_{0,1}^A(\m,\omega(t)), \tau_xj_{0,1}^A(\m,\omega(0))\Big]\right\}.\]} 
}

\ccl{To prove the first result, w}e follow the lines of  \cite[Section 3]{sasolla}. The proof of Theorem \ref{theo:fluctuations} is divided into three steps. First, we need to show that the sequence $\{\mathfrak Y^N\}_{N\geqslant 1}$ is tight. 
This point  follows a standard argument, given for instance in \cite[Section 11]{MR1707314}, and recalled in Appendix \ref{sec:tightness} for the sake of completeness. 
Then, we prove \ccl{(in Section \ref{ssec:martingale})} that any limit point $\mathfrak Y^\star$ of $\{\mathfrak Y^N\}_{N\geqslant 1}$ is concentrated on trajectories whose marginals at time $t$ have, for any $t\in [0, T]$, the distribution of a centered Gaussian field with covariances given by 
\begin{equation}
\label{eq:covariance}
\mathfrak Y^\star\big[\Y_t(H)\Y_t(G)]=\chi(\beta)\int_\T  H(u) G(u)\, \text{d}u ,
\end{equation} where $H,G\in C^2(\T)$ are  test functions.
Since $\mu_\beta^N$ is stationary for the process $\ccl{\{\omega(t)\}_{t\geq 0}}$, this statement comes from the Central Limit Theorem for independent variables. 
Finally, we prove the main ingredient: all limit points $\mathfrak Y^\star$ of the sequence 
$\{\mathfrak Y^N\}_{N\geqslant 1}$ solve the  martingale problems \eqref{eq:mart1} and \eqref{eq:mart2} given below, namely, 
for any function $H\in C^2(\T)$, \begin{equation}
\label{eq:mart1} \mathfrak M_t(H):=\Y_t(H)-\Y_0(H)-\int_0^tD\Y_s(H'')\text ds, \end{equation}
and \begin{equation}
\label{eq:mart2} \mathfrak N_t(H):=\left(\mathfrak M_t(H)\right)^2-2t\chi(\beta)D\int_\T H'(u)^2\text du\end{equation}
are $\mathbf L^1(\mathfrak Y^\star)$-martingales. \ccl{This proves as wanted that $\mathfrak Y^\star=\mathfrak Y$ (see for instance  \cite[Chapter 11, Theorem 0.2, p. 289]{MR1707314}).} \ccl{All the ingredients to prove Theorem \ref{theo:fluctuations} are put together in Section \ref{ssec:martingale}.}

\bigskip 
\ccl{
To prove Theorem \ref{theo:GKF}, we will show that the limit $z\downarrow 0 $ in the right hand side exists (Proposition \ref{prop:conv_GK}), and  then prove that the limit coincides with \eqref{eq:GCDiff} (Proposition \ref{prop:equivalence}).
}

\subsection{Notations and definitions}

\subsubsection{Cylinder functions\label{ssec:cylinder}}

For every $x \in \Z$ and every   measurable function $f$ on $\ccl{\Omega^{\cD}} \times \Omega$, 
we define the translated function $\tau_xf$ on $\ccl{\Omega^{\cD}}\times\Omega$  by: $\tau_xf(\mathbf m,\omega):=f(\tau_x\mathbf m,\tau_x\omega)$, 
where $\tau_x\mathbf m$ and $\tau_x\omega$ are the disorder and particle configurations translated by $x\in\Z$, respectively: 
\begin{equation*}
{(\tau_x\mathbf m)_z:=m_{x+z}, \qquad (\tau_x\omega)_z=\omega_{x+z}.} 
\end{equation*} 
 Let $\Lambda$ be a finite subset of $\Z$, and denote by $\mathcal{F}_\Lambda$ 
 the $\sigma$-algebra generated by $\{m_x, \omega_x \ ; \ x \in \Lambda\}$. 
 For a fixed positive integer $\ell$, we define $\Lambda_\ell:=\{-\ell,...,\ell\}$. 
 If the box is centered at site $x\neq 0$, we denote it by $\Lambda_\ell(x):=\{-\ell+x,...,\ell+x\}$. 
If $f$ is a measurable function on $\ccl{\Omega^{\cD}}\times\Omega$, the {support} of $f$, denoted by $\Lambda_f$, 
is the smallest subset of $\Z$ such that $f(\mathbf m,\omega)$ only depends on $\{m_x, \omega_x \ ; \ x \in \Lambda_f\}$ and $f$ 
is called a {cylinder (or local) function} if $\Lambda_f$ is finite. In that case, we denote by $s_f$ the smallest positive integer $s$ 
such that \ccl{ $\Lambda_f\subset\Lambda_{s}$}. 
For every cylinder function $f:\ccl{\Omega^{\cD}}\times\Omega\to\R$, consider the formal sum 
\begin{equation*}
\Gamma_f:=\sum_{x \in \Z} \tau_x f
\end{equation*}
 which is ill defined, but for which both gradients
 \begin{align*}\nabla_0(\Gamma_f)&:=\Gamma_f(\mathbf m,\omega^0)-\Gamma_f(\mathbf m,\omega),\\
\nabla_{0,1}(\Gamma_f)&:=\Gamma_f(\mathbf m,\omega^{0,1})-\Gamma_f(\mathbf m,\omega),
\end{align*} 
only involve a finite number of non-zero contributions and are therefore well defined.  Similarly, we define for any $x\in \mathbb{T}_N$
\begin{align*} 
(\nabla_x f)(\mathbf m,\omega)&:=f(\mathbf m,\omega^x)-f(\mathbf m,\omega), \\
(\nabla_{x,x+1}f)(\mathbf m,\omega)&:=f(\mathbf m,\omega^{x,x+1})-f(\mathbf m,\omega). 
\end{align*}

\begin{de}[\ccl{Space $\cC$}]
\label{de:C}
We denote by $\cC$ the set of  cylinder functions $\varphi$ on $\ccl{\Omega^{\cD}}\times\Omega$, such that
\ccl{\begin{align}
\label{eq:P1}
\tag{P1}
&\mbox{$\forall\ \omega \in \Omega$, \; $\m \mapsto \varphi(\m,\omega)$ is continuous on $\ccl{\Omega^{\cD}}$,}\\
\label{eq:P2}
\tag{P2}
&\mbox{$\forall\ \m\in\ccl{\Omega^{\cD}}$,  $\omega \mapsto \varphi(\m,\omega)$ is in $\mathbf L^2(\mu_{\beta})$,} \\
\label{eq:P3}
\tag{P3}
&\mbox{$\forall\ \m\in\ccl{\Omega^{\cD}}$, $\omega \mapsto \varphi(\m,\omega)$ and has mean zero with respect to $\mu_\beta$.} \end{align}
Note that any cylinder function can be identified with a function on $\Omega^{\cD}_N\times\Omega_N$ for $N$ large enough, and vice-versa.}
\end{de} 
\ccl{\begin{rem} It is straightforward to check that any cylinder function $\varphi \in \cC$ also belongs to $\mathbf L^2(\P_\beta^\star).$ To prove this claim, assume that it is not the case. Since $\varphi$ is a cylinder function, the disorder $\mathbf m$ is integrated over a compact set $[C^{-1},C]^{\Lambda_\varphi}$, and therefore there must exist a convergent sequence $\mathbf m_k\to\mathbf m_*$ such that $\int |\varphi(\mathbf m_k, \omega)|^2 \text d\mu_\beta(\omega)\to\infty$.
Then, denoting $\Omega^A=[-A,A]^{\Lambda_\varphi}$, one can write,  by the monotone convergence result and Moore-Osgood's Theorem:
\begin{align*}\int |\varphi(\mathbf m_*, \omega)|^2\; \text d\mu_\beta(\omega)&=\lim_{A\to\infty} \int_{\Omega^A} |\varphi(\mathbf m_*, \omega)|^2 \; \text d \mu_\beta(\omega)=\lim_{A\to\infty} \lim_{k\to\infty}\int_{\Omega^A} |\varphi(\mathbf m_k, \omega)|^2\; \text d\mu_\beta(\omega)\\
=&\lim_{k\to\infty}\lim_{A\to\infty} \int_{\Omega^A} |\varphi(\mathbf m_k, \omega)|^2 \; \text d\mu_\beta(\omega)=\lim_{k\to\infty}\int |\varphi(\mathbf m_k, \omega)|^2\; \text d\mu_\beta(\omega)=\infty,\end{align*}
which contradicts \eqref{eq:P2}. \end{rem}
}

\begin{de}[\ccl{Space $\cQ$}]
\label{de:quadratic} 
We introduce \ccl{a specific set of quadratic} cylinder functions on $\ccl{\Omega^{\cD}}\times\Omega$, denoted by $\Q \subset \cC$, and defined as follows: $f \in \Q$ if there exists  a sequence $\big\{\psi_{i,j}(\mathbf m)\big\}_{i,j\in\Z}$ 
of  cylinder  functions on $\ccl{\Omega^{\cD}}$, with finite support in $\ccl{\Omega^{\cD}}$, such that 
 \begin{enumerate}[(i)]
 \item for all $i,j \in \Z$ and all $\omega \in \Omega$,  the random variable $\m \mapsto \psi_{i,j}(\m)$ is continuous on $\ccl{\Omega^{\cD}}\ $;
 \item  $\psi_{i,j}$ vanishes for all but a finite number of pairs  $(i,j)$, and is symmetric: $\psi_{i,j}=\psi_{j,i}\ $; 
 \item  $f$ is written as \begin{equation}
f(\mathbf m,\omega)=\sum_{i \in \Z} \psi_{i,i}(\m) (\omega_{i+1}^2-\omega_i^2)+\sum_{\substack{i,j\in\Z\\ i\neq j}} \psi_{i,j}(\mathbf m)\omega_i\omega_j. \label{eq:quadratic}
\end{equation} 
\end{enumerate}
\end{de}
\ccl{One easily checks that $\cQ$ is invariant under the action of the generator $\cL_N^\m$ and that any function in $\cQ$ is an homogeneous polynomials of degree 
two in the variable $\omega$, with mean zero with respect to $\mu_\beta$ for every $\m\in\ccl{\Omega^{\cD}}$}. Another definition through \emph{Hermite polynomials} 
is given in Appendix \ref{sec:hermite} (see Section \ref{sub:quad-herm}). 
We are now ready to define two sets of functions that will play a crucial role later on.

\bde{[\ccl{Spaces $\cC_0$, $\cQ_0$}]

\label{def:space0}
Let $\cC_0$  be the set of cylinder functions $\varphi$ on $\ccl{\Omega^{\cD}}\times \Omega$  such that there exists  a finite subset $\Lambda$ of $\Z$,  and  cylinder, measurable functions $\{F_x, G_x\}_{x\in\Lambda}$ defined on $\ccl{\Omega^{\cD}}\times \Omega$, that verify  
\begin{equation*}\varphi=\sum_{x \in \Lambda} \Big\{\nabla_x(F_x)+\nabla_{x,x+1}(G_x)\Big\}, 
\end{equation*}
and such that, for all $x \in \Lambda$\ccl{, the functions $F_x$ and $G_x$ satisfy \eqref{eq:P1} and \eqref{eq:P2} in Definition \ref{de:C}.}

Let $\Q_0 \subset \cC_0$ be the set of such functions $\varphi$, with the additional assumption that  the cylinder  functions $F_x$, $G_x$ are homogeneous polynomials of degree two in the variable $\omega$. 
\ccl{One easily checks that, by construction, $\cC_0 \subset \cC$ and $\Q_0 \subset \cQ$.}
}
%

\subsubsection{Dirichlet form and properties of $\cC_0$ and $\cQ_0$}

Before giving the main properties of the sets introduced above, we introduce the usual quadratic form associated to the generator.
Given $\Lambda_\ell=\{-\ell, \dots, \ell\}$, we define   $\cL^\m_{\Lambda_\ell}$, resp.  $\S_{\Lambda_\ell}$, as the restriction of the generator $\cL^\m$, resp. $\S$, to $\Lambda_\ell$. 
For the jump dynamics $\S^{\text{exch}}$, we consider in $\S_{\Lambda_\ell}$ only the jumps with both ends in $\Lambda_\ell$, namely 
\[(\S_{\Lambda_\ell}) f(\m,\omega)=\gamma\sum_{x \in \Lambda_\ell} (\nabla_x f) (\m,\omega)+\lambda \sum_{x \in\Lambda_\ell\setminus\{\ell\}} (\nabla_{x,x+1} f) (\m,\omega).\]
Finally, for any $x \in \Z$, any cylinder functions $f,g \in \cC$, and any  positive integer $\ell$, let us define
\begin{equation}
\label{eq:DefDir}
\mathcal{D}_\ell(\mu_\beta ; f):=\big\langle (-\cL^\m_{\Lambda_\ell}) f, f \big\rangle_\beta=\big\langle (-\S_{\Lambda_\ell}) f, f \big\rangle_\beta.  
\end{equation}
Since 
\begin{equation*} 
\big\langle \nabla_x f,g \big\rangle_{\beta}  = - \frac{1}{2}\big\langle \nabla_x f,\nabla_x g \big\rangle_{\beta}\quad \mbox{ and } \quad   \big\langle \nabla_{x,x+1} f,g \big\rangle_{\beta}  = -\frac{1}{2}\big\langle \nabla_{x,x+1} f,  \nabla_{x,x+1}g \big\rangle_{\beta},
\end{equation*}
equation \eqref{eq:DefDir} in particular rewrites
\begin{equation}
\label{eq:dirich}
\mathcal{D}_\ell(\mu_\beta ; f) = \frac{\gamma}{2} \sum_{x\in\Lambda_\ell}  \big\langle \left(\nabla_x f\right)^2 \big\rangle_\beta + \frac{\lambda}{2} \sum_{x\in\Lambda_\ell \setminus \{\ell\}}  \bps{\left(\nabla_{x,x+1} f\right)^2}_\beta.
\end{equation}
The symmetric form $\cD_\ell$ is called the \emph{Dirichlet form}, and is well defined on $\cC$. It is a random variable with respect to the disorder $\mathbf m$. Note that, since the symmetric part of the generator does not depend on $\m$, we can write \ccl{(cf. \eqref{eq:DefPbstar})} 
\[
 \E \big[\cD_\ell(\mu_\beta;f)\big]=\cD_\ell(\P_\beta^\star;f).
\]
\bprop{\label{prop:propertiesC0} 
The following elements belong to $\Q_0$: \begin{align*}
(a) \quad & \mbox{\ccl{The instantaneous currents }} j_{0,1}^S,  \ j_{0,1}^A \ccl{\mbox{ defined in \eqref{eq:DefjA} and \eqref{eq:DefjS}}}.\\
(b) \quad & \cL^\m f,   \  \S f \text{ and } \A^\m f, \ \text{ for all } f \in \Q.
\end{align*}
}

\begin{proof}
The first statement $(a)$ is directly obtained from the following identities: for  $x \in \Z$, and $k \geqslant 1$, \begin{align}
\omega_{x+1}^2 - \omega_x^2 &=  \nabla_{x,x+1} \left(\omega_x^2\right) \label{eq:equalities_C0_1}\\
\omega_x\omega_{x+k} & = -\frac{1}{2}\nabla_x\big(\omega_x\omega_{x+1} \big)+\sum_{\ell=1}^{k-1}\nabla_{x+\ell, x+\ell+1}\big(\omega_x\omega_{x+\ell}\big).
\label{eq:equalities_C0_2}
\end{align}
Then, if $f \in \Q$  is of the form \eqref{eq:quadratic}, it is easy to see that \eqref{eq:equalities_C0_1} and \eqref{eq:equalities_C0_2} are sufficient to prove $(b)$. For instance, \begin{multline*}
\cL^\m(\omega_x\omega_{x+1})=\frac{\omega_x\omega_{x+2}}{\sqrt{m_{x+1}m_{x+2}}}-\frac{\omega_{x+1}\omega_{x-1}}{\sqrt{m_xm_{x-1}}}+\frac{\omega_{x+1}^2-\omega_x^2}{\sqrt{m_xm_{x+1}}} \\ -4\gamma\omega_x\omega_{x+1}+\lambda(\omega_{x+2}-\omega_{x+1})\omega_x+\lambda(\omega_{x-1}-\omega_x)\omega_{x+1}.
\end{multline*}
The integrability and regularity conditions are straightforward.
\end{proof}

\bprop{[Dirichlet bound]
\label{lem:ibp}
Let $\varphi$ be a cylinder function in $\cC_0$, written by definition as 
\[
\varphi=\sum_{x \in \Lambda} \Big\{ \nabla_x(F_x) + \nabla_{x,x+1}(G_x)\Big\},
\]
for some  $\Lambda \subset \Z$ and some functions $F_x$ and $G_x$ satisfying \ccl{ conditions \eqref{eq:P1} and \eqref{eq:P2} in Definition \ref{de:C}}. 
Let us consider $h\in \cC$ with support in $\Lambda_\ell$. Denote by $\ell_\varphi$ the integer $\ell_\varphi:=\ell-s_\varphi-1$ so that the supports of \ccl{$ \tau_y\varphi$ and its gradients $\nabla_{x,x+1}(\tau_y \varphi)$
are included in $\Lambda_\ell$ for every $y \in \Lambda_{\ell_\varphi}$. }

Then, there exists a positive constant $C(\varphi,\gamma)$ which depends only on $\varphi$ and $\gamma$ such that 
\begin{equation}\label{eq:dirich-bd}
\bigg|\E_\beta^\star\bigg[\sum_{|x| \leqslant \ell_\varphi} \tau_x\varphi,h\bigg]\bigg| \leqslant C(\varphi,\gamma) \; \big(\cD_\ell(\P^\star_\beta;h)\big)^{1/2}.
\end{equation}
}

\bprf{
Let us assume first that $\varphi=\sum_{y\in \Lambda} \nabla_y(F_y)$. Then we have 
\begin{align}
\bigg|\E_\beta^\star\bigg[\sum_{|x| \leqslant \ell_\varphi} \tau_x\varphi,h\bigg]\bigg| & = \bigg|\sum_{|x|\leqslant \ell_\varphi}\sum_{y \in \Lambda}\E_\beta^\star\big[ \tau_x F_y,\nabla_{y+x} h\big]\bigg| \notag\\
& \leqslant \sum_{|x|\leqslant \ell_\varphi}\sum_{y\in\Lambda}\E_\beta^\star\big[ (\tau_xF_y)^2\big]^{1/2} \; \E_\beta^\star\big[ \big(\nabla_{x+y} h\big)^2\big]^{1/2}\notag\\
&  \leqslant \sqrt{2\gamma^{-1}}\; \ccl{|\Lambda|} \; |2\ell_\varphi +1 |^{1/2} \Big( \sup_{y \in \Lambda} \E_\beta^\star\big[ F_y^2\big]\Big)^{1/2}  \big(\cD_\ell(\P_\beta^\star;h)\big)^{1/2}, \label{eq:ippbound}
\end{align}
\ccl{where $|\Lambda|$ denotes the cardinal of the set $\Lambda$.} In the previous inequalities we used the Cauchy-Schwarz inequality twice, and the fact coming from \eqref{eq:dirich} that 
\[
\sum_{|x|\leqslant \ell_\varphi} \E_\beta^\star\big[\big(\nabla_x h\big)^2\big]\leqslant \frac{2}{\gamma} \; \cD_\ell(\P_\beta^\star;h).
\]
 Therefore, the following constant 
\[
C(\varphi,\gamma):= \sqrt{2\gamma^{-1}} \; \ccl{|\Lambda|} \; |2\ell_\varphi +1 |^{1/2} \bigg\{\Big( \sup_{x \in \Lambda} \E_\beta^\star\big[ F_x^2\big]\Big)^{1/2} + \Big(\sup_{x \in \Lambda} \E_\beta^\star\big[ G_x^2\big] \Big)^{1/2} \bigg\}
\]
satisfies \eqref{eq:dirich-bd}.
The general case is similar and left to the reader.}

The main focus of this paper will be on the following quantities: for any $\varphi \in \cC$ let us define 
\begin{equation}
\label{def:varfor}
\E_\beta^\star\big[ \varphi, (-\S)^{-1} \varphi\big]  \ccl{: =} \sup_{h \in \cC} \Big\{ 2 \E_\beta^\star\big[ \varphi, h \big] - \cD_{s_h}(\P_\beta^\star ; h)\Big\}\in [0,+\infty].
\end{equation}
By polarization, this definition can be extended to give a meaning to $ \E_\beta^\star\big[ \varphi, (-\S)^{-1} \psi\big], $ 
for any $\varphi, \psi \in \cC$.
As a consequence of the previous results, these quantities are well defined for functions in $\cC_0$:

\bcor{\label{cor:wd}For every function $\varphi \in \cC_0$, the quantity $\E_\beta^\star[\varphi, (-\S)^{-1} \varphi]$ is finite. }

\bprf{This is a consequence of \eqref{eq:dirich-bd} and of the variational formula \eqref{def:varfor}: 
\begin{align*}
\E_\beta^\star\big[ \varphi, (-\S)^{-1} \varphi\big] & = \sup_{h \in \cC} \Big\{ 2 \E_\beta^\star\big[ \varphi, h \big] - \cD_{s_h}(\P_\beta^\star ; h)\Big\} \\
& \leq \sup_{h \in \cC}\Big\{ C(\varphi,\gamma) \big(\cD_{s_h}(\P_\beta^\star ; h)\big)^{1/2} - \cD_{s_h}(\P_\beta^\star ; h)\Big\} = \frac{C^2(\varphi,\gamma)}{4} < \infty.
\end{align*}}

Finally, if we use the decomposition of every function in $\mathbf L^2(\mu_\beta)$ over the basis of Hermite polynomials, 
we can prove the following result for functions in $\cQ_0$ 
(the details for the proof are given in Appendix \ref{sec:hermite}, Proposition \ref{prop:variational_restriction}.):

\bprop{[Variance of quadratic functions] \label{prop:variance} If $\varphi \in \cQ_0$, then  \[\E_\beta^\star\big[\varphi, (- \ccl{\S})^{-1} \varphi\big] = \sup_{\substack{g \in \cQ\\ s_g=s_\varphi}} \left\{ 2\, \E_\beta^\star\big[\varphi,g\big]-\cD_{s_\varphi}(\P_\beta^\star;g)\right\}.   \]}

\subsubsection{Semi inner products and diffusion coefficient \label{ssec:products}} 

For cylinder functions $g,h \in \cC$, let us define:

\begin{equation}\label{eq:prod1}\ll g,h\gg_{\beta,\star}:=\sum_{x \in \Z} \E^\star_\beta[g\, \tau_x h], 
\qquad \text{and} \qquad \ll g \gg_{\beta,\star\star}  : =\sum_{x \in \Z} x\, \E_\beta^\star\big[ g\,\omega^2_x\big].
\end{equation}
Both quantities are well defined because $g$ and $h$ belong to $\cC$ and therefore all but a finite number of terms on each sum vanish.

\begin{rem}\label{rem:prod} Note that $\ll \cdot, \cdot \gg_{\beta,\star}$ is a semi inner product, since the following equality holds: \begin{equation*}\ll g,h \gg_{\beta,\star}=\lim_{\Lambda \uparrow \Z} \frac{1}{\vert \Lambda \vert} \E_\beta^{\star}\bigg[\sum_{x \in \Lambda} \tau_x  g,  \sum_{x \in \Lambda} \tau_x h \bigg].\end{equation*}
Since $\ll g-\tau_x g ,h\gg_{\beta,\star}=0$ for all $x \in \Z$, this inner product is not definite. In particular we have $\ll j_{0,1}^S, h \gg_{\beta,\star}=0$ for any $h \in \cC$.
\end{rem}

\begin{rem}\label{rem:beta}\ccl{Observe that if $\omega$ is distributed according to $\mu_\beta$ then $\beta^{1/2}\omega$ is distributed according to $\mu_1$. Therefore, if $g,h \in \cQ$, then
\begin{equation}\label{eq:invbeta}
\ll g,h \gg_{\beta,\star} = \beta^2 \ll g,h \gg_{1,\star},\qquad  \text{and} \qquad \ll g \gg_{\beta,\star\star} = \beta^2 \ll g \gg_{1,\star}.
\end{equation}

}

\end{rem}

 In the next proposition we give explicit formulas for elements of $\cC_0$.

\bprop{ \label{prop:directcomputations}If $\varphi \in \cC_0$ with \begin{equation*} \varphi=\sum_{x \in \Lambda} \Big\{\nabla_x(F_x)+\nabla_{x,x+1}(G_x)\Big\},\end{equation*} then \begin{align*}
\ll \varphi \gg_{\beta,\star\star}&= \E_\beta^\star\bigg[( \omega_0^2-\omega_1^2) \sum_{x \in \Lambda} \tau_{-x} G_x \bigg],\\
\ll \varphi,g \gg_{\beta,\star}&=  \E_\beta^\star\bigg[ \nabla_0(\Gamma_g) \sum_{x\in\Lambda} \tau_{-x} F_x + \nabla_{0,1}(\Gamma_g) \sum_{x\in\Lambda} \tau_{-x} G_x\bigg] \qquad \text{ for all } g \in \cC.
\end{align*}}
\bprf{
\ccl{
The proof comes from straightforward computations using \eqref{eq:prod1}, we simply sketch it. Regarding the first identity, first note that by a change of variable $\omega\mapsto \omega^x$, for any cylinder function $F$, the quantity $\E_\beta^\star(\omega_y^2\nabla_x F )$ vanishes for any $y.$ A change of variable $\omega\mapsto \omega^{x,x+1}$ also yields \[\E_\beta^\star\Big[\omega_y^2\nabla_{x,x+1} F \Big]=\begin{cases}
-\E_\beta^\star\big[ (\omega_x^2-\omega_{x+1}^2) F\big] & \mbox{ if } y=x \vphantom{\Big(}\\
\E_\beta^\star\big[ (\omega_x^2-\omega_{x+1}^2) F\big]& \mbox{ if } y=x+1 \vphantom{\Big(}\\
0 & \mbox{ else},
\end{cases}\]
which, together with the translation invariance of $\P_\beta^\star$, proves the first identity.
The second identity is an immediate consequence of the symmetry of the gradients in $\mathbf{L}^2(\mu_\beta)$, \textit{i.e.} 
\[\E_\beta^\star\Big[ \Gamma_g \nabla_x F \Big]=\E_\beta^\star\Big[F\; \nabla_x (\Gamma_g)  \Big],
\qquad\E_\beta^\star\Big[ \Gamma_g \nabla_{x,x+1} G \Big]=\E_\beta^\star\Big[G \; \nabla_{x,x+1}( \Gamma_g ) \Big],\]
and of the translation invariance of $\P_\beta^\star$.
}
}

We are now able to give the definition of the diffusion coefficient, which is going to be rigorously derived from the non-gradient approach detailed in the next sections.

\bde{\label{def:diffusion}We define the {diffusion coefficient} $D(\beta)\equiv D$ for any $\beta >0$ as
\ccl{
\begin{equation}
\label{eq:GCDiff}
D:=\lambda+\frac{1}{\chi(\beta)}\inf_{f \in\Q}\sup_{g \in \Q}\left\{\ll f,-\S f \gg_{\beta,\star}+2\ll j_{0,1}^A-\A^{\mathbf m} f,g\gg_{\beta,\star}-\ll g , -\S g \gg_{\beta,\star}\right\},
\end{equation}
where $\chi(\beta)$ is the static compressibility defined in \eqref{eq:Defchi}. Note that according to   Remark \ref{rem:beta} the coefficient $D$ does not depend on $\beta$ (since by definition $\chi(\beta)=2\beta^{-2}$), which is why we omit $\beta $ in our notation.
}
}
The first term in the sum (\ccl{i.e. $\lambda$}) is only due to the exchange noise, whereas the second one comes from the hamiltonian part of the dynamics. Formally, this formula could be read as \begin{equation}D  \ccl{\text{ ``='' }}\lambda + \frac{1}{\chi(\beta)}\ll j_{0,1}^A, (-\cL^{\mathbf m})^{-1}j_{0,1}^A\gg_{\beta,\star}, \label{eq:gk1}\end{equation} but the last term is not well defined because $j_{0,1}^A$ is not in the range of $\cL^\m$.  More rigorously, we should define
\begin{equation}
\overline{D}(\beta):=\lambda + \frac{1}{\chi( \beta)}\limsup_{z \ccl{\downarrow} 0} \ll j_{0,1}^A, (z-\cL^\m)^{-1} j_{0,1}^A \gg_{\beta,\star}.
\label{eq:greenK} 
\end{equation} 
The last expression, called \emph{Green-Kubo formula}, is now well defined. \ccl{Moreover, once again Remark \ref{rem:beta} implies that \[ \bar{D}(\beta)=:\bar{D} \] does not depend on $\beta$:  in order to see it, let $\mathbf{L}^2_{\beta,\star}$ be  the Hilbert space generated by the closure of  the set  $\{g \in \cC \; ; \; \ll g,g \gg_{\beta,\star} < \infty\}$ 
 w.r.t.~the inner product $\ll \cdot \gg_{\beta,\star}$. Consider  $h_z:=h_z(\mathbf m,\omega;\beta)$  the solution to the resolvent equation in $\mathbf{L}^2_{\beta,\star}$ 
 which reads
 \begin{equation*} 
 (z-\cL^{\mathbf m})h_z=j_{0,1}^A. 
 \end{equation*}
Since $j_{0,1}^A$ is a homogeneous polynomial of degree two in $\omega$,  $h_z$ is also a homogeneous polynomial of degree two 
(by the fact that $\cL^\m$ preserves every class of homogeneous polynomials). It follows that 
\begin{equation*}
\ll j_{0,1}^A, (z-\cL^\m)^{-1} j_{0,1}^A \gg_{\beta,\star}\; =\ll h_z,j_{0,1}^A \gg_{\beta,\star}=\frac{1}{\beta^2}\ll h_z,j_{0,1}^A \gg_{1,\star}.
\end{equation*}  
One famous mathematical problem is to prove that, in the Green-Kubo formula \eqref{eq:greenK}, $\limsup$ can be replaced by $\lim$, and therefore to prove the convergence of the product $\ll \cdot \gg_{\beta,\star}$ in \eqref{eq:greenK}}  as $z\ccl{\downarrow} 0$. 
In \ccl{Propositions \ref{prop:conv_GK} and \ref{prop:equivalence}}, we prove that this indeed holds (the proof being inspired by  \cite{MR2448630}), 
and we also show that the diffusion coefficient can be equivalently defined in the two ways\ccl{, i.e. $ D=\bar D$}.

\section{Macroscopic fluctuations of energy}

\subsection{\ccl{Proof of  Theorem \ref{theo:fluctuations}}}
\label{ssec:martingale}
In what follows, in order to simplify notation we write $f(\m,s):=f(\m,\omega(s))$ for any $f$ which is defined on $\ccl{\Omega^{\cD}_N \times \Omega_N}$.
Let us fix  $H\in C^2(\T)$ and $t\in[0,T]$.  It\^o calculus, and a discrete integration by parts, permits to decompose $\Y_{t,\m}^N(H)$ as 
\begin{equation}
\label{eq:dec}
\Y_{t,\m}^N(H)=
\Y_{0,\m}^N(H)+\int_0^t \sqrt N \sum_{x \in \T_N} \nabla_NH\left(\frac x N \right) j_{x,x+1}(\m,s)\, \text ds+\mathcal M_{t,\m}^N(H)
\end{equation} 
where $\mathcal M_{t,\m}^N$ is the martingale defined as 
\begin{equation}\label{eq:defmart}\mathcal M_{t,\m}^N(H)=\int_0^t \frac{1}{N\sqrt N} \sum_{x \in \T_N} \nabla_NH\Big(\frac x N \Big) \big(\omega_{x}^2-\omega_{x+1}^2\big)(s) \, \text d\big[N_{x,x+1}(\lambda N^2 s)-\lambda N^2 s\big].\end{equation} 
Here and after,  $\{N_{x,x+1}(t)\}_{x\in\Z, t\geq 0}$ and  $  \{N_x(t)\}_{x\in\Z, t\geq 0}$ are independent Poisson processes of intensity 1, 
and $\nabla_N$ stands for the discrete gradient: 
\[\nabla_NH\Big(\frac{x}{N}\Big)=N\left[H\Big(\frac{x+1}{N}\Big)-H\Big(\frac{x}{N}\Big)\right].\] 
The discrete Laplacian $\Delta_N$ is defined in a similar way:  
\[\Delta_NH\Big(\frac{x}{N}\Big)=N^2\left[H\Big(\frac{x+1}{N}\Big)+H\Big(\frac{x-1}{N}\Big)-2H\Big(\frac{x}{N}\Big)\right].\]
To close the equation, we are going to replace the term involving the microscopic currents in \eqref{eq:dec} with a term involving $\Y_{t,\m}^N$. 
In other words, the dominant contribution in 
\[\int_0^t \sqrt N \sum_{x \in \T_N} \nabla_NH\left(\frac x N \right) j_{x,x+1}(\m,s)\, \text ds\] 
is its projection over the conservation field $\Y_{t,\m}^N$ (recall that the total energy is the unique conserved quantity of the system). 
The non-gradient approach consists in using the \emph{fluctuation-dissipation approximation} of the current $-j_{x,x+1}$  as 
$D\big(\omega_{x+1}^2-\omega_{x}^2\big)+\cL^\m(\tau_x f)$. This replacement is made rigorous in the \emph{Boltzmann-Gibbs principle} (Proposition \ref{prop:macro2}).

After adding and subtracting  $D\big(\omega_{x+1}^2-\omega_{x}^2\big)+\cL^\m(\tau_x f)$ in \eqref{eq:dec} above, we can rewrite it, for any $f \in  \Q$, as follows: 
\begin{equation}
\Y_{t,\m}^N(H)=\Y_{0,\m}^N(H)+\int_0^tD\Y^N_{s,\m}(\Delta_NH) \text ds +\mathfrak I^{1,N}_{t,\m,f}(H)+\mathfrak I^{2,N}_{t,\m,f}(H)+  \mathfrak M^{1,N}_{t,\m,f}(H)+\mathfrak M^{2,N}_{t,\m,f}(H), 
\label{eq:decomp} 
\end{equation} 
where
\begin{align}
\label{eq:DefI1}
\mathfrak I^{1,N}_{t,\m,f}(H)&=\int_0^t \sqrt N \sum_{x \in \T_N} \nabla_NH\left(\frac x N\right)\Big[j_{x,x+1}(\m,s)\nonumber\\
&\hspace{15em}+D\big(\omega_{x+1}^2-\omega_x^2\big)(s)+\cL^\m(\tau_x f)(\m,s)\Big]\text ds,\\
\label{eq:DefI2}
\mathfrak I^{2,N}_{t,\m,f}(H)&=-\int_0^t \sqrt N \sum_{x\in\T_N} \nabla_NH\left(\frac x N\right) \cL^\m(\tau_x f)(\m,s)\text ds,\\
\mathfrak M^{1,N}_{t,\m,f}(H)&=\int_0^t \frac{1}{N\sqrt N} \sum_{x\in\T_N} \nabla_NH\left(\frac x N\right)\Bigg\{ \nabla_{x}(\Gamma_f)(s)\, \text d\big[N_x(\gamma N^2 s)-\gamma N^2 s \big]\notag\\
&\label{eq:DefM1}
 \hspace{9em}-\big[\nabla_{x,x+1}(\omega_x^2-\Gamma_f) \big](s) \, \text d\big[N_{x,x+1}(\lambda N^2 s)-\lambda N^2 s\big]\Bigg\},\\
\mathfrak M^{2,N}_{t,\m,f}(H)&=-\int_0^t \frac{1}{N\sqrt N} \sum_{x\in\T_N} \nabla_NH\left(\frac x N\right) \Bigg\{\nabla_{x,x+1}(\Gamma_f)(s) \, \text d\big[N_{x,x+1}(\lambda N^2 s)-\lambda  N^2 s\big] \notag \\
&\label{eq:DefM2}
 \qquad \qquad \qquad +\nabla_{x}(\Gamma_f)(s)\, \text d\big[N_x(\gamma N^2 s)-\gamma N^2 s\big]\Bigg\}.
\end{align}
\ccl{The  following  two results are the main ingredients to prove Theorem \ref{theo:fluctuations}.}

\blem{\label{lem:macro1}For every function $H\in C^2(\T)$, and every function $f\in\Q$, \begin{equation}
\lim_{N\to\infty} \ccl{\ccl{\mathbb Q_{\mu_\beta^N}}}\bigg[\ccl{\sup_{0\leq t\leq T} }\ccl{\Big(}\mathfrak I^{2,N}_{t,\m,f}(H)+\mathfrak M^{2,N}_{t,\m,f}(H)\ccl{\Big)^2}\bigg] =0. \label{eq:macro1}
\end{equation}}

\bprop{[Boltzmann-Gibbs principle]\label{prop:macro2} There exists a sequence of functions $\{f_k\}_{k\in\N} \in \Q$ such that 
\begin{enumerate}[(i)] \item for every function $H\in C^2(\T)$, 
\begin{equation}
\lim_{k\to\infty} \lim_{N\to\infty} \ccl{\ccl{\mathbb Q_{\mu_\beta^N}}}\bigg[ \sup_{0\leqslant t \leqslant T} \Big(\mathfrak I^{1,N}_{t,\m,f_k}(H)\Big)^2\bigg]=0, \label{eq:macro2}
\end{equation} \item and moreover \begin{equation}
\lim_{k\to\infty} \E_{\beta}^\star\bigg[ \lambda\Big(\nabla_{0,1}(\omega_0^2-\Gamma_{f_k})\Big)^2 + \gamma \Big(\nabla_0(\Gamma_{f_k})\Big)^2\bigg]= 2D\chi(\beta). \label{eq:BG_principle}
\end{equation}
\end{enumerate}}

\begin{rem}
Note that the expectation at the left hand side of \eqref{eq:BG_principle} also rewrites as 
\[ 
2\lambda\chi(\beta) +  \E_\beta^\star \Big[ \lambda \big( \nabla_{0,1}(\Gamma_{f_k})\big)^2 +   \gamma \big(\nabla_0(\Gamma_{f_k}\big)^2 \Big] ,
\] since for any $f\in \mathcal{C}$, one can check that $\E_\beta^\star\left[(\omega_0^2-\omega_1^2)\nabla_{0,1}\Gamma_f\right]=0$.
\end{rem}

\ccl{\begin{rem}
In order to identify the limit point solving the martingale problems, we can use a weaker version of \eqref{eq:macro1} and \eqref{eq:macro2}, namely the same convergence without the supremum over $t$ would be enough.
However, this strong version will be necessary in order to prove tightness (Appendix \ref{sec:tightness}), which explains our choice to state the results in that way.
\end{rem}}
\ccl{We now put together the pieces to prove Theorem \ref{theo:fluctuations}. According to Proposition \ref{prop:tight}, the sequence $\{\mathfrak{Y}^N\}_N$ is tight. Consider one of its limit points $\mathfrak{Y}^*$.}
\ccl{As detailed in Appendix \ref{sec:tightness}, see Lemma \ref{lem:proofapp},} using analogous ingredients as in \ccl{the proof of} Lemma \ref{lem:macro1} and \ccl{Proposition \ref{prop:macro2}}, one can prove that, \ccl{up to extraction}, the martingale $\mathfrak M^{1,N}_{t,\m,f_k}$ converges in $\mathbf L^2(\P_\beta^\star)$, as $N\to\infty$ \ccl{and afterwards} $k\to\infty$, to
 a martingale $\mathfrak M_t(H)$ of quadratic variation \begin{equation*}2tD\chi(\beta)\int_\T H'(u)^2\, \text du.\end{equation*} 
In particular, using Lemma \ref{lem:macro1} and \eqref{eq:macro2}, one obtains that the limit $\Y_t(H)$ of $\Y_{t,\m}^N(H)$ satisfies the equation \begin{equation*}
\Y_t(H)=\Y_0(H)+\int_0^t\Y_s(DH'')\, \text ds + \mathfrak M_t(H). 
\end{equation*}
As a consequence, $\mathfrak{Y}^*$ is concentrated on trajectories $\Y$ solving the martingale problems \eqref{eq:mart1} and \eqref{eq:mart2}, which uniquely \ccl{characterizes  (from \cite[Theorem 0.2, p. 289]{MR1707314})} the generalized Ornstein-Uhlenbeck process $\Y_t$\ccl{, and therefore the proof of Theorem \ref{theo:fluctuations} is concluded}. \ccl{Lemma \ref{lem:macro1} is proved in Section \ref{sec:lemmacro1} below. 
 The proof of \ccl{Proposition \ref{prop:macro2}} is more challenging, it is proved in Sections \ref{sec:CLTvariances_equ}, \ref{sec:hilbert} and \ref{sec:diffusion}. More precisely,  \eqref{eq:macro2} is a consequence of Proposition \ref{prop:relation_variance} and Remark \ref{rem:DDt} (see Section \ref{subseci}), whereas \eqref{eq:BG_principle} is a consequence of Proposition \ref{prop:limit} (see Section \ref{subsecii}).
 }

\subsection{Proof of Lemma \ref{lem:macro1}}
\label{sec:lemmacro1}
In this \ccl{section} we give a proof of Lemma \ref{lem:macro1}. We define for any $f\in \Q$
\begin{equation*}
X^N_{t,\m,f}(H):=-\frac{1}{N\sqrt N} \sum_{x \in \T_N} \nabla_NH\left(\frac x N\right) \tau_x f(\m,t).
\end{equation*}
First, by rewriting \eqref{eq:dec} with  $X^N_{t,\m,f}(H)$ instead of $\Y_{t,\m}^N(H)$, one straightforwardly obtains
\begin{align}
&\mathfrak I^{2,N}_{t,\m,f}(H)+\mathfrak M^{2,N}_{t,\m,f}(H)=X^N_{t,\m,f}(H)-X^N_{0,\m,f}(H)\label{eq:xt}\\
&\quad +\frac{1}{N\sqrt N}\int_0^t \sum_{x\in\T_N} \nabla_{x,x+1}\bigg(\sum_{z\in\T_N} \nabla_NH\Big(\frac z N\Big) \tau_z f\nonumber \\
&\hspace{15em}- \nabla_NH\Big(\frac x N\Big) \Gamma_f\bigg)(\m,s) \text d\big[N_{x,x+1}(\lambda N^2 s)-\lambda N^2  s\big] \label{eq:firstmart}\\
&\quad +\frac{1}{N\sqrt N}\int_0^t \sum_{x\in\T_N} \nabla_{x}\bigg(\sum_{z\in\T_N} \nabla_NH\Big(\frac z N\Big) \tau_z f \nonumber\\
&\hspace{15em}- \nabla_NH\Big(\frac x N\Big) \Gamma_f\bigg)(\m,s) \text d\big[N_{x}(\gamma N^2  s)-\gamma N^2  s\big].\label{eq:decompDoob}\\
& =: X^N_{t,\m,f}(H)-X^N_{0,\m,f}(H)+M^{1,H,f}_{N,t}+M^{2,H,f}_{N,t}\nonumber \vphantom{\int_0^t}
\end{align}
\ccl{where $M^{1,H,f}_{N,t}$ (resp. $M^{2,H,f}_{N,t}$) is the  martingale term \eqref{eq:firstmart} (resp. \eqref{eq:decompDoob}). 
From the elementary inequality $(a+b+c)^2 \leq 3(a^2 + b^2 + c^2)$, we are left to bound the contributions of the squares of the three terms \eqref{eq:xt}, \eqref{eq:firstmart} and \eqref{eq:decompDoob}.

We first estimate the right hand side of \eqref{eq:xt}: remember that $f \in \cQ$. Therefore, using the definition \eqref{eq:quadratic}, and since the coefficients $\psi_{i,j}$ are continuous functions on $\Omega^\cD$, there exists a constant $C(f) >0$ such that 
\[
\sup_{\m \in \Omega^\cD}\big| f(\m, \omega) \big| \leqslant C(f) \sum_{i \in \Lambda_f} \omega_i^2
\]
where $\Lambda_f \subset \mathbb{Z}$ is some finite subset of $\mathbb{Z}$.  This implies,  for any $t\in [0, T]$}
\begin{equation*}
\ccl{\big|X_{t,\m,f}^N(H)\big| \leqslant \frac{1}{N\sqrt N} \sum_{x\in\T_N} \Big|\nabla_NH\left(\frac x N\right)\Big| \;  \big|\tau_x f(\m,\omega(t)) \big| \leqslant \frac{C(f,H)}{N\sqrt N}  \sum_{x \in \T_N} \sum_{i\in\Lambda_f} \omega_{x+i}^2(t),}
\end{equation*} 
\ccl{where the constant $C(H) >0$ comes from the fact that $H$ is of class $C^2$. By conservation of the total energy, we obtain the following bound: 
\[  \sup_{0\leq t \leq T}\big|X_{t,\m,f}^N(H)\big| \leqslant  \frac{C(f,H)|\Lambda_f|}{N\sqrt N}  \sum_{x\in\T_N} \omega_x^2 (0).\]
As a consequence, for some different constants $C(f,H)$ and $C(f,H,\beta)$, we obtain
\begin{equation}
\label{eq:boundsupXt}
\mathbb Q_{\mu_\beta^N}\left[\sup_{0\leq t\leq T}\big(X_{t,\m,f}^N(H) - X_{0,\m,f}^N(H)\big)^2\right]\leq C(f,H)\mathbb{E}_\beta^\star\bigg[ \bigg( \frac{1}{N\sqrt N} \sum_{x \in \T_N} \omega_x^2 \bigg)^2 \bigg] \leqslant \frac{C(f,H,\beta)}{N}.
\end{equation}
 On the other hand, introduce
\begin{equation*}
Y_x(\m,\omega):=\sum_{z \in \T_N}\nabla_NH\left(\frac z N\right)\tau_zf-\nabla_NH\left(\frac x N\right)\sum_{z\in\Z}\tau_z f,
\end{equation*} 
which is ill defined, but for which 
$\nabla_{x,x+1}Y_x(\m,\omega)$ and $\nabla_{x}Y_x(\m,\omega)$ only involves a finite number of non-zero contributions and are therefore well defined. 
For any square integrable martingale $M_t$, denote $[M]_t$ its quadratic variation. One easily computes
\[ \big[M^{1,H,f}_{N,\cdot}\big]_t=\frac{\lambda}{N}\int_0^t \sum_{x\in\T_N} \left(\nabla_{x,x+1}Y_x(\m,s)\right)^2 ds.\]
In particular, since $\E_\beta^\star((\nabla_{x,x+1}Y_x)^2)\leq C(f)/N$ (because $H$ is assumed to be of class $C^2$), we obtain 
\[\mathbb Q_{\mu_\beta^N}\Big[\big[M^{1,H,f}_{N,\cdot}\big]_T\Big]\leq \frac{C(f,H,\lambda,\beta)T}{N},\]
 for some other constant $C(f,H,\gamma,\beta)$, and a similar bound holds for $ [M^{2,H,f}_{N,\cdot}]_T$. As a consequence of Doob's martingale inequality, for any square integrable martingale $M_t$, one has $\E[\sup_{0\leq t \leq T}M_t^2]\leq 4 \E[[M]_T]$. Together with \eqref{eq:boundsupXt}, this yields
\begin{align}
\mathbb Q_{\mu_\beta^N}\bigg[\sup_{0\leq t \leq T} \ccl{\Big(}\mathfrak I^{2,N}_{t,\m,f}(H)+\mathfrak M^{2,N}_{t,\m,f}(H)\ccl{\Big)^2}\bigg] \leq \frac{C(f,H,\lambda,\gamma,\beta)}{N},
\end{align} 
which proves the result.
}

\section{{{Towards the Boltzmann-Gibbs principle}}\label{sec:CLTvariances_equ}}

In this section we are going to identify the diffusion coefficient $D$ that appears in \eqref{eq:decomp}. 
Roughly speaking, $D$ can be viewed as the asymptotic component of the energy current $j_{x,x+1}$ in the direction 
of the gradient $-(\omega_{x+1}^2-\omega_x^2)$, which makes the expression below vanish for any fixed $t\geq 0$ 
\[
\inf_{f \in \cQ}\limsup_{N\to\infty}\ccl{ \ccl{\mathbb Q_{\mu_\beta^N}}} \bigg[ \bigg|\int_0^t \sum_{x \in \T_N} \big[j_{x,x+1}+D(\omega_{x+1}^2-\omega_x^2)+\cL^\m(\tau_x f)\big] \text ds\bigg| \;\bigg], \quad \text{for any } \beta >0.
\] 
Let us start by giving some known tools that will help understand the forthcoming results, at least at an informal level.

\subsection{On additive functionals of Markov processes}
\label{sec:insight}
Consider a continuous time Markov process $\{Y_s\}_{s \geqslant 0}$ on a complete and separable metric space $E$, and admitting an invariant   measure  $\pi$. We denote by $\langle \cdot,\cdot \rangle_\pi$ the inner product in $\mathbf L^2(\pi)$ and by $\cL$  the infinitesimal generator of the process.  The adjoint of $\cL$ in $\mathbf L^2(\pi)$ is denoted by $\cL^\star$. Fix a function $V:E\to\R$ in $\mathbf L^2(\pi)$ such that $\ps{V}_\pi=0$. Theorem 2.7 in  \cite{MR2952852} gives conditions on $V$ which guarantee a central limit theorem for 
\[ \frac{1}{\sqrt{t}} \int_0^t V(Y_s) \text ds\] 
and shows that the limiting variance equals 

\[\sigma^2(V,\pi)=2\lim_{\substack{z \to 0\\ z>0}} \bps{V, (z-\cL)^{-1} V}_\pi.\]
Let the generator $\cL$ be decomposed as $\cL=\S+\A$, where $\S=(\cL+\cL^\star)/2$ and $\A=(\cL-\cL^\star)/2$ are respectively the symmetric and antisymmetric parts of $\cL$.  Let $\cH_1$ be the completion of the quotient of $\mathbf L^2(\pi)$ with respect to constant functions, for the semi-norm $\Vert \cdot \Vert_1$ defined as: 
\[\Vert f \Vert_{1}^2:=\bps{f,(-\cL)f}_\pi=\bps{f,(-\S)f}_\pi.\]
Let $\cH_{-1}$ be the dual space of $\cH_1$ with respect to $\mathbf L^2(\pi)$, in other words, the Hilbert space endowed with the norm $\Vert \cdot \Vert_{-1}$ defined by 
\[\Vert f\Vert^2_{-1}:=\sup_g \left\{ 2 \bps{f,g}_\pi-\Vert g \Vert_1^2\right\},\] where the supremum is carried over some suitable set of  functions $g$. Formally, $\Vert f \Vert_{-1}$ can also be thought as 
\[\bps{f,(-\S)^{-1} f}_{\pi},\]
\ccl{and this identity actually holds rigorously as soon as $f=\S g$ is in the range of $\S$}. Note the difference with the variance $\sigma^2(V,\pi)$ which formally reads 
\[ 2\bps{V,(-\cL)^{-1} V}_\pi=2\left\langle V, \big[(-\cL)^{-1}\big]_s V\right\rangle_\pi.\] 
Hereafter, $B_s$ represents the symmetric part of the operator $B$. We can write, at least formally, that 
\[ \left\{ \big[ (-\cL)^{-1} \big]_s \right\}^{-1}=-\S + \A^\star (-\S)^{-1} \A \geqslant -\S, \] 
where $\A^\star$ stands for the adjoint of $\A$. We have therefore that $\big[ (-\cL)^{-1} \big]_s  \leqslant (-\S)^{-1}$. The following result is a rigorous estimate of the variance in terms of the $\cH_{-1}$ norm, which is proved in \cite[Lemma 2.4]{MR2952852}.

\blem{
\label{lem:varianceKLO}
Given $T>0$ and a mean zero function $V$ in $\mathbf L^2(\pi) \cap \cH_{-1}$, 
\begin{equation}
\E_\pi\left[\sup_{0\leqslant t \leqslant T} \left(\int_0^t V(s) {\rm d}s\right)^2\right] \leqslant 24 T \Vert V \Vert_{-1}^2.
\label{eq:varianceKLO}
\end{equation}
}
In our case, the fact that the symmetric part of the generator does not depend on the disorder  implies that \eqref{eq:varianceKLO} still holds if we take the expectation with respect to the disorder $\P$, and thus replace $\pi$ with $\P\otimes \pi$. If we compare the previous left hand side to the Boltzmann-Gibbs principle \eqref{eq:BG_principle}, the next step should be to take $V$ proportional to 
\begin{equation}
\label{eq:formal} 
\sum_{x \in \T_N} \big[  j_{x,x+1}+D(\omega_{x+1}^2-\omega_x^2)+\cL^\m(\tau_x f)\big]
\end{equation} 
and then take the limit as $N$ goes to infinity. In the right hand side of \eqref{eq:varianceKLO} we will obtain a variance that depends on $N$, and the main task will be to show that this variance converges: this is studied in more details in what follows. Precisely, we prove that the limit of the variance results in a semi-norm, which is denoted by $\vertiii{ \cdot }_\beta$ and defined in \eqref{eq:norm1} below. More explicitly, we are going to see that \eqref{eq:norm1} involves a variational formula, which formally reads 
\[ 
\vertiii{ \varphi }_\beta^2=\ll \varphi ,(-\S)^{-1} \varphi \gg_{\beta,\star} + \frac{1}{\lambda\, \chi(\beta)}\; \ll \varphi \gg_{\beta,\star\star}^2.
\]
The final step consists in minimizing this semi-norm on a well-chosen subspace in order to get the Boltzmann-Gibbs principle, through orthogonal projections in Hilbert spaces. One significant difficulty is that $\vertiii{ \cdot }_\beta$ only depends on the symmetric part of the generator $\S$, and the latter is really degenerate, since it does not have a spectral gap.

 In Subsection \ref{ssec:seminorm}, we relate the previous limiting variance (which is obtained by taking the limit as $N$ goes to infinity) to the suitable semi-norm. Subsection \ref{ssec:macro2} is devoted to proving the Boltzmann-Gibbs principle (using Lemma \ref{lem:varianceKLO}). Note that \eqref{eq:formal} is a sum of local functions in $\cQ_0$, from Proposition \ref{prop:propertiesC0} (recall that $f \in \cQ$). Therefore, all our results will be restricted to that subspace. Then, in Section \ref{sec:hilbert} we investigate the Hilbert space generated by the semi-norm, and prove  decompositions into direct sums. Finally, Section \ref{sec:diffusion} focuses on the diffusion coefficient and its different expressions. 
These three main steps are quite standard, and many of the arguments can be found in \cite{sasolla}. For that reason, we shall be more brief in the exposition, and refer the reader to \cite{sasolla} for more details.

\subsection{Limiting variance and semi-norm\label{ssec:seminorm}}

We now assume $\beta=1$. All statements are valid for any $\beta >0$, and the general argument can be easily  written. In the following, we deliberately keep the notation $\chi(1)$, even if the latter could be replaced with its exact value $\chi(1)=2$. We are going to obtain a variational formula for the \ccl{limiting variance}  
\begin{equation*}
\frac{1}{2\ell}\E_1^\star\bigg[ \left(-\S_{\Lambda_\ell}\right)^{-1} \sum_{\vert x \vert \leqslant \ell_\varphi}\tau_x \varphi, \sum_{\vert x \vert \leqslant \ell_\varphi}\tau_x \varphi \bigg]
\end{equation*}
\ccl{as $\ell\to\infty$}, where $\varphi \in \Q_0$ and $\ell_\varphi=\ell-s_\varphi-1$. \ccl{Note that the quantity above is well defined, since $\sum_{\vert x \vert \leqslant \ell_\varphi}\tau_x \varphi$ has mean $0$ and is therefore in the range of $\S_{\Lambda_\ell}$ (see e.g. \cite[p. 152]{MR1707314}).} We first introduce a semi-norm on $\Q_0$:

\begin{de} \label{de:norm1} For any cylinder function $\varphi$ in $\Q_0$, let us define
\begin{align} &\vertiii{ \varphi }^2_1\!=\sup_{g \in \Q} \left\{ 2 \ll \varphi,g \gg_{1,\star}\!-\frac{\gamma}{2}\E_1^\star\left[\left(\nabla_0\Gamma_g\right)^2\right]\! -\frac{\lambda}{2}\E_1^\star\left[\left(\nabla_{0,1}\Gamma_g\right)^2\right] \right\} \! +\! \frac{1}{\lambda  \chi(1)}\!\ll \varphi \gg^2_{1,\star\star} \label{eq:norm1} \\
~ \notag \\
&=\sup_{\substack{g \in \Q\\ a\in\R}}\Big\{2 \ll \varphi,g \gg_{1,\star}+2a\ll \varphi \gg_{1,\star\star} -\frac{\gamma}{2}\E_1^\star\left[\left(\nabla_0\Gamma_g\right)^2\right] -\frac{\lambda}{2}\E_1^\star\left[\left(a(\omega_0^2-\omega_1^2)+\nabla_{0,1}\Gamma_g\right)^2\right]     \Big\}, \label{eq:norm1bis}\end{align}
where $\ll \cdot \gg_{1,\star}$ and $\ll \cdot \gg_{1,\star\star}$ were introduced in \eqref{eq:prod1}.
\end{de}
\begin{rem}  The second identity in \eqref{eq:norm1bis} follows from an explicit computation of the supremum in $a \in \mathbb{R}$, which can be obtained by standard arguments, using the fact that  $\E_1^\star\left[(\omega_0^2-\omega_1^2)\nabla_{0,1}\Gamma_g\right]=0$ for any $g\in \mathcal{C}$.
\end{rem}
Note that, from Proposition \ref{prop:directcomputations}, one can easily bound $\vertiii{\varphi}^2_1$ for any $\varphi \in \cQ_0$ as follows: if \begin{equation*} \varphi=\sum_{x \in \Lambda} \Big\{\nabla_x(F_x)+\nabla_{x,x+1}(G_x)\Big\},\end{equation*}
then 
\[
\vertiii{ \varphi }^2_1 \leqslant \frac{2}{\gamma} \E_1^\star\bigg[ \Big(\sum_{x \in \Lambda} \tau_{-x}F_x\Big)^2\bigg]+\frac{3}{\lambda} \E_1^\star\bigg[ \Big(\sum_{x \in \Lambda} \tau_{-x}G_x\Big)^2\bigg] < \infty.
\]
We are now in position to state the main result of this subsection. 

\bprop{\label{prop:variance_equ} Consider a quadratic function $\varphi\in\Q_0$. Then 
\begin{equation}
\label{eq:limH-1}
{\lim_{\ell \to \infty} (2\ell)^{-1} \E_1^\star \bigg[ \big(-\S_{\Lambda_\ell}\big)^{-1}\sum_{\vert x \vert \leqslant \ell_\varphi}\tau_x\varphi,\sum_{\vert x \vert \leqslant \ell_\varphi}\tau_x \varphi\bigg]=\ \vertiii{ \varphi }^2_{1}.} 
\end{equation} 
Here, $\ell_\varphi$ stands for $\ell-s_\varphi-1$ so that the support of $\tau_x \varphi$ is included in $\Lambda_\ell$ for every $x \in \Lambda_{\ell_\varphi}$.}

This result is the key ingredient of the standard non-gradient  method. As usual, the proof is done in two steps that we separate as two different lemmas for the sake of clarity.  First, we  bound the variance of a cylinder function $\varphi \in \Q_0$, with respect to $\P^\star_1$, by the semi-norm $\vvvert \varphi \vvvert_1^2$ (\ccl{see Lemma \ref{lem:firststep}}). In the second step, 
a lower bound for the variance can be easily deduced from the variational formula which expresses the variance as a supremum \eqref{def:varfor} \ccl{(see Lemma \ref{lem:secondstep})}. \ccl{The rest of this section is devoted to proving these two bounds.}

\blem{\label{lem:firststep}Under the assumptions of \ccl{Proposition \ref{prop:variance_equ}}, \begin{equation*}{\limsup_{\ell \to \infty} (2\ell)^{-1}\E_1^\star\bigg[ \big(-\S_{\Lambda_\ell}\big)^{-1}\sum_{\vert x \vert \leqslant \ell_\varphi}\tau_x\varphi,\sum_{\vert x \vert \leqslant \ell_\varphi}\tau_x \varphi   \bigg] \leqslant\, \vertiii{ \varphi }^2_1 .} \end{equation*}}
\begin{proof}[\ccl{Proof of Lemma \ref{lem:firststep}}]
\ccl{To prove the lemma, we study the} weak limits of some particular sequences in $\cQ_0$. In the typical approach, 
these weak limits are viewed as \emph{germs of closed forms}, 
but for the harmonic chain, this is not necessary: this is one of the main technical novelties in this work.

Let us start by following the proof given in \cite[Lemma 4.3]{sasolla}\ccl{. We} assume for the sake of clarity that $\varphi=\nabla_0(F)+\nabla_{0,1}(G)$, for two quadratic cylinder functions $F, G$ (the general case can then be deduced quite easily). We write the variational formula 
\begin{multline}
\label{eq:normphi1}
(2\ell)^{-1} \E_1^\star\bigg[ \big(-\S_{\Lambda_\ell}\big)^{-1}\sum_{\vert x \vert \leqslant \ell_\varphi}\tau_x\varphi,\sum_{\vert x \vert \leqslant \ell_\varphi}\tau_x \varphi   \bigg]  = \sup_{h \in \cC} \bigg\{  2\E_1^\star\bigg[ \varphi, \frac{1}{2\ell} \sum_{\vert x \vert \leqslant \ell_\varphi} \tau_x h\bigg] - \frac{1}{2\ell} \cD_\ell(\P^\star_1;h)\bigg\} \\
= \sup_{h \in \cC}\bigg\{  2\E_1^\star\bigg[ F\nabla_0\bigg( \frac{1}{2\ell}\sum_{\vert x \vert \leqslant \ell_\varphi} \tau_x h\bigg)  + G \nabla_{0,1}\bigg(\frac{1}{2\ell} \sum_{\vert x \vert \leqslant \ell_\varphi} \tau_x h\bigg)\bigg] - \frac{1}{2\ell} \cD_\ell(\P^\star_1 ; h) \bigg\}.
\end{multline}
Since $\varphi$ is quadratic, we can restrict the supremum in the class of quadratic functions $h$ with support contained in $\Lambda_\ell$ (\ccl{see Proposition} \ref{prop:variational_restriction}). We can also restrict the supremum to functions $h$ such that $\cD_\ell(\P^\star_1;h)\leqslant C\ell$, as a standard consequence of Proposition \ref{lem:ibp} (namely,  there is some constant $C(\varphi)$ such that the right hand side \ccl{of \eqref{eq:normphi1}} is non-positive when $\cD_\ell(\P^\star_1;h) > C(\varphi)\ell$).
Next, we want to replace the sums over $\Lambda_{\ell_\varphi}$ with the same sums over $\Lambda_\ell$ (recall that $\ell_\varphi = \ell-s_\varphi-1\leqslant \ell$).  For that purpose, we denote 
\begin{equation}
\zeta_0^\ell(h)\ccl{:=}\nabla_0\bigg( \frac{1}{2\ell}\sum_{x= -\ell}^\ell \tau_x h\bigg), \qquad \zeta_1^\ell(h)\ccl{:=}\nabla_{0,1}\bigg(\frac{1}{2\ell} \sum_{x=-\ell+1}^\ell \tau_x h\bigg). \label{eq:def_zeta}
\end{equation}
First of all, from the Cauchy-Schwarz inequality, we have \begin{equation*}
\E_1^\star\bigg[ \frac\gamma 2 \Big(\zeta^\ell_0(h)\Big)^2+\frac \lambda 2\Big(\zeta_1^\ell(h)\Big)^2\bigg] \leqslant \frac{1}{2\ell}\cD_\ell(\P^\star_1;h).
\end{equation*}
Then, from elementary computations (similar to the proof of Proposition \ref{lem:ibp}), we can write \begin{equation*}
\bigg\vert \E_1^\star\bigg[ \varphi,\frac{1}{2\ell} \sum_{\ell_\varphi\leqslant x \leqslant \ell}\tau_x h\bigg] \bigg\vert \leqslant \frac{1}{2\ell} C(\varphi,\gamma) \big( \cD_\ell(\P^\star_1;h) \big)^{1/2},
\end{equation*}
where $C(\varphi,\gamma)$ is a constant which depends only on $\varphi$ and $\gamma$. These last two inequalities give the upper bound 
\begin{multline}
 (2\ell)^{-1} \E_1^\star\bigg[ \big(-\S_{\Lambda_\ell}\big)^{-1}\sum_{\vert x \vert \leqslant \ell_\varphi}\tau_x\varphi,\sum_{\vert x \vert \leqslant \ell_\varphi}\tau_x \varphi \bigg]   \\
 \leqslant \sup_{\substack{h\in\cQ\\ \cD_\ell(\P_1^\star;h) \leqslant C\ell}} \left\{  2\E_1^\star\Big[ F \,\zeta_0^\ell(h) + G\, \zeta_1^\ell(h) \Big] - \E_1^\star\bigg[ \frac \gamma 2\Big(\zeta^\ell_0(h)\Big)^2+\frac \lambda 2\Big(\zeta_1^\ell(h)\Big)^2\bigg] \right\} + {\frac{C}{\sqrt \ell}}, \label{eq:bb1}
\end{multline}
from some constant $C>0$. From now on, we denote generically by $C$ a positive constant that does not depend on $\ell$, but may depend on $\varphi$ (and $\gamma$), and may change from line to line.
%
The conclusion is now based on the following lemma: 
\blem{\label{lem:explicitbis} Assume that $h \in \cQ$ with support in $\Lambda_\ell$. From Definition \ref{de:quadratic}, it   reads as
\[h(\m,\omega) = \sum_{\substack{i,j=-\ell \\ i\neq j}}^\ell \psi_{i,j}(\m) \omega_i\omega_j + \sum_{i=-\ell}^{\ell-1} \psi_{i,i}(\m) (\omega_{i+1}^2-\omega_i^2). \]
Then there exists $a_\ell(\m)$ and $\mathcal{R}_\ell(\m,\omega)$ such that, 
\begin{align}
\zeta_0^\ell(h)& = \nabla_0\big( \Gamma_{h/(2\ell)}\big) \label{eq:id11} \\
\zeta_1^\ell(h)& = \nabla_{0,1} \big( \Gamma_{h/(2\ell)}\big) + a_\ell(\m) (\omega_0^2-\omega_1^2) + \mathcal{R}_\ell(\m,\omega). \label{eq:id22}
\end{align}
Moreover, if $\cD_\ell(\P_1^\star ; h) \leqslant C  \ell$, then
\begin{equation}\label{eq:estimdiri} \E_1^\star \Big[ \big(\mathcal{R}_\ell(\m,\omega)\big)^2\Big] \leqslant \frac{C}{\gamma \ell}.\end{equation}
} 
\begin{proof}[Proof of Lemma \ref{lem:explicitbis}]
The proof of this lemma is rather straightforward, we merely sketch it. First,  given the shape of the function $h$, elementary computations yield that \eqref{eq:id11} and \eqref{eq:id22} hold with
\begin{align*}
a_\ell(\m) & =  \frac{1}{2\ell} \big(\psi_{\ell-1,\ell-1}(\tau_{-\ell}\m) + \psi_{-\ell,-\ell}(\tau_{\ell+1}\m)\big), \\
\mathcal{R}_\ell(\m,\omega) & = \frac{1}{\ell}  \bigg( -\sum_{j=-\ell}^{\ell-1} \psi_{\ell,j}(\tau_{-\ell}\m)\; \omega_{j-\ell} + \sum_{j=-\ell+1}^\ell \psi_{-\ell,j}(\tau_{\ell+1}\m) \; \omega_{j+\ell+1} \bigg)(\omega_1-\omega_0).
\end{align*} 
Then, we straightforwardly obtain, by translation invariance of $\P_1^\star$,
\begin{equation*}
\E_1^\star \Big[ \big(\mathcal{R}_\ell(\m,\omega)\big)^2\Big]\leq \frac{C_1}{\ell^2} \bigg(\E_1^\star \bigg[ \sum_{j=-\ell}^{\ell-1}\big(\psi_{\ell,j}(\m)\big)^2\bigg]+\E_1^\star \bigg[ \sum_{j=-\ell+1}^{\ell}\big(\psi_{-\ell,j}(\m)\big)^2\bigg] \bigg),
\end{equation*}
where $C_1=4\E_1^\star\big[\omega_{-1}^2(\omega_1-\omega_0)^2\big]=8 \E_1^\star\big[\omega_{-1}^2\omega_0^2\big]=8$. Furthermore, since both parts of the Dirichlet form given in \eqref{eq:dirich} are non-negative, in particular, we have
\begin{align*}
\cD_\ell(\P_1^\star ; h)\geq& \; \frac{\gamma}{2}\sum_{x\in\Lambda_\ell}\E_1^\star\Big[\big(h(\m, \omega^x)-h(\m, \omega)\big)^2\Big]\\
\geq&\;\frac{\gamma}{2}\E_1^\star\Big[\big(h(\m, \omega^\ell)-h(\m, \omega)\big)^2\Big]+  \frac{\gamma}{2}\E_1^\star\Big[\big(h(\m, \omega^{-\ell})-h(\m, \omega)\big)^2\Big]\\
=& \;\frac{\gamma}{2}\E_1^\star \bigg[ 16 \omega_\ell^2 \; \bigg(\sum_{j=-\ell}^{\ell-1}\psi_{\ell,j}(\m)\omega_j\bigg)^2\bigg]+\frac{\gamma}{2}\E_1^\star \bigg[16\omega_{-\ell}^2 \bigg(\sum_{j=-\ell+1}^{\ell}\psi_{-\ell,j}(\m)\omega_j\bigg)^2\bigg]\\
= &\;8\gamma \E_1^\star \bigg[ \sum_{j=-\ell}^{\ell-1}\big(\psi_{\ell,j}(\m)\big)^2\bigg]+8 \gamma \E_1^\star \bigg[ \sum_{j=-\ell+1}^{\ell}\big(\psi_{-\ell,j}(\m)\big)^2\bigg] .
\end{align*}
The previous two  bounds finally yield
\[\E_1^\star \Big[ \big(\mathcal{R}_\ell(\m,\omega)\big)^2\Big]\leq\frac{1}{\gamma \ell^2} \cD_\ell(\P_1^\star ; h),\]
which proves \eqref{eq:estimdiri}. \end{proof}

Lemma \ref{lem:explicitbis} above permits to bound  the limit as $\ell\to\infty$ of \eqref{eq:bb1} by 
\begin{multline}
\sup_{\substack{f \in \cQ \\ a:\ccl{\Omega^{\cD}}\to\R }} \bigg\{  2\E_1^\star\Big[ F \,\nabla_0\Gamma_f + G \big(a(\m)(\omega_0^2-\omega_1^2) + \nabla_{0,1}\Gamma_f\big) \Big] \\
- \E_1^\star\bigg[ \frac \gamma 2\big(\nabla_0\Gamma_f\big)^2+\frac \lambda 2\big(a(\m)(\omega_0^2-\omega_1^2)+\nabla_{0,1}\Gamma_f\big)^2\bigg] \bigg\}=:\sup_{\substack{f \in \cQ \\ a:\ccl{\Omega^{\cD}}\to\R }} \mathfrak{H}(\varphi, a, f),\label{eq:bb2} 
\end{multline}
where we denote  by $\mathfrak{H}(\varphi, a, f)$ the quantity inside brackets.
To conclude we want to restrict the supremum on \emph{real numbers} $a$ which do not depend on the disorder. This is done in a similar way as in \cite[Lemma 7.7]{MR2021195}. To that aim, for any positive $\varepsilon$, fix $a_\varepsilon(\m)$ such that
\begin{equation}
\label{eq:aepsilon}
\sup_{f \in \cQ } \mathfrak{H}(\varphi, a_\varepsilon, f)\geq \sup_{\substack{f \in \cQ \\ a:\ccl{\Omega^{\cD}}\to\R }} \mathfrak{H}(\varphi, a, f)-\varepsilon,
\end{equation}
and shorten $\tilde a_\varepsilon(\m):=a_\varepsilon(\m)-\E[a_\varepsilon]$. Let us define, for any $x \in \Z$, the function $b_x \in \mathbf{L}^2(\P)$ given by 
\[ b_x(\m) = \begin{cases} \displaystyle
\sum_{k=0}^{x-1} \tau_k \tilde a_\varepsilon(\m) & \text{ if } x \geqslant 1, \\
\displaystyle \sum_{k=x}^{-1} - \tau_k \tilde a_\varepsilon(\m) & \text{ if } x \leqslant -1.
\end{cases}  \qquad \text{ and } \quad b_0(\m)=0,\]
which is defined in such a way that for any $x \in \Z$, $b_{x+1}(\m)-b_x(\m) = \tau_x \tilde a_\varepsilon(\m).$ 
For any $n\in\N$, let  us introduce the quadratic function 
\[g_n(\m,\omega) = - \sum_{x \in \Lambda_n} b_x(\m) \omega_x^2.\]
One can easily check that, for any $z \in \Z$ such that $\{z,z+1\}\subset \Lambda_n$, 
\[\nabla_{z,z+1}(g_n) = \tau_z \tilde a_\varepsilon(\m)(\omega_{z+1}^2-\omega_z^2), \qquad \text{and} \qquad \nabla_0(g_n) = 0.\]
Therefore, letting $\tilde f_n:=f+ g_n/(2n)$ which still belongs to $\cQ$, we get
\begin{align*}
&\nabla_0\Gamma_f  = \nabla_0\Gamma_{\tilde f_n} \\
&\nabla_{0,1}\Gamma_f + a_\varepsilon (\m)(\omega_0^2-\omega_1^2)  = \nabla_{0,1}\Gamma_{\tilde f_n} + \E[a_\varepsilon](\omega_0^2-\omega_1^2) +\mathfrak{R}_n(\m,\omega), \end{align*}where 
\[\mathfrak{R}_n(\m,\omega) =  \tilde a_\varepsilon(\m)(\omega_0^2-\omega_1^2) - \frac{1}{2n} \nabla_{0,1}\Gamma_{g_n}.
\]
We are now going to estimate the $\mathbf{L}^2(\P_1^\star)$-norm of $\mathfrak{R}_n$, as follows: basic computations show that 
\begin{align*} 
\frac{1}{2n} \nabla_{0,1}\Gamma_{g_n} = \tilde a_\varepsilon(\m)(\omega_0^2-\omega_1^2) & + \frac{1}{2n} \tau_{-n}\Big(b_n(\m)(\omega_n^2-\omega_{n+1}^2)\Big)  \\ & + \frac{1}{2n} \tau_{n+1}\Big(b_{-n-1}(\m)(\omega_{-n-1}^2 - \omega_{-n}^2)\Big).
\end{align*}
Hence, from the Cauchy-Schwarz inequality and translation invariance of $\P$, it is enough to show that 
\begin{equation}
\label{eq:boundbn}
 \frac{\E\big[b_n^2\big]}{n^2}=\frac{1}{n^2}\E\bigg[\bigg(\sum_{k=0}^{n-1}\tau_k \tilde a_\varepsilon(\m)\bigg)^2\bigg] \xrightarrow[n\to\infty]{} 0, \qquad \text{and} \qquad \frac{\E\big[b_{-n-1}^2\big]}{n^2}\xrightarrow[n\to\infty]{} 0.  
 \end{equation}
These convergences are standard consequences of the translation invariance of $\P$: more precisely, let us fix a positive integer $p$ and introduce for any $x \in \Z$ the conditional expectation 
 \[
 \widetilde a_x^{(\varepsilon, p)} = \E\Big[ \tau_x \tilde a_\varepsilon(\m) \; \big| \; m_y \; ; \; {y \in \Lambda_p(x)} \Big]. 
 \]
 From our assumptions, note that $\widetilde a_x^{(\varepsilon, p)} = \tau_x \widetilde a_0^{(\varepsilon,p)}$ and $\E\big[ \widetilde a_x^{(\varepsilon,p)} \big] =0$. As a result,
 \begin{align*}
 \frac{1}{n^2} \E\bigg[\bigg(\sum_{k=0}^{n-1}\tau_k \tilde a_\varepsilon(\m)\bigg)^2\bigg] &\leqslant  \frac{2}{n^2} \E\bigg[\bigg(\sum_{k=0}^{n-1}\Big\{\tau_k \tilde a_\varepsilon(\m)-\widetilde a_k^{(\varepsilon,p)}  \Big\}\bigg)^2\bigg]+ \frac{2}{n^2} \E\bigg[\bigg(\sum_{k=0}^{n-1}\widetilde a_k^{(\varepsilon,p)} \bigg)^2\bigg]\\
 & \leqslant 2 \E \bigg[ \Big\{\tilde  a_\varepsilon(\m)-\widetilde a_0^{(\varepsilon,p)}  \Big\}^2\bigg] + \frac{C(\varepsilon,p)}{n}.
 \end{align*}
The last inequality comes from the fact that $\sum \widetilde a_k^{(\varepsilon, p)}$ is  a sum of identically distributed variables (because of the translation invariance of $\P$), for which we have a good control of the variance. Letting now, in the bound above, $n\to\infty$, and then $p\to\infty$, we obtain that \eqref{eq:boundbn} holds, thus \eqref{eq:aepsilon} rewrites
\begin{equation*}
\sup_{f \in \cQ } \mathfrak{H}(\varphi, \E[a_\varepsilon], f)\geq \sup_{\substack{f \in \cQ \\ a:\ccl{\Omega^{\cD}}\to\R }} \mathfrak{H}(\varphi, a, f)-\varepsilon.
\end{equation*}
Since this holds for any $\varepsilon>0$, we finally obtain as wanted that \begin{equation*}
\sup_{\substack{f \in \cQ \\ a\in\R }} \mathfrak{H}(\varphi, a, f)= \sup_{\substack{f \in \cQ \\ a:\ccl{\Omega^{\cD}}\to\R }} \mathfrak{H}(\varphi, a, f),
\end{equation*}
and therefore
\begin{multline*}
 (2\ell)^{-1} \E_1^\star\bigg[ \big(-\S_{\Lambda_\ell}\big)^{-1}\sum_{\vert x \vert \leqslant \ell_\varphi}\tau_x\varphi,\sum_{\vert x \vert \leqslant \ell_\varphi}\tau_x \varphi   \bigg]  \\
  \le  \sup_{\substack{g \in \cQ \\a \in \R}}  \bigg\{  2\E_1^\star\Big[ F \,\nabla_0\Gamma_g + G\, \big(a(\omega_0^2-\omega_1^2) + \nabla_{0,1}\Gamma_g\big) \Big]  \\
  -\frac \gamma 2 \E_1^\star\Big[ (\nabla_0\Gamma_g)^2\Big]-\frac\lambda 2\E_1^\star\Big[\big(a(\omega_0^2-\omega_1^2) + \nabla_{0,1}\Gamma_g\big)^2\Big] \bigg\}.
\end{multline*}
Lemma \ref{lem:firststep} follows, after recalling \eqref{eq:norm1bis}.
\end{proof}

\medskip

We now turn to the upper bound. 

\blem{
\label{lem:secondstep}
Under the assumptions of \ccl{Proposition \ref{prop:variance_equ}}, 
\[
\limsup_{\ell \to \infty}(2\ell)^{-1} \E_1^\star\bigg[ \big(-\S_{\Lambda_\ell}\big)^{-1}
\sum_{\vert x \vert \leqslant \ell_\varphi}\tau_x\varphi,\sum_{\vert x \vert \leqslant \ell_\varphi}\tau_x \varphi   \bigg]  
\geqslant\, \vertiii{ \varphi }^2_1 . 
\] }

\bprf{[\ccl{Proof of Lemma \ref{lem:secondstep}}]\ccl{F}or $f \in \Q$, define $\ell_f=\ell-s_f-1$ and
\begin{equation*}
J_\ell:=\sum_{y,y+1 \in\Lambda_\ell} \tau_y j_{0,1}^S, \qquad H_\ell^f:=  \sum_{\vert y \vert \leqslant \ell_f} \S (\tau_y f). 
\end{equation*}
The following limits hold:
\begin{align} 
\lim_{\ell \to \infty} (2\ell)^{-1}\E_1^\star\bigg[\big(-\S_{\Lambda_\ell}\big)^{-1} \sum_{\vert x \vert \leqslant \ell_\varphi}\tau_x \varphi,  \;  J_\ell \bigg]
& = \quad -\ll \varphi \gg_{1,\star\star}\ , \label{eq:limit1} \\ 
\lim_{\ell \to \infty} (2\ell)^{-1}\E_1^\star\bigg[\big(-\S_{\Lambda_\ell}\big)^{-1}\sum_{\vert x \vert \leqslant \ell_\varphi}\tau_x\varphi, \;  H_\ell^f     \bigg]
&=\quad - \ll \varphi,f \gg_{1, \star}\ , \notag\\
\lim_{\ell \to\infty} (2\ell)^{-1} \E_1^\star\bigg[\big(-\S_{\Lambda_\ell}\big)^{-1}\Big(aJ_\ell+H_\ell^f \Big),  \Big(aJ_\ell +H_\ell^f\Big) \bigg] 
&= \notag\\
 \frac{\lambda}{2}\E_1^\star\Big[\big(a(\omega_0^2-\omega_1^2)&+\nabla_{0,1}\Gamma_f\big)^2\Big]+\frac{\gamma}{2}\E_1^\star\Big[\big(\nabla_0\Gamma_f\big)^2\Big].  \notag
\end{align}
We only prove \eqref{eq:limit1}, the other relations can be obtained in a similar way. 
As previously, we assume for the sake of simplicity that $\varphi=\nabla_0(F) +\nabla_{0,1}(G)$.
One can easily check the elementary identity 
\begin{equation}
\S_{\Lambda_\ell}\bigg(\sum_{x \in \Lambda_\ell} x \omega_x^2\bigg)=J_\ell(\omega).
\label{eq:identity1} 
\end{equation} 
Therefore,
\begin{align*}
 (2\ell)^{-1}\E_1^\star\bigg[\big(-\S_{\Lambda_\ell}\big)^{-1}&\sum_{\vert x \vert \leqslant \ell_\varphi}\tau_x \varphi,   J_\ell \bigg] \\
& = -(2\ell)^{-1} \sum_{y\in\Lambda_\ell} \sum_{\vert x \vert \leqslant \ell_\varphi} y \ \E_1^\star\big[ \varphi \, \omega_{y-x}^2\big] \\
 &= -(2\ell)^{-1}  \sum_{y\in\Lambda_\ell} \sum_{\vert x \vert \leqslant \ell_\varphi} y\ \E_1^\star\big[G\, \nabla_{0,1}(\omega_{y-x}^2)\big]\\
 & = -(2\ell)^{-1} \sum_{\vert x \vert \leqslant \ell_\varphi} x\ \E_1^\star\big[G\, \nabla_{0,1}(\omega_{0}^2) \big]+ (x+1) \E_1^\star\big[G\, \nabla_{0,1}(\omega_{1}^2)\big]\\
 & =- (2\ell)^{-1} (2\ell_\varphi+1)\ \E_1^\star\big[G(\omega_0^2-\omega_1^2)\big] \xrightarrow[\ell\to\infty]{} -\ll \varphi \gg_{1,\star \star}.
 \end{align*}
 The last limit comes from Proposition \ref{prop:directcomputations}  and the fact that $\ell_\varphi=\ell-s_\varphi-1$. We also have used the translation invariance of $\P_1^\star$. Then, we use the variational formula \eqref{def:varfor}, chosing 
 \[h=(\S_{\Lambda_\ell})^{-1}(aJ_\ell + H_\ell^f)=a\sum_{y \in \Lambda_\ell} y \omega_y^2+\sum_{\vert y \vert \leqslant \ell_f} \tau_y f,\] 
 we obtain:
 \begin{align*}
 &\liminf_{\ell\to\infty}  \, (2\ell)^{-1} \E_1^\star\bigg[ \big(-\S_{\Lambda_\ell}\big)^{-1}
 \sum_{\vert x \vert \leqslant \ell_\varphi}\tau_x\varphi,\sum_{\vert x \vert \leqslant \ell_\varphi}\tau_x \varphi  \bigg] \\ 
 &\quad \geqslant \liminf_{\ell \to \infty} \,  (2\ell)^{-1} \bigg\{ 2 \E_1^\star\bigg[ \sum_{\vert x \vert \leqslant \ell_\varphi}\tau_x \varphi
 ,  \, a\sum_{y \in \Lambda_\ell} y \omega_y^2+\sum_{\vert y \vert \leqslant \ell_f} \tau_y f \bigg]\\ 
& \hspace{20em}+\E_1^\star\Big[ a\sum_{y \in \Lambda_\ell} y \omega_y^2+\sum_{\vert y \vert \leqslant \ell_f} \tau_y f, aJ_\ell +H_\ell^f \Big]\bigg\}\\
 & \quad = 2a \ll \varphi \gg_{1,\star\star}  + 2 \ll \varphi, f\gg_{1,\star}  - 
 \frac{\lambda}{2}\E_1^\star\Big[\big(a(\omega_0^2-\omega_1^2)+\nabla_{0,1}\Gamma_f\big)^2\Big]-\frac{\gamma}{2}\E_1^\star\Big[\big(\nabla_0\Gamma_f\big)^2\Big].  
 \end{align*}
The result follows after taking the supremum on $f \in \Q$ and $a\in \R$, and recalling \eqref{eq:norm1bis}. }

\subsection{Proof of \ccl{Proposition \ref{prop:macro2}} \label{ssec:macro2}}

In this paragraph, we  \ccl{prove Proposition \ref{prop:macro2}} by using the  result given in \ccl{Proposition \ref{prop:variance_equ}, and leaving technical parts for the next two sections.}  First, we show how to relate \eqref{eq:macro2} to such variances, as was rapidly sketched in Section \ref{sec:insight}. Recall that we have assumed for convenience $\beta=1$, but the same argument remains in force for any $\beta>0$. 

\begin{prop} \label{prop:variance_bound}
Let $\psi \in \cC_0$, with $s_\psi \leqslant N$. Then \begin{equation}
\ccl{\mathbb Q_{\mu_1^N}}\left[ \sup_{0\leqslant t \leqslant T} \left\{\int_0^t \psi(s) \, {\rm d}s\right\}^2\right] \leqslant \frac{24T}{N^2}\E^\star_{1}\big[ \psi, (-\S_N)^{-1} \psi\big]. \label{eq:vari}
\end{equation}
\end{prop}
This result is proved for example in \cite[Section 2, Lemma 2.4]{MR2952852}, when there is no disorder. The  average w.r.t.~the disorder can be added (as in the estimate \eqref{eq:vari}) without any trouble, since $\S_N$ does not depend on $\m$.
 We are going to use this bound for functions of type $\sum_x G(x/N) \tau_x \varphi$, where $\varphi$ belongs to $\Q_0$. The main result of this subsection is the following.

\begin{prop}\label{prop:relation_variance}
Let $\varphi \in \Q_0$, and $G \in C^2(\T)$. Then, \begin{equation}
\limsup_{N\to\infty}\ccl{\mathbb Q_{\mu_1^N}}\left[ \sup_{0\leqslant t \leqslant T} \bigg\{\sqrt N\int_0^t \sum_{x\in\T_N}G\left(\frac x N\right)\tau_x \varphi(\m,s){\rm d}s\bigg\}^2\;\right] \leqslant CT \vertiii{ \varphi }_{1}^2 \int_\T G^2(u) {\rm d}u. \label{eq:theo_variance}
\end{equation}
\end{prop}

\begin{proof}
From Proposition \ref{prop:variance_bound}, the left hand side of \eqref{eq:theo_variance} is bounded by 
\begin{equation*}
24T\, \E_1^\star\bigg[ \sqrt N \sum_{x\in\T_N} G\left(\frac x N\right) \tau_x \varphi, (-N^2 \S_N)^{-1} \bigg(\sqrt N \sum_{x\in\T_N} G\left(\frac x N \right) \tau_x \varphi\bigg)\bigg],
\end{equation*} 
which can be written with the variational formula as 
\begin{equation*}
24T \,\sup_{f \in \cC} \bigg\{\sqrt N \sum_{x\in\T_N} G\left(\frac x N\right) \E_1^\star\big[ f\, \tau_x \varphi\big] -N^2\cD_N(\P^N_1;f) \bigg\}.
\end{equation*}
Since $\varphi \in \Q_0$, from Proposition \ref{prop:variational_restriction} we can restrict the supremum over $f \in \cQ$. Proposition \ref{lem:ibp} gives
\begin{equation*}
\E_1^\star\big[f \, \tau_x \varphi\big] \leqslant C(\varphi,\gamma) \E_1^\star\Big[ \tau_{-x} f, (-\S_{\Lambda_\varphi}) (\tau_{-x} f)\Big]^{1/2}
\end{equation*} and by Cauchy-Schwarz inequality, \begin{equation*}
\sqrt N \sum_{x\in\T_N}G\left(\frac x N\right)\E_1^\star\big[f\, \tau_x\varphi\big]\leqslant\bigg(\frac 1 N\sum_{x\in\T_N}G\left(\frac x N\right)^2\bigg)^{1/2} N \, C(\varphi,\gamma)\,\E_1^\star\big[f, (-\S_N) f\big]^{1/2}.
\end{equation*}
The supremum on $f$ can be explicitly computed, and gives the final bound
\begin{equation}
 \ccl{\mathbb Q_{\mu_1^N}}\bigg[ \sup_{0\leqslant t \leqslant T} \bigg\{\sqrt N\int_0^t \sum_{x\in\T_N}G\left(\frac x N\right)\tau_x \varphi(\m,s)\text ds\bigg\}^2\;\bigg] \leqslant  C'(\varphi,\gamma) T\bigg(\frac{1}{N} \sum_{x\in\T_N} G\left( \frac x N\right)^2\bigg). \label{eq:boundvariance}
\end{equation}
We are now going to  show that, after sending $N$ to infinity, the constant on the right hand side is proportional to $\vertiii{ \varphi }_1^2$. For that purpose, we average on microscopic boxes: for $\ell \ll N$, we denote 
\begin{equation*} 
\overline \varphi_\ell=\frac{1}{2\ell_\varphi+1}\sum_{|y|\leq \ell_\varphi} \tau_y \varphi, \end{equation*}
where as before $\ell_\varphi=\ell-s_\varphi-1$. We want to substitute 
\begin{equation*}
\sqrt N \sum_{x\in\T_N} G\left(\frac x N\right) \tau_x \varphi
\end{equation*} with 
\begin{equation*}
\sqrt N\sum_{x\in\T_N} G\left(\frac x N\right) \tau_x\overline\varphi_\ell.
\end{equation*}
The error term that appears is estimated by \begin{equation*}
\ccl{\mathbb Q_{\mu_1^N}}\bigg[ \sup_{0\leqslant t \leqslant T} \bigg\{\sqrt N\int_0^t \sum_{\substack{x,y\in\T_N\\ \vert x-y\vert \leqslant \ell_\varphi}}\frac{1}{2\ell_\varphi+1}\left(G\left(\frac x N\right)- G\left(\frac y N \right)\right)\tau_x \varphi(\m,s)\text ds\bigg\}^2\;\bigg].
\end{equation*}
Since $G_\ell(x):=G(x/N)-(2\ell_\varphi+1)^{-1}\sum_{\vert y-x\vert \leqslant \ell_\varphi} G(y/N)$ is of order $\ell/N$, we obtain from \eqref{eq:boundvariance} that the expression above is bounded by $C(\ell)/N^2$, and therefore vanishes as $N \to \infty$. We are now reduced to estimate 
\begin{equation}
\ccl{\mathbb Q_{\mu_1^N}}\bigg[ \sup_{0\leqslant t \leqslant T} \bigg\{\sqrt N\int_0^t \sum_{x\in\T_N}G\left(\frac x N\right)\tau_x \overline \varphi_\ell(\m,s)\text ds\bigg\}^2\;\bigg].\label{eq:ty}
\end{equation}
Using once again \eqref{eq:vari} and the variational formula for its right hand side, one obtains straightforwardly, using the translation invariance of $\P_1^\star$, that \eqref{eq:ty} is bounded by 
\begin{align*}
CT\sup_{g\in\cQ}\bigg\{\sqrt{N}&\sum_{x\in\T_N}  G\left(\frac x N\right) \E_1^\star\big[g\, \tau_x \overline \varphi_\ell\big] - N^2\E_1^N\big[ g, \left(-\S_N\right) g\big]  \bigg\}\\
& \leq CT\sup_{g\in\cQ}\bigg\{\sqrt{N}\sum_{x\in\T_N}  G\left(\frac x N\right) \E_1^\star\big[\tau_{-x}g\,  \overline \varphi_\ell\big] - \frac{N^2}{2\ell+1}\sum_{x\in \T_N}\E_1^\star\left[ g, \big(-\S_{\Lambda_\ell(x)}\big) g\right]  \bigg\}\\
& \leq \frac{CT(2\ell+1)}{N}\sum_{x\in\T_N} G^2\left(\frac x N\right) \sup_{f\in\cQ}\Big\{\E_1^\star\big[f\, \overline \varphi_\ell \big] - \E_1^\star\big[ f,\left(-\S_{\Lambda_\ell}\right) f\big]  \Big\}\\
& \leq \frac{ CT(2\ell+1)}{N}\sum_{x\in\T_N} G^2\left(\frac x N\right) \sup_{f\in\cQ_\ell}\Big\{\E_1^\star\big[f\, \overline \varphi_\ell \big] - \E_1^\star\big[ f, \left(-\S_{\Lambda_\ell}\right) f\big]  \Big\},
\end{align*}
where in the last inequality we denote by $\cQ_\ell$ the set of functions in $\cQ$ depending only on the sites in $ \Lambda_{\ell-1}$. To obtain the second bound, we split the supremum over $x$, and let $f:=(2\ell+1)\tau_{-x}g/G(x/N)$, and to obtain the third bound, we used the convexity of the Dirichlet form to replace $f$ by its conditional expectation w.r.t.~sites in $\Lambda_{\ell-1}$. Since $\varphi\in \cQ_0$, from Corollary \ref{cor:wd}, one straightforwardly obtains, using the polarization identity related to \eqref{eq:dirich} and the elementary inequality $ab\leq \frac14 a^2+b^2$, that 
\begin{align}\E_1^\star\big[f\, \overline \varphi_\ell \big]& \!=\! \E_1^\star\big[f (- \S_{\Lambda_\ell})(-\S_{\Lambda_\ell})^{-1}\, \overline \varphi_\ell\big] \nonumber\\ 
&\! =\! \frac \gamma 2 \sum_{x =-\ell}^\ell \E_1^\star\big[ \nabla_x f , \nabla_x \big((-\S_{\Lambda_\ell})^{-1}\,\overline\varphi_\ell\big)\big]+ \frac \lambda 2  \sum_{x=-\ell }^{\ell-1} \E_1^\star\big[ \nabla_{x,x+1} f ,  \nabla_{x,x+1} \big((-\S_{\Lambda_\ell})^{-1}\,\overline\varphi_\ell\big)\big]\nonumber\\
&\!\leq\! \frac{1}{4}\E_1^\star\big[ \overline\varphi_\ell , (-\S_{\Lambda_\ell})^{-1}\,\overline\varphi_\ell\big]+\E_1^\star\big[ f(-\S_{\Lambda_\ell})f\big]\label{eq:crossproduct}.\end{align}
We can now plug this bound in the previous estimate, let $\ell\to \infty$ after $N\to\infty$ and use \ccl{Proposition \ref{prop:variance_equ}} to finally obtain as wanted
 \begin{equation}
\limsup_{N\to\infty} \ccl{\mathbb Q_{\mu_1^N}}\left[ \sup_{0\leqslant t \leqslant T} \bigg\{\sqrt N\int_0^t \sum_{x\in\T_N}G\left(\frac x N\right)\tau_x \varphi(\m,s){\rm d}s\bigg\}^2\;\right] \leqslant CT \vertiii{ \varphi }_{1}^2 \int_\T G^2(u) {\rm d}u. 
\end{equation}
\end{proof}

\subsubsection{Proof of \ccl{Proposition \ref{prop:macro2} (i)}}  \label{subseci}

We apply Proposition \ref{prop:relation_variance} to $\mathfrak I^{1,N}_{t,\m,f}(H)$, and we get \begin{equation*}
\limsup_{N\to\infty} \ccl{\mathbb Q_{\mu_1^N}}\bigg[ \sup_{0\leqslant t \leqslant T} \left(\mathfrak I^{1,N}_{t,\m,f}(H)\right)^2\bigg] \leqslant CT \vertiii{ j_{0,1}+D(\omega_1^2-\omega_0^2)+\cL^\m f }_1^2\; \int_\T H'(u)^2 \text du.
\end{equation*}
\ccl{We will show in Section \ref{sec:hilbert}, Lemma \ref{lem:def_d} the following result: 
there exists a unique number $\widetilde{D}$, and a sequence of cylinder functions $\{f_k\} \in \Q$ such that \begin{equation*}
\vertiii{ j_{0,1} + \widetilde{D}(\omega_1^2-\omega_0^2)+\cL^\m f_k }_1 \xrightarrow[k\to\infty]{}0.
\end{equation*}
Recall that for any quadratic function $f\in \cQ$, $\vertiii{ f }_\beta=\beta^{-2}\vertiii{ f }_1$, therefore in particular, this convergence also holds with the same constant $\widetilde{D}$ and the same sequence $\{f_k\}$ if we replace the semi-norm $\vertiii{ \cdot }_1$ with $\vertiii{ \cdot }_\beta$ for any $\beta>0$ (as a consequence of a standard change of variables argument). To conclude the proof of Proposition \ref{prop:macro2}, it only remains to prove the identity $\widetilde{D}=D$ to obtain the first statement of Proposition \ref{prop:macro2}, which will be done below in Remark \ref{rem:DDt}.} 

\subsubsection{Proof of \ccl{Proposition \ref{prop:macro2} (ii)}}  \label{subsecii}

The second statement  \eqref{eq:BG_principle} will be proved in Section \ref{sec:diffusion}, see Proposition \ref{prop:limit} below.


\section{Construction of the sequence $\{f_k\}$\label{sec:hilbert}}

We now focus on the semi-norm $\vertiii{ \cdot }_1$ that was introduced in the previous section, see \eqref{eq:norm1}. 
We can easily define from $\vertiii{ \cdot }_1$ a semi inner product on $\cC_0$ through polarization, 
which is denoted by $\ll \cdot, \cdot \gg_1$. Let $\mathcal{N}$ be the kernel of the semi-norm $\vertiii{ \cdot }_1$ on $\cC_0$. 
Then, the completion of  $\Q_0\ccl{/\mathcal{N}}$ \ccl{(the quotient of $\Q_0$ by $\mathcal{N}$),} denoted by $\mathcal{H}_1$ is a Hilbert space.  
Let us explain how Varadhan's non-gradient  approach is modified. Usually, the Hilbert space on which orthogonal projections 
are performed is the completion of $\cC_0\ccl{/\mathcal{N}}$, in other words it involves all local functions. 
Then, the standard procedure aims at proving that each element of that Hilbert space can be approximated by a sequence of functions 
in the range of the generator plus an additional term which is proportional to the current. 
Since for our model, all functions of interest are in $\cQ$, and since the decomposition of germs of closed form is explicit 
in the set $\cQ$ (recall \eqref{eq:id11} and  \eqref{eq:id22}), the crucial step to obtain this decomposition  is to control 
the antisymmetric part of the generator by the symmetric one for quadratic functions.

In Subsection \ref{ssec:sym_decomp}, we show that $\mathcal{H}_1$ is the completion of $\S\Q\ccl{/\mathcal{N}}+\ccl{\R j_{0,1}^S}$. 
In other words, all elements of $\mathcal{H}_1$ can be approximated by $aj_{0,1}^S+\S g$ for some $a\in\R$ and $g \in \Q$. 
This is quite natural since the symmetric part of the generator preserves the degree of polynomial functions. Moreover, 
the two subspaces $\ccl{\R j_{0,1}^S}$ and $\overline{\S\Q}\ccl{/\mathcal{N}}$ are orthogonal, and we  denote their sum by 
\[  \overline{\S\Q}\ccl{/\mathcal{N}} \oplus^\perp \ccl{\R j_{0,1}^S}.\]
Nevertheless, this decomposition is not satisfactory, because we want the fluctuating term to be on the form $\cL^\m(f_k)$, 
and not $\S(f_k)$. In order to make this replacement, we need to prove  the weak {sector condition}, that gives a control 
of $\vertiii{ \A^\m g }_1$ by $\vertiii{ \S g }_1$,  when $g$ is a quadratic function. 
The argument is explained is Subsection \ref{ssec:replacement} and \ref{ssec:decomposition}, and the weak sector condition is proved 
in Appendix \ref{app:sector}.  The only trouble is that this new decomposition is no longer orthogonal, so that we can not directly express 
the diffusion coefficient as a variational formula, like \eqref{eq:diffusion}. This problem is solved in Section \ref{sec:diffusion}. 

\subsection{Decomposition according to the symmetric part\label{ssec:sym_decomp}}

 We begin this subsection with a table of calculus, very useful in the sequel. Recall that $\ll \cdot,\cdot \gg_1$ is obtained by polarization from the norm $\vertiii{\cdot}_1$ defined in Definition \ref{de:norm1}, and also that $\ll \cdot \gg_{1,\star}$ and $\ll \cdot \gg_{1,\star\star}$ have been defined in \eqref{eq:prod1}.

\bprop{\label{prop:identities}For any $\varphi \in \Q_0$ and $ g \in \Q$ (which implies $\S g \in \Q_0$ from Proposition \ref{prop:propertiesC0}), \ccl{we have}
\begin{align}
\ll \varphi,\S g  \gg_{1} &=-\ll \varphi, g\gg_{1,\star} \label{eq:id1}\\
 \ll \varphi, j_{0,1}^S \gg_1 & = -\ll \varphi \gg_{1,\star\star}\label{eq:id2}\\
\ll j_{0,1}^S,\S g\gg_1 & =0\label{eq:id3},
\end{align}
and then
\begin{align}
\vertiii{ j_{0,1}^S }^2_{1} & = - \ll j_{0,1}^S \gg_{1,\star\star} \; = \lambda \chi(1) \label{eq:id4}\\
\vertiii{\S g}_1^2 & = \frac{\lambda}{2}\E_1^\star\left[(\nabla_{0,1}\Gamma_g)^2\right]+\frac{\gamma}{2}\E_1^\star\left[(\nabla_0\Gamma_g)^2\right]\label{eq:id5}.
\end{align}}

\bprf{These identities are consequences of \ccl{Proposition \ref{prop:variance_equ}}. \ccl{Fix two local functions $\varphi \in \Q_0$ and $ g \in \Q$, we first write according to equation \eqref{eq:limH-1} and by translation invariance of $\E_1^\star $
\[\ll \varphi,\S g  \gg_{1}=-\lim_{\ell \to \infty} (2\ell)^{-1} \E_1^\star \bigg[\sum_{\vert x \vert \leqslant \ell_\varphi}\tau_x\varphi,\sum_{\vert y \vert \leqslant \ell_{\S g}}\tau_y g\bigg].\]
For any $|x|\leq \ell_{\S g}-s_\varphi-s_{ g}$, and any $y>\ell_{\S g}$,  $\tau_x\varphi$ and $\tau_y g$ have disjoint support and are therefore independent, so that we obtain as wanted 
\[\ll \varphi,\S g  \gg_{1}=-\lim_{\ell \to \infty}\Bigg\{ (2\ell)^{-1} \E_1^\star \bigg[\sum_{\vert x \vert \leqslant \ell_{\S g}-s_\varphi-s_{ g}}\tau_x\varphi,\sum_{y\in \Z }\tau_y g\bigg]+ \mathcal O(\ell^{-1})\Bigg\}=-\ll \varphi, g\gg_{1,\star},\]
where $\mathcal{O}(\ell^{-1})$ denotes a small correction term of order $\ell^{-1}$, which proves \eqref{eq:id1}.

The second identity is proved in the same way, using the elementary identity 
\[\sum_{-\ell\leq x\leq \ell-1 }j^S_{x,x+1}=\lambda \S_{\Lambda_\ell} \left(\sum_{|x|\leq \ell} x\omega^2_x\right).\]
We can then write $\ll \varphi,j_{0,1}^S  \gg_{1}$ as the limit as $\ell\to\infty$ of 
\[- (2\ell)^{-1} \E_1^\star \bigg[\sum_{\vert x \vert \leqslant \ell_{\varphi}}\tau_x\varphi,\sum_{y\in \Lambda_\ell}y \omega_y^2\bigg]+(2\ell)^{-1} \E_1^\star \bigg[j_{-\ell, 1-\ell}^S,(-\S_{\Lambda_\ell })^{-1}\sum_{\vert x \vert \leqslant \ell_{\varphi}}\tau_x\varphi \bigg].\]
Similarly to the proof of \eqref{eq:id1}, the first term above converges as $\ell\to\infty$ to $-\ll \varphi \gg_{1,\star\star}$. Therefore we only need to prove that the second term vanishes. To do so, we apply \eqref{eq:crossproduct} to $f:=\ell^{-1/4}(-\S_{\Lambda_\ell })^{-1}j_{-\ell,1-\ell}^S$ and $\widetilde{\varphi}=\ell^{1/4}\varphi$, to obtain
\begin{multline}
(2\ell)^{-1} \E_1^\star \bigg[j_{-\ell, 1-\ell}^S,(-\S_{\Lambda_\ell })^{-1}\sum_{\vert x \vert \leqslant \ell_{\varphi}}\tau_x\varphi \bigg]\\
\leq\frac{1}{\sqrt{\ell}} \left\{\frac{1}{16\ell}\E_1^\star\bigg[\sum_{\vert x \vert \leqslant \ell_{\varphi}}\tau_x\varphi, (-\S_{\Lambda_\ell})^{-1}\,\sum_{\vert x \vert \leqslant \ell_{\varphi}}\tau_x\varphi\bigg]+ \E_1^\star\Big[ j_{-\ell,1-\ell}^S(-\S_{\Lambda_\ell})^{-1}j_{-\ell,1-\ell}^S\Big]\right\}.
\end{multline}
The first term inside the braces above converges as $\ell\to\infty$, according to Proposition \ref{prop:variance_equ}, to $\vertiii{\varphi}_1/8$. Furthermore, using the variational formula for the second term, it is straightforward to prove that it is bounded from above by a constant as well. The right hand side above is therefore of order $\ell^{-1/2}$ and vanishes as wanted as $\ell\to\infty$.

Identity \eqref{eq:id3} is immediate given that $\ll j_{0,1}^S, g\gg_{1,\star} =0$ by translation invariance of $\E_1^\star $. Similarly, \eqref{eq:id4} follows from \eqref{eq:id2} and elementary computations, and  \eqref{eq:id5} follows from \eqref{eq:id1}.
}
}

\bcor{\label{corr}For all $a\in\R$ and $g\in\Q$, \begin{equation*}
\vertiii{ aj_{0,1}^S+\S g }^2_1=a^2 \lambda\chi(1)+\frac{\lambda}{2}\E_1^\star\left[(\nabla_{0,1}\Gamma_g)^2\right]+\frac{\gamma}{2}\E_1^\star\left[(\nabla_0\Gamma_g)^2\right].
\end{equation*} In particular, the variational formula for $\vertiii{\varphi }_1$, $\varphi \in \Q_0$, writes 
\begin{equation} 
\vertiii{ \varphi }^2_1=\frac{1}{\lambda\chi(1)}\ll \varphi,j_{0,1}^S\gg_1^2+\sup_{g \in \Q} \left\{2\ll \varphi,(-\S) g\gg_1- \vertiii{ \S g }^2_1\right\}.
\label{eq:variational_new}
\end{equation}}
\bprop{\label{cor:decomp} 
We denote by $\S\Q$ the space $\{\S g \ ; \ g \in \Q\}$. Then, 
\begin{equation*}
\cH_1=\overline{\S\Q}\ccl{/\mathcal{N}}\oplus^\perp\ccl{\R j_{0,1}^S}.
\end{equation*}}

\begin{proof} We divide the proof into two steps.
\paragraph{\small (a) The space is well generated --} 

The inclusion $\overline{\S\Q}\ccl{/\mathcal{N}}+\ccl{\R j_{0,1}^S} \subset \cH_1$ is obvious (and follows from Proposition \ref{prop:propertiesC0}). Moreover, from the variational formula \eqref{eq:variational_new} we know that: if $h\in\cH_1$ satisfies $\ll h,j_{0,1}^S\gg_1=0$ and $\ll h,\S g \gg_1 = 0$ for all $g \in \Q$, then $\vertiii{ h }_1=0$.

\paragraph{\small (b) The sum is orthogonal -- } 

This follows directly from the previous proposition and from the fact that: $\ll j_{0,1}^S,\S g \gg_1=0$ for all $g \in \Q$.  
\end{proof}

\subsection{Replacement of $\S$ with $\cL$ \label{ssec:replacement}}

In this subsection, we prove identities which mix the antisymmetric and the symmetric part of the generator, which will be used to get the \emph{weak} sector condition (Proposition \ref{prop:sector}).

\blem{\label{lem:AS}For all $g,h \in \Q$,\begin{equation*}
\ll \S g, \A^\m h \gg_1= - \ll \A^\m g, \S h \gg_1.
\end{equation*}}

\bprf{ This easily follows from the first identity of Proposition \ref{prop:identities} and from the invariance by translation of the measure $\P_1^\star$: \begin{align*}
\ll \S g,\A^\m h\gg_1 & = -\ll g, \A^\m h \gg_{1,\star} = -\sum_{x \in \Z} \E_1^{\star}\big[{\tau_x g, \A^\m h}\big] = \sum_{x \in \Z} \E_1^\star\big[\A^\m (\tau_x g),h\big] \\
& = \sum_{x\in\Z} \E_1^\star\big[\tau_x(\A^\m g),h\big]=\sum_{x \in \Z} \E_1^\star\big[\A^\m g,\tau_{-x}h\big]=\sum_{x \in \Z} \E_1^\star\big[\A^\m g,\tau_x h\big]\\
&=- \ll \A^\m g,\S h\gg_1.
\end{align*}
}

\blem{\label{lem:AandS}For all $g \in \Q$, 
\begin{equation*}
\ll \S g,j_{0,1}^A \gg_1=-\ll \A^\m g, j_{0,1}^S\gg_1.
\end{equation*}}

\bprf{From Proposition \ref{prop:identities}, \begin{align*}
\ll \S g,j_{0,1}^A\gg_1 &= - \ll g,j_{0,1}^A\gg_{1,\star} = -\sum_{x \in \Z} \E_1^\star\big[\tau_x g, j_{0,1}^A\big]= -\sum_{x \in \Z} \E_1^{\star}\big[g,j_{x,x+1}^A\big] \\ & =-\sum_{x \in \Z} x \E_1^\star\big[g,j_{x-1,x}^A-j_{x,x+1}^A\big]  = -\sum_{x \in \Z} x\E_1^\star\big[g, \A^\m(\omega_x^2)\big]\\ &=\sum_{x \in \Z} x \E_1^\star\big[\A^\m g,\omega_x^2\big]= \ll \A^\m g \gg_{1,\star\star}=-\ll \A^\m g,j_{0,1}^S\gg_1.
\end{align*}
}
These two lemmas together with the second identity of Proposition \ref{prop:identities} (and the fact that $\ll j_{0,1}^A \gg_{1,\star\star}=0$) imply the following:

\bcor{\label{cor:productnull}For all $a \in \R$, $g \in \Q$, \begin{equation*} \ll a j_{0,1}^S + \S g, a j_{0,1}^A + \A^\m g \gg_1 =0. \end{equation*}}
We now  state the main result of this subsection.
\bprop{[Weak sector condition]\label{prop:sector} \begin{enumerate}[(i)] \item There exist two constants $C_0:=C(\gamma,\lambda)$ and $C_1:=C(\gamma,\lambda)$ such that the following inequalities hold for all $f,g \in \Q$: 
\begin{align}
\label{eq:sector_condition}\left\vert \ll \A^\m f, \S g \gg_1\right\vert &\leqslant C_0 \vertiii{ \S f }_1\, \vertiii{ \S g }_1. \\
\left\vert \ll \A^\m f, \S g \gg_1\right\vert &\leqslant  C_1 \vertiii{ \S f }^{2}_1+ \frac{1}{2}\vertiii{ \S g }^{2}_1.
\end{align}
\label{sector_item1}
\item There exists a positive constant $C$ such that, for all $g \in \Q$, \begin{equation*} \vertiii{ \A^\m g }_1\ \leqslant C \vertiii{ \S g }_1.\end{equation*} \label{sector_item2}\end{enumerate} }

\bprf{ The proof is technical because made of explicit computations for quadratic functions. For that reason, we \ccl{give it in Appendix} \ref{app:sector}. }

\subsection{Decomposition of the Hilbert space \label{ssec:decomposition}}
We deduce from the previous two subsections the expected decomposition of $\mathcal{H}_1$.

\bprop{\label{prop:decomposition}We denote by $\cL^\m\Q$ the space $\{\cL^\m g \ ; \ g \in \Q\}$. Then, \begin{equation*}
\cH_1=\overline{\cL^\m \Q}\ccl{/\mathcal{N}} \oplus \ccl{\R j_{0,1}^S}. \end{equation*}}

\bprf{We first prove that $\cH_1$ can be written as the sum of the two subspaces. Then, we show that the sum is direct.
\paragraph{\small (a) The space is well generated -- } The inclusion $\overline{\cL^\m \Q}\ccl{/\mathcal{N}} + \ccl{\R j_{0,1}^S} \subset \cH_1$ follows from Proposition \ref{prop:propertiesC0}.  To prove the converse inclusion, let $h \in \cH_1$ so that $\ll h, j_{0,1}^S \gg_1=0$ and $ \ll h, \cL^\m g \gg_1=0$ for all $g \in \Q$.  From Proposition \ref{cor:decomp}, $h$ can be written as \begin{equation*}
h=\lim_{k\to\infty} \S g_k 
\end{equation*} for some sequence $\{g_k\} \in \Q$. More precisely, since $\ll \S g_k, \A^\m g_k \gg_1=0$ by 
Lemma \ref{lem:AS}, 
\begin{equation*}
\vertiii{ h }_1^2 = \lim_{k \to\infty} \ll \S g_k, \S g_k \gg_1=\lim_{k \to\infty} \ll \S g_k, \cL^\m g_k\gg_1.
\end{equation*}Moreover, we also have by assumption that $\ll h, \cL^\m g_k \gg_1 = 0$ for all $k$, and from Proposition \ref{prop:sector}, \begin{equation*}
\sup_{k \in \N} \vertiii{ \cL^\m g_k }_1 \leqslant (C+1) \sup_{k\in\N} \vertiii{ \S g_k }_1=:C_h
\end{equation*} is finite. Therefore, \begin{equation*}
\vertiii{ h }_1^2=\lim_{k\to\infty} \ll \S g_k, \cL^\m g_k\gg_1=\lim_{k\to\infty} \ll \S g_k-h,\cL^\m g_k \gg_1 \, \leqslant \lim_{k\to\infty} C_h \vertiii{ \S g_k -h }_1=0.
\end{equation*} 
\paragraph{\small (b) The sum is direct -- } Let $\{g_k\} \in \Q$ be a sequence  such that, for some $a \in \R$, \begin{equation*}
\lim_{k \to \infty} \cL^\m g_k=aj_{0,1}^S \quad \text{ in } \cH_1,
\end{equation*} By a similar argument, \begin{equation*}
\limsup_{k\to\infty} \ll \S g_k, \S g_k \gg_1=\limsup_{k\to\infty} \ll \cL^\m g_k,\S g_k\gg_1=\limsup_{k\to\infty} \ll \cL^\m g_k - a j_{0,1}^S, \S g_k \gg_1=0,
\end{equation*} where the last equality comes from the fact that $\ll j_{0,1}^S,\S g_k\gg_1=0$ for all $k$. On the other hand, by Proposition \ref{prop:sector}, $\vertiii{ \cL^\m g_k }_1 \leqslant (C+1) \vertiii{ \S g_k }_1 \to 0$. Then, $a=0$. This concludes the proof.  }
Recall that $j_{0,1}^S(\m,\omega)=\lambda(\omega_0^2-\omega_1^2)$. We have obtained the following result.
\blem{\label{lem:def_d}For every $g \in \Q_0$, there exists a unique constant $a \in \R$, such that \begin{equation}
g + a(\omega_1^2-\omega_0^2) \in \overline{\cL^\m \Q} \quad \text{ in } \cH_1.
\end{equation}}
This lemma states \ccl{exactly what we were expecting for:} there exists a unique number $\widetilde{D}$, and a sequence of cylinder functions $\{f_k\} \in \Q$ such that \begin{equation*}
\vertiii{ j_{0,1} + \widetilde{D}(\omega_1^2-\omega_0^2)+\cL^\m f_k }_1 \xrightarrow[k\to\infty]{}0,
\end{equation*}
and this convergence also holds with the same constant $\widetilde{D}$ and the same sequence ${f_k}$ if we replace the semi-norm $\vertiii{ \cdot }_1$ with $\vertiii{ \cdot }_\beta$ for any $\beta>0$. 

\ccl{We are left to prove: first, the identity $\widetilde{D}=D$,  and second,  the statement (ii) of Proposition \ref{prop:macro2}, namely \eqref{eq:BG_principle}, both being contained in the following section.}

\section{Diffusion coefficient and end of the proof \label{sec:diffusion}}

The main goal of this section is to express the diffusion coefficient by various variational formulas. We also prove the second statement of \ccl{Proposition \ref{prop:macro2}}. 
First, recall the we defined the coefficient $D$ in Definition \ref{def:diffusion} as
\begin{equation} 
D=\lambda+\frac{1}{\chi(1)}\inf_{f \in\Q}\sup_{g \in \Q}\left\{\ll f,-\S f \gg_{1,\star}+2\ll j_{0,1}^A-\A^{\mathbf m} f,g\gg_{1,\star}-\ll g , -\S g \gg_{1,\star}\right\}. 
\label{eq:diffusion}
\end{equation} 
From Lemma  \ref{lem:def_d}, there exists a unique $\widetilde{D}\in \R$ such that \begin{equation*}
j_{0,1} + \widetilde{D}(\omega_1^2-\omega_0^2) \in \overline{\cL^\m \Q} \quad \text{ in } \cH_1.
\end{equation*}
 We are going to obtain variational formulas for $\widetilde{D}$, and prove that $\widetilde{D}=D$, by following the argument in \cite{sasolla} \ccl{(see Remark \ref{rem:DDt} below)}. We first rewrite the decomposition of the Hilbert space given in Proposition \ref{prop:decomposition}, by replacing $j_{0,1}^S$ with $j_{0,1}$. This new statement is based on Corollary \ref{cor:productnull}, which gives an orthogonality relation. The second step is to find \ccl{another} orthogonal decomposition (see \eqref{eq:newdecomp} below), which will \ccl{justify} the variational formula \eqref{eq:diffusion} for $D$. 
Hereafter, we denote $\cL^{\m,\star}:=\S-\A^\m$  and $j_{0,1}^\star:=j_{0,1}^S-j_{0,1}^A.$

\blem{The following decompositions hold \begin{equation*}
\cH_1=\overline{\cL^\m \Q}\vert_{\mathcal N} \oplus \ccl{\R j_{0,1}}=\overline{\cL^{\m,\star} \Q}\vert_{\mathcal N} \oplus \ccl{\R j^\star_{0,1}}.
\end{equation*}}

\bprf{We only sketch the proof of the first decomposition, since it is done in \cite{sasolla}. Let us recall from Proposition \ref{prop:decomposition} 
that $\overline{\cL^\m \Q}$ has a complementary subspace in $\cH_1$ which is one-dimensional. Therefore, it is sufficient to prove that $\cH_1$ is generated by $\overline{\cL^\m \Q}$ and the total current. Let $h \in \cH_1$ such that $\ll h,j_{0,1}\gg_1=0$ and $ \ll h,\cL^\m g\gg_1=0$ for all $g \in \Q$. By Proposition \ref{cor:decomp}, $h$ can be written as \begin{equation*}
h=\lim_{k\to\infty} \S g_k  + a j_{0,1}^S
\end{equation*} for some sequence $\{g_k\} \in \Q$, and $a\in\R$, and from Corollary \ref{cor:productnull}, \begin{align*}
\vertiii{ h }_1^2 &= \lim_{k \to\infty} \ll aj_{0,1}^S + \S g_k, a j_{0,1}+\cL^\m g_k \gg_1.
\end{align*} Moreover, from Proposition \ref{prop:sector} and the standard inequality $\vertiii{\varphi + \psi}_1^2 \leqslant 2 \vertiii{\varphi}_1^2 + 2 \vertiii{\psi}_1^2$, we have \begin{equation*} 
\sup_{k\in\N} \vertiii{ aj_{0,1} + \cL^\m g_k }_1^2 \leqslant 2 a^2 \vertiii{ j_{0,1} }^2_1 + 2(C+1) \sup_{k\in\N} \vertiii{ \S g_k }_1^2 =:C_h \end{equation*} is finite.
Therefore, \begin{align*}
\vertiii{ h }_1^2 & = \lim_{k\to\infty} \ll aj_{0,1}^S+\S g_k-h,aj_{0,1}+\cL^\m g_k \gg_1\\
& \leqslant  C_h \limsup_{k\to\infty} \vertiii{ aj_{0,1}^S+\S g_k - h}_1=0.
\end{align*}
The same arguments apply to the second decomposition. 
}

We define bounded linear operators $T, T^\star:\cH_1 \to \cH_1$ as \begin{align*}
T(aj_{0,1}+\cL^\m f) & := aj_{0,1}^S + \S f, \\
T^\star(aj_{0,1}^\star + \cL^{\m,\star} f) & := aj_{0,1}^S + \S f. 
\end{align*} 
From the following identity (which is a direct consequence of Corollary \ref{cor:productnull})
\begin{equation*}
\vertiii{ aj_{0,1}+\cL^\m f}_1^2=\vertiii{aj_{0,1}^\star+\cL^{\m,\star}f}_1^2=\vertiii{ aj_{0,1}^S+\S f }_1^2 + \vertiii{ aj_{0,1}^A+\A^\m f}^2_1,
\end{equation*} we can easily see that $T^\star$ is the adjoint operator of $T$ and we also have the relations
\begin{align*}
\ll Tj_{0,1}^S, j_{0,1}^\star \gg_1& =\ll T^\star j_{0,1}^S, j_{0,1}\gg_1 = \lambda \chi(1) \\
\ll T j_{0,1}^S, \cL^{\m,\star} f\gg_1&=\ll T^\star j_{0,1}^S,\cL^\m f \gg_1=0, \text{ for all } f\in\Q.
\end{align*}
In particular, \begin{equation}
\cH_1=\overline{\cL^{\m,\star}\Q}\ccl{/\mathcal{N}} \oplus^\perp \ccl{\R Tj_{0,1}^S} \label{eq:newdecomp}
\end{equation} 
and there exists a unique number $Q$ such that \begin{equation*}
j_{0,1}^\star-QTj_{0,1}^S \in \overline{\cL^{\m,\star} \Q} \quad \text{ in } \cH_1.
\end{equation*} We are going to show that $\tilde D=\lambda Q$.

\blem{\begin{equation}
Q=\frac{\lambda \chi(1)}{ \vertiii{ Tj_{0,1}^S }_1^2}=\frac{1}{\lambda \chi(1)} \inf_{f \in \Q} \vertiii{ j_{0,1}^\star - \cL^{\m,\star} f}^2_1.
\end{equation}}

\bprf{The first identity follows from the fact that \begin{equation*}
\ll T j_{0,1}^S, j_{0,1}^\star - Q Tj_{0,1}^S\gg_1=\lambda \chi(1)- Q \vertiii{ Tj_{0,1}^S}_1^2=0.
\end{equation*}
The second identity is  straightforwardly obtained from the first identity, together with 
\begin{equation}
\inf_{f\in\Q} \vertiii{ j_{0,1}^\star - Q Tj_{0,1}^S-\cL^{\m,\star} f}^2_1=0,
\label{eq:identity_inf}
\end{equation}  
which holds by construction of $Q$.
}

Thanks to Corollary \ref{cor:productnull}, for any $g\in \cQ$, $Tg$ and $g-Tg$ are orthogonal, and therefore $\ll Tg,g\gg_1=\ll Tg,Tg \gg_1$ for all $g \in \cH_1$. 
In particular, $j_{0,1}^S-Tj_{0,1}^S$ is orthogonal to $Tj_{0,1}^S$, thus 
\begin{equation*}
j_{0,1}^S-Tj_{0,1}^S \in \overline{\cL^{\m,\star}\Q}.
\end{equation*} 
We can then obtain the following variational formula for $\vertiii{ Tj_{0,1}^S}_1$.
\bprop{\begin{equation}
\vertiii{Tj_{0,1}^S}_1^2=\inf_{f\in\Q} \vertiii{ j_{0,1}^S-\cL^{\m,\star} f}^2_1.
\end{equation}}
\bprf{With a similar argument (as in the proof of the previous proposition), we have \begin{equation*}
\inf_{f\in\Q} \vertiii{ j_{0,1}^S-Tj_{0,1}^S-\cL^{\m,\star} f }_1=0,
\end{equation*} and \begin{equation*}
\inf_{f\in\Q} \vertiii{ j_{0,1}^S-Tj_{0,1}^S-\cL^{\m,\star} f }^2_1=\inf_{f\in\Q} \vertiii{ j_{0,1}^S-\cL^{\m,\star} f }^2_1- \vertiii{ Tj_{0,1}^S }^2_1,
\end{equation*}where we used the fact that $j_{0,1}^S-Tj_{0,1}^S$ and  $\cL^{\m,\star} f$ are both orthogonal to $ Tj_{0,1}^S$, which concludes the proof. }

We are now ready to derive variational formulas for $\widetilde{D}$:
\blem{\label{lem:diff1}
\begin{equation}
\label{eq:Dtvar}
\widetilde{D}=\frac{1}{\chi(1)}\inf_{f\in\Q} \vertiii{ j_{0,1}^\star-\cL^{\m,\star} f}^2_1 = \frac{\chi(1)}{4\inf_{f\in\Q} \vertiii{ j_{0,1}^S-\cL^{\m,\star} f}^2_1}.
\end{equation}}

\bprf{By construction, $j_{0,1}-(\widetilde{D}/\lambda)j_{0,1}^S \in \overline{\cL^{\m}\Q}$ and therefore \begin{equation}
\ll j_{0,1}-\frac{\widetilde{D}}{\lambda}j_{0,1}^S,T^\star j_{0,1}^S\gg_1=\lambda\chi(1)-\frac{\widetilde{D}}{\lambda}\vertiii{ Tj_{0,1}^S }^2_1 =0.
\end{equation} 
As a result, we obtain as wanted that, $\widetilde{D}=\lambda Q$, and the variational formula for $\widetilde{D}$ can be deduced from the one for $Q$. 
}

\brem{\label{rem:DDt}
We can rewrite the variational formula \eqref{eq:Dtvar} for $\widetilde{D}$ as: 
\begin{align}
\widetilde{D}& =\frac{1}{\chi(1)}\inf_{f\in\Q} \left\{ \vertiii{ j_{0,1}^S}^2_1+\vertiii{ \S f }^2_1 + \vertiii{j_{0,1}^A-\A^\m f}^2_1 \right\} \notag \\ 
& = \lambda + \frac{1}{\chi(1)}\inf_{f\in\Q} \left\{ \vertiii{ \S f }^2_1 + \vertiii{ j_{0,1}^A-\A^\m f }^2_1 \right\} \label{eq:line2} \\
& = \lambda + \frac{1}{\chi(1)}\inf_{f\in\Q} \sup_{g \in \Q} \left\{ \vertiii{ \S f }^2_1 - 2 \ll j_{0,1}^A-\A^\m f,\S g \gg_1 - \vertiii{ \S g }^2_{1} \right\} \notag\\
& = \lambda + \frac{1}{\chi(1)} \inf_{f\in\Q} \sup_{g \in \Q} \left\{ \ll f,-\S f \gg_{1,\star} + 2 \ll j_{0,1}^A - \A^\m f, g \gg_{1,\star} - \ll g, -\S g \gg_{1,\star} \right\} \label{eq:line4} \\
& =D , \label{eq:line5}
\end{align}
by definition of the diffusion coefficient, see \eqref{eq:diffusion}.
To establish the third identity, we used \eqref{eq:variational_new} to restrict the infimum in \eqref{eq:line2}, to functions $f$ satisfying $\ll j_{0,1}^A - \A^\m f,j_{0,1}^S \gg_1=0$.
}
We are now in position to prove the remaining statement of \ccl{Proposition \ref{prop:macro2}}:
\bprop{\label{prop:limit} \ccl{There exists a sequence $\{f_k\} \in \Q$ such that \begin{equation}
\label{eq:idhilb}
\lim_{k\to\infty} \vertiii{ j_{0,1} + D(\omega_1^2-\omega_0^2)+\cL^\m f_k }_1 =
\lim_{k\to\infty} \vertiii{ j^\star_{0,1} + D(\omega_1^2-\omega_0^2)+\cL^{\m,\star} f_k }_1 =0
,\end{equation}
and it satisfies}
\begin{equation*}
\lim_{k\to\infty} \E_1^\star\Big[ \lambda\Big(\nabla_{0,1}(\omega_0^2-\Gamma_{f_k})\Big)^2 +\gamma \Big(\nabla_{0}(\Gamma_{f_k})\Big)^2\Big]= 2D\chi(1).
\end{equation*} }
\bprf{
\ccl{The first statement is a consequence of Lemma \ref{lem:def_d}, the definition of $\widetilde{D}$,  Remark \ref{rem:DDt} and Corollary \ref{cor:productnull}. We now turn to the second statement.} By assumption,
\begin{equation*}
\lim_{k\to\infty} \vertiii{ T \big( j_{0,1} + D (\omega_1^2-\omega_0^2) + \cL^{\m} f_k \big) }_1=0
\end{equation*}
and therefore \begin{equation*}
\lim_{k\to\infty}\vertiii{ j_{0,1}^S+\S f_k}_1^2=D^2\vertiii{ T(\omega_1^2-\omega_0^2)}_1^2.
\end{equation*}
Then, the result follows from 
\begin{equation}
\label{eq:eqDTJs}
D=\lambda Q = \frac{\chi(1)}{\vertiii{ T(\omega_1^2-\omega_0^2)}_1^2}
\end{equation} 
and Corollary \ref{corr}, which yields 
\begin{equation}
\vertiii{ j_{0,1}^S + \S f_k }_1^2 = \frac{\lambda}{2}\E_1^\star\Big[ \Big(\omega_1^2-\omega_0^2-\nabla_{0,1}(\Gamma_{f_k})\Big)^2\Big] +\frac{\gamma}{2}\E_1^\star\Big[ \Big(\nabla_{0}(\Gamma_{f_k})\Big)^2\Big].
\end{equation}
}

\section{Green-Kubo formulas}\label{sec:gk}
In this  section, \ccl{we prove the Green-Kubo formula given in Theorem \ref{theo:GKF}, which is a consequence of Propositions \ref{prop:conv_GK} and \ref{prop:equivalence} below.} We first prove the convergence of the infinite volume Green-Kubo formula, then we rigorously show that it is equivalent to the diffusion coefficient given by Varadhan's approach.
For the sake of clarity, in the following we simplify notations, and we denote $\ll \cdot \gg_{1,\star}$ by $\ll \cdot \gg_\star$.

\subsection{Convergence of Green-Kubo formula}\label{ssec:conv}
Linear response theory predicts that the diffusion coefficient  is given by the {Green-Kubo formula}. In \cite[Section 3]{MR2448630} its homogenized infinite volume version  is given by: 
 \begin{equation}\overline \kappa(z)=\lambda+\frac{1}{2} \int_0^{+\infty} \text dt \ e^{-z t} \sum_{x \in \Z} \ccl{\mathbb Q_{\mu_1^N}}\Big[j_{0,1}^A(\m,t), \tau_xj_{0,1}^A(\m,0)\Big]. \label{eq:gkkk}\end{equation} 
 That formula can be guessed from the better-known finite volume Green-Kubo formula thanks to the ergodicity property of the disorder measure $\P$.
We denote by $L(z)$ the second term of the right hand side of \eqref{eq:gkkk}, that is \begin{equation*} L(z):= \frac{1}{2} \int_0^{+\infty} \text  dt e^{-z t}\   \sum_{x \in \Z} \ccl{\mathbb Q_{\mu_1^N}}\Big[j_{0,1}^A(\m,t),\tau_xj_{0,1}^A(\m,0)\Big].\end{equation*} 
We also denote by $\mathbf L^2_\star$ the Hilbert space generated by the elements of $\cC$ (recall Definition \ref{de:C}) and the inner product $\ll \cdot \gg_{\star}$.  
We define $h_z:=h_z(\m,\omega)$ as the solution to the resolvent equation in $\mathbf{L}^2_\star$ \begin{equation} (z-\cL^\m)h_z=j_{0,1}^A. \label{eq:resolvent2}\end{equation}
Hille-Yosida Theorem (see Proposition 2.1 in \cite{MR838085} for instance) implies that the Laplace transform $L(z)$ is well defined, is smooth on $(0,+\infty)$, 
and such that
\begin{equation}
\overline\kappa(z)=\lambda + L(z)=\lambda + \frac12 \ll j_{0,1}^A, h_z \gg_\star.
\label{eq:homogenizedGK}
\end{equation}
 Since the generator $\cL^\m$ conserves the degree of homogeneous polynomial functions, the solution to the resolvent equation is on the form 
\[h_z(\m,\omega)=\sum_{x \in \Z}\varphi_z(\m,x,x)(\omega_{x+1}^2-\omega_x^2)+\sum_{\substack{x,y\in\Z\\x\neq y}} \varphi_z(\m,x,y)\omega_x\omega_y,\] where, for all $\m \in\ccl{\Omega^{\cD}}$, the function $\varphi_z(\m,\cdot,\cdot):\Z^2\to\R$ is square-summable and symmetric.
%
%
In the Hilbert space $\mathbf{L}^2_\star$, there exist equations involving the symmetric part $\S$ that can be explicitly solved:
\begin{lem}\label{lem:image}
There exists $f \in \mathbf{L}^2_\star$ such that $\S f=j_{0,1}^A$ in $\mathbf{L}^2_\star$.
\end{lem}

\begin{proof}
We look at the solutions $f$ to $\S f=j_{0,1}^A$ on the form 
\[
f(\m,\omega)=\sum_{\substack{x \in \Z\\ k \geq 1}} \varphi_k(\m,x)\omega_x\omega_{x+k},
\] 
such that, for all $\m \in \ccl{\Omega^{\cD}}$, 
\begin{equation}
\sum_{\substack{x \in \Z\\ k \geq 1}} |\varphi_k(\m,x)|^2 < +\infty. 
\label{eq:bound}
\end{equation} 
To simplify notations, let us erase the dependence of $\m$ for a while, and keep it in mind. Then, the sequence $\{\varphi_k(x) ; x \in \Z, k \geq 1\}$ has to be solution to
\begin{equation}
\label{eq:systphi}\begin{cases}
&-\lambda \big(\varphi_{k+1}(x)+\varphi_{k+1}(x-1)\big) + 4(\lambda+\gamma)\varphi_k(x)-\lambda\big(\varphi_{k-1}(x)+\varphi_{k-1}(x-1)\big)=0,\;  k \geq 2,\\
&(\lambda+2\gamma)\varphi_1(x) - \lambda \big(\varphi_2(x)+\varphi_2(x-1)\big)  = \displaystyle \frac{\delta_0(x)}{\sqrt{m_0m_1}},
\end{cases}
\end{equation}
for $x \in \Z$.

We introduce the Fourier transform $\widehat{\varphi}_k$ defined on $\T$ as follows:
\[ \widehat{\varphi}_k(\xi):= \sum_{x \in \Z} \varphi_k(x) e^{2i\pi x\xi}, \qquad \xi \in \T.\]
From \eqref{eq:bound}, this is well defined, and  the inverse Fourier transform together with the Plancherel-Parseval relation imply: 
\[\begin{cases}
&\varphi_k(x) =\displaystyle \int_\T \widehat{\varphi}_k(\xi) e^{-2 i \pi x \xi} \text d\xi \\
 &\displaystyle \sum_{\substack{x \in \Z\\ k \geq 1}} |\varphi_k(x)|^2 = \sum_{k\geqslant 1} \int_\T |\widehat{\varphi}_k(\xi)|^2 \text d\xi.
\end{cases}\]
The system of equations \eqref{eq:systphi} rewrites as 
\begin{equation}
\label{eq:systphihat}\begin{cases}
&-\lambda \big(e^{2i\pi\xi}+1\big)\widehat{\varphi}_{k+1}(\xi) + 4(\lambda+\gamma)\widehat{\varphi}_k(\xi)-\lambda \big(e^{-2i\pi\xi}+1\big)\widehat{\varphi}_{k-1}(\xi)=0, \text{ for } k \geq 2, \xi \in \T,\\
&(\lambda+2\gamma)\widehat{\varphi}_1(\xi) - \lambda  \big(e^{2i\pi\xi}+1\big)\widehat{\varphi}_{2}(\xi)  = \displaystyle \frac{1}{\sqrt{m_0m_1}}, \quad \text{ for } \xi \in \T.
\end{cases}
\end{equation}
Therefore, for any $\xi \in \T$ fixed, $\xi \neq 1/2$, the sequence $\{\widehat{\varphi}_k(\xi)\}_{k \geqslant 1}$ is solution to the second order linear recurrence relation:
\begin{equation}
\widehat{\varphi}_{k+1}(\xi) - \frac{2\alpha e^{-i\pi\xi}}{\cos(\pi\xi)}\widehat{\varphi}_k(\xi)+e^{-2i\pi\xi}\widehat{\varphi}_{k-1}(\xi) = 0, \quad \text{ for } k \geq 2,
\label{eq:recurs}
\end{equation}
where $\alpha:=(\lambda+\gamma)/\lambda$, with the two conditions: 
\[\begin{cases}
& (\lambda+2\gamma)\widehat{\varphi}_1(\xi) - \lambda  \big(e^{2i\pi\xi}+1\big)\widehat{\varphi}_{2}(\xi)  = \theta(\m), \quad \text{ for } \xi \in \T,\\
& \displaystyle \sum_{k\geqslant 1} \int_\T |\widehat{\varphi}_k(\xi)|^2 \text d\xi < + \infty,
\end{cases}
\]
where $\theta(\m):=1/\sqrt{m_0m_1}$. This system is explicitely solvable, and one can easily check that the following function is solution: 
\[ \widehat{\varphi}_k(\xi)=\frac{\theta(\m)}{\gamma r(\xi)+\lambda\big(1+e^{-2i\pi\xi}\big)}  \big(r(\xi)\big)^{k-1},\] 
where
\[r(\xi):=\frac{\alpha e^{-i\pi\xi}}{\cos(\pi\xi)} \left(1-\sqrt{1-\alpha^{-2}\cos^2(\pi\xi)}\right).\]
Note that $r(\cdot)$ is continuous on $\T$ (from a direct Taylor expansion), and $\varphi_k(x)$ can  then be written as the inverse Fourier transform of $\widehat{\varphi}_k(\xi)$.
\end{proof}

We are now able to prove the existence of the Green-Kubo formula:
\begin{prop}
 \label{prop:conv_GK} The following limit
\begin{equation} \overline{D}:= \lim_{\ccl{z \downarrow 0}} \overline\kappa(z) \label{eq:D2}\end{equation} exists, and is finite.
\end{prop}

\begin{proof}
  We investigate the existence of the  limit
\begin{equation} \lim_{\ccl{z \downarrow 0}}  \ll j_{0,1}^A,(z-\cL^\m)^{-1}j_{0,1}^A\gg_{\star}= 2 \lim_{\ccl{z \downarrow 0}}  L(z). \label{eq:theo}\end{equation}
With the notations above, we have to prove that $ \ll h_z,j_{0,1}^A \gg_{\star}$ converges as $z$ goes to 0, and that the limit is finite and non-negative. Then, from \eqref{eq:homogenizedGK} it will follow that $\overline D\geqslant \lambda>0$  and $\overline D$ is positive. We denote by $\Vert \cdot \Vert_{\mathbb{1} \star}$ the semi-norm corresponding to the symmetric part of the generator: \[\Vert f \Vert_{{\mathbb{1} \star}}^2=\ll f,(-\S)f\gg_{\star}\] and $\mathbb{H}_{\mathbb{1} \star}$ is the Hilbert space obtained by the completion of $\cC$ w.r.t.~that semi-norm. The corresponding dual norm is defined as 
\begin{equation}
\label{eq:dual}
\Vert f \Vert_{-{\mathbb{1} \star}}^2:=\sup_{g \in \cC} \big\{ 2 \ll f,g\gg_\star - \Vert g \Vert^2_{\mathbb{1} \star} \big\}.
\end{equation}
We denote by $\mathbb{H}_{-\mathbb{1}\star}$ the Hilbert space obtained by the completion of $ \cC$ w.r.t.~that norm. We already know from the previous sections that $\cQ_0 \subset \mathbb{H}_{-\mathbb{1}\star}$ (and we recover the result of Lemma \ref{lem:image}, namely $j_{0,1}^A \in \mathbb{H}_{-\mathbb{1}\star}$).

We are going to prove the existence of the Green-Kubo formula by using some arguments given in \cite[Section 2.6]{MR2952852}. 
For the reader's convenience, we recall here the main steps, and refer to \cite{MR2952852} for the technical details of the proof. 
First, we take the inner product $\ll \cdot, \cdot \gg_{\star}$ of \eqref{eq:resolvent2} and $h_z$ to obtain
\begin{equation} \label{eq:eq}z \ll h_z,h_z\gg_{\star}+ \Vert h_z\Vert_{\mathbb{1} \star}^2  =\ll h_z,j_{0,1}^A\gg_{\star}.\end{equation}
%
Since $j_{0,1}^A \in \mathbb{H}_{-\mathbb{1}\star}$,  the Cauchy-Schwarz inequality for the scalar product $\ll  \cdot\gg_{\star}$ gives
\begin{equation*}
z \ll h_z,h_z\gg_{\star}+ \Vert h_z\Vert_{\mathbb{1} \star}^2  \leqslant \Vert h_z\Vert_{\mathbb{1}\star} \; \Vert j_{0,1}^A\Vert_{-\mathbb{1}\star}
\end{equation*}
and we obtain that 
 \[\Vert h_z\Vert_{\mathbb{1} \star} \leqslant \Vert j_{0,1}^A\Vert_{-\mathbb{1}\star}.\] 
The family $\{h_z\}_{z>0}$ is therefore bounded in $\mathbb{H}_{\mathbb{1} \star}$, and one can extract a weakly converging subsequence in $\mathbb{H}_{\mathbb{1} \star}$. We continue to denote this subsequence by $\{h_z\}$ and we denote by $h_0$ the limit. We also have 
\[ z \ll h_z, h_z \gg_{\star} \leqslant \Vert j_{0,1}^A\Vert_{-\mathbb{1}\star}^2,\] 
and then $\{zh_z\} 
$ strongly converges to 0 in $\mathbf{L}_\star^2$.
We now invoke the weak sector condition given in Proposition \ref{prop:sector}: there exists $C_0>0$ such that, for any homogeneous polynomials of degree two $f,g \in \mathbb{H}_{\mathbb{1}\star}$,
\begin{equation} \big|\ll f, \cL^\m g \gg_\star\big| \leqslant (C_0+1) \Vert f \Vert_{\mathbb{1} \star} \ \Vert g \Vert_{\mathbb{1} \star}\; . \label{eq:sector2} \end{equation}
Indeed, this is a consequence of \eqref{eq:sector_condition}, since
\[
\big|\ll f, \A^\m g \gg_\star\big|  = \big|\ll \S f, \A^\m g \gg_1\big|  \leqslant C_0 \ \vertiii{\S f}_1 \ \vertiii{\S g}_1 = C_0 \Vert f \Vert_{\mathbb{1} \star} \ \Vert g\Vert_{\mathbb{1} \star},\]
and we also have from the Cauchy-Schwarz inquality,
\[\big|\ll f, \S g \gg_\star\big| \leqslant \Vert f \Vert_{\mathbb{1} \star} \ \Vert g \Vert_{\mathbb{1} \star}\; .
\]
The estimate given in \eqref{eq:sector2} applied to $g=h_z$ yields  
\begin{equation} 
\label{eq:h1} 
\Vert \cL^\m h_z\Vert_{-\mathbb{1} \star} = \sup_{f\in\cC} \big\{ 2 \ll f, \cL^\m h_z \gg_\star - \Vert f \Vert_{\mathbb{1}\star}^2\big\} \leqslant (C_0+1) \Vert h_z \Vert^2_{\mathbb{1} \star} \leqslant (C_0+1) \Vert j_{0,1}^A\Vert^2_{-\mathbb{1}\star}.
\end{equation}
%
From \eqref{eq:resolvent2} and \eqref{eq:h1} we deduce that \[\sup_{z > 0} \Vert zh_z\Vert_{-\mathbb{1} \star} < \infty.\] 
Let us know refer to \cite[Section 2.6, Lemma 2.16]{MR2952852}: the condition \eqref{eq:h1} is sufficient to prove that \begin{itemize}
\item the sequence $\{(-\cL^\m) h_z\}$ weakly converges to $j_{0,1}^A$ in $\mathbb{H}_{-\mathbb{1}\star}$\; ;
\item the following identity holds 
\begin{equation}\label{eq:limit}
\ll h_0, (-\cL^\m)h_0 \gg_\star=\ll h_0, (-\S)h_0 \gg_\star = \ll h_0, j_{0,1}^A\gg_\star\; ;
\end{equation}
\item  the sequence $\{h_z\}$ strongly converges to $h_0$ in $\mathbb{H}_{\mathbb{1} \star}$, and the limit is unique.
\end{itemize}

  We have proved the first part: the limit \eqref{eq:theo} exists. To obtain its finiteness, we are going to give an upper bound, using the following variational formula: \[ \ll j_{0,1}^A, (z-\cL^\m)^{-1} j_{0,1}^A \gg_{\star}\, = \sup_{f\in\cC} \left\{ 2 \ll f, j_{0,1}^A \gg_{\star} - \Vert f \Vert^2_{1,z} - \Vert \A^\m f \Vert^2_{-1,z} \right\},\] where the two  norms $\Vert \cdot \Vert_{\pm 1,z}$ are defined by \[ \Vert f \Vert^2_{\pm 1,z} = \ll f, (z-\S)^{\pm 1} f\gg_{\star}.\] 
For the upper bound, we neglect the term coming from the antisymmetric part $\A^\m f$, which gives 
\[ \ll j_{0,1}^A, (z-\cL^\m)^{-1} j_{0,1}^A \gg_{\star} \, \leqslant\,  \ll j_{0,1}^A, (z-\S)^{-1} j_{0,1}^A \gg_{\star}. \] In the right hand side we can also neglect the part coming from the exchange symmetric part $\S^{\text{exch}}$, and remind that $\S^{\text{flip}}(j_{0,1}^A)=-2j_{0,1}^A$. This gives an explicit finite upper bound.  Then, we have from \eqref{eq:limit} that 
\[\lim_{\ccl{z \downarrow 0}} \ll j_{0,1}^A, (z-\cL^\m)^{-1} j_{0,1}^A \gg_{\star}\, = \ll j_{0,1}^A,h_0\gg_{\star}\, = \ll h_0,(-\S)h_0 \gg_{\star}\,  \geqslant 0,\] and the positiveness is proved.  \end{proof}

\subsection{Equivalence of the definitions}

In this subsection we rigorously prove the equality between the variational formula for the diffusion coefficient and the Green-Kubo formula.

\begin{prop}\label{prop:equivalence}
For every $\lambda >0$ and $\gamma>0$, 
\[ \overline D=\lambda+\frac{1}{2}\lim_{\substack{z\to 0\\z>0}} \ll j_{0,1}^A,(z-\cL^\m)^{-1}j_{0,1}^A\gg_{\star} \]
coincides with the coefficient $\widetilde{D}=D$ defined  in Lemma \ref{lem:diff1}.
\end{prop}

\begin{proof}
\ccl{We obtained in Section \ref{sec:diffusion} several equivalent expressions for the diffusion coefficient}. For instance, since $\chi(1)=2$, \ccl{\eqref{eq:eqDTJs} rewrites}
\[ D=  \frac{2}{\vvvert T(\omega_1^2-\omega_0^2)\vvvert_1^2}.
\]
Furthermore, according to \eqref{eq:idhilb}, there exists a sequence $\{f_\varepsilon\}_{\varepsilon >0}$ of functions in $\cQ$ such that 
\[ g_\varepsilon:=j_{0,1}^\star + D(\omega_1^2-\omega_0^2)+\cL^{\m,\star}f_\varepsilon \] 
satisfies $\vvvert g_\varepsilon \vvvert_1 \to 0$ as $\varepsilon$ goes to 0. Observe that $g_\varepsilon \in \cQ_0 \subset \mathbb{H}_{-\mathbb{1}\star}$ from Proposition \ref{prop:propertiesC0}. By substitution in the equality above, we get
\[D^{-1}=\frac{1}{2D^2} \ll  g_\varepsilon -j_{0,1}^\star - \cL^{\m,\star} f_\varepsilon, T^\star( g_\varepsilon -j_{0,1}^\star -\cL^{\m,\star} f_\varepsilon) \gg_1 \] 
recalling that $\ll Tg,Tg\gg_1=\ll g, T^\star g \gg_1$ for all $g \in \cH_1$. Therefore,
\begin{align*}
D & = \frac{1}{2} \ll g_\varepsilon -j_{0,1}^\star - \cL^{\m,\star} f_\varepsilon ,  g_\varepsilon  -j_{0,1}^S -\S f_\varepsilon  \gg_1 \\
& = \frac{1}{2} \ll j_{0,1}^\star +\cL^{\m,\star} f_\varepsilon, j_{0,1}^S+\S f_\varepsilon \gg_1 + R_\varepsilon
\end{align*}
where $R_\varepsilon$  is bounded by $C \vvvert g_\varepsilon \vvvert_1^2$, and then vanishes as $\varepsilon$ goes to 0. Finally, from Proposition \ref{prop:identities}, we can write 
\[ D=\lambda + \frac{1}{2} \lim_{\varepsilon \to 0} \ll f_\varepsilon, (-\S) f_\varepsilon \gg_{\star}\; ,\]
and we know that the limit above exists, which implies that $\big| \ll f_\varepsilon, (-\S) f_\varepsilon \gg_{\star}\big|$ is uniformly bounded in $\varepsilon$ by a constant $C>0$. The problem is now reduced to prove that 
\begin{equation}
\label{eq:equallimits}
\lim_{\varepsilon \to 0} \ll f_\varepsilon, (-\S) f_\varepsilon \gg_{\star}\, = \lim_{\substack{z\to 0\\z>0}} \ll j_{0,1}^A,(z-\cL^\m)^{-1}j_{0,1}^A\gg_{\star}.
\end{equation}
For every $z>0$ and $\varepsilon >0$, we have by definition above and \eqref{eq:resolvent2}, 
\begin{align}
j_{0,1}^A & = z h_z - \cL^\m h_z \label{eq:limits1}\\
j_{0,1}^\star & =  g_\varepsilon- D(\omega_1^2-\omega_0^2)  - \cL^{\m,\star} f_\varepsilon  . \label{eq:limits2}
\end{align}
First, we take the inner product $\ll \cdot , \cdot \gg_{\star}$ of \eqref{eq:limits2} with $f_\varepsilon$ (recall that $\ll j_{0,1}^S,f_\varepsilon \gg_\star =- \ll j_{0,1}^S , \S f_\varepsilon \gg_1 = 0$), to get 
\[
\ll j_{0,1}^A, f_\varepsilon\gg_{\star}\, = - \ll f_\varepsilon,g_\varepsilon\gg_{\star} - \ll f_\varepsilon,(-\S) f_\varepsilon \gg_{\star} \]
and using \eqref{eq:limits1}, 
\[
-\ll \cL^\m h_z,f_\varepsilon \gg_{\star} + z \ll h_z, f_\varepsilon \gg_{\star}\, = - \ll f_\varepsilon,g_\varepsilon\gg_{\star} - \ll f_\varepsilon,(-\S) f_\varepsilon \gg_{\star}. 
\]
First, let $z$ go to 0, and observe that the limit of $\ll  \cL^\m h_z,f_\varepsilon \gg_{\star}$ exists since $\{\cL^\m h_z\}$ weakly converges in $\mathbb{H}_{-\mathbb{1}\star}$ and $f_\varepsilon \in \mathbb{H}_{\mathbb{1}\star}$. Let us take the limit as $\varepsilon$ goes to 0, and write 
\begin{align*}  \ll f_\varepsilon,g_\varepsilon\gg_{\star} \ \leqslant \ \Vert f_\varepsilon \Vert_{\mathbb{1}\star} \; \Vert g_\varepsilon \Vert_{-\mathbb{1}\star} \leqslant  C \vvvert g_\varepsilon \vvvert_1 \xrightarrow[\varepsilon \to 0]{} 0.\end{align*}
 The first equality is justified by the fact that $g_\varepsilon$ belongs to $\cQ_0 \subset \mathbb{H}_{-\mathbb{1}\star}$, and the last inequality comes from the definition of the semi-norm $\vvvert \cdot \vvvert_1$ given in \eqref{eq:norm1}. As a consequence, we have obtained 
\[\lim_{\varepsilon \to 0} \ll f_\varepsilon, (-\S) f_\varepsilon \gg_{\star}\, = \lim_{\varepsilon \to 0} \lim_{z \to 0} \ll \cL^\m h_z, f_\varepsilon \gg_{\star}.\] 
In the same way,  take the inner product $\ll \cdot, \cdot \gg_{\star}$ of \eqref{eq:limits2} with $h_z$ to obtain
\[ \ll j_{0,1}^A,h_z \gg_{\star}\, =-\ll g_\varepsilon, h_z \gg_{\star} + \ll \cL^{\m,\star} f_{\varepsilon}, h_z \gg_{\star}.\]
If we send first $z$ to 0, then $\ll g_\varepsilon, h_z \gg_{\star} $ converges to $\ll g_\varepsilon, h_0\gg_{\star}$ from the weak convergence of $\{h_z\}$ in $\mathbb{H}_{\mathbb{1}\star}$ and since $g_\varepsilon \in \mathbb{H}_{-\mathbb{1}\star}$. As before, we write \[ \ll g_\varepsilon, h_0\gg_{\star} \leqslant C \vvvert g_\varepsilon \vvvert_1 \xrightarrow[\varepsilon \to 0]{} 0.\]
Therefore,  \begin{align*}
\lim_{z\to 0} \ll j_{0,1}^A, h_z \gg_{\star}  = \lim_{\varepsilon \to 0} \lim_{z \to 0} \ll\cL^{\m,\star} f_\varepsilon, h_z \gg_{\star} =\lim_{\varepsilon \to 0} \ll f_\varepsilon, (-\S) f_\varepsilon \gg_{\star}
\end{align*}
and the claim is proved.
\end{proof}

\section{The anharmonic chain perturbed by a diffusive noise} \label{sec:anharmonic}

In this section we say a few words about the anharmonic chain, when the interactions between atoms are non-linear,  given by a potential $V$. As in \cite{sasolla}, we assume that the function $V: \R \to \R_+$ satisfies the following properties: \begin{enumerate}[(i)]
\item $V(\cdot)$ is a smooth symmetric function,
\item there exist $\delta_-$ and $\delta_+$ such that $0<\delta_-\leqslant V''(\cdot) \leqslant \delta_+ < +\infty$,
\item $\delta_-/\delta_+ > (3/4)^{1/16}$.
\end{enumerate}
Using the same notations as in the introduction, the configuration $\{p_x,r_x\}$ now evolves according to \begin{equation}
\left\{ \begin{aligned} \frac{\text d p_x}{\text d t}&=V'(r_{x+1})-V'(r_{x}), \\
\frac{\text d r_x}{\text d t}&=\frac{p_{x}}{{M_{x}}}-\frac{p_{x-1}}{{M_{x-1}}}. \end{aligned}\right. \label{eq:motionsasada}
\end{equation}
We define $\pi_x:=p_x/\sqrt{M_x}$, and the dynamics on $\{\pi_x,r_x\}$ rewrittes as: \begin{equation}
\left\{ \begin{aligned}  \frac{\text d \pi_x}{\text d t}&=\frac{1}{\sqrt{M_x}}\left[V'(r_{x+1})-V'(r_{x})\right], \\
\frac{\text d r_x}{\text d t}&=\frac{\pi_{x}}{\sqrt{M_{x}}}-\frac{\pi_{x-1}}{\sqrt{M_{x-1}}}. \end{aligned}\right. \label{eq:motionsasada_pi}
\end{equation} The total energy \[ \mathcal{E}:= \sum_{x\in \Z} \bigg\{\frac{\pi_x^2}{2} + V(r_x)\bigg\}\] is conserved. The flip and exchange noises have poor ergodic properties, and can be used for harmonic chains only. For the anharmonic case, we introduce a stronger stochastic perturbation. Now, the total generator of the dynamics writes $\cL^\m=\A^\m+\gamma\S$, where \begin{equation}
\A^\m:=\sum_{x\in \Z} \frac{1}{\sqrt{M_x}}\left(X_x - Y_{x,x+1}\right), \qquad \S:=\frac{1}{2} \sum_{x\in\Z} \big\{X_x^2+Y_{x,x+1}^2\big\}, 
\end{equation}
where \[Y_{x,y}=\pi_x \frac{\partial}{\partial{r_{y}}}- V'(r_{y}) \frac{\partial}{\partial\pi_x}, \qquad X_x=Y_{x,x}.\] For this anharmonic case, the needed ingredients can be proved directly from \cite{sasolla}: first, note that the symmetric part of the generator does not depend on the disorder and is exactly the same as in \cite{sasolla}. Then, the proof of the spectral gap is done in Section 12 of that paper, and the sector condition can also be proved, following Section 8.  More precisely, after taking into account the disorder and its fluctuation, the same argument of \cite[Lemma 8.2, Section 8]{sasolla} can be applied, since it is mainly based on the fact that both antisymmetric and symmetric parts involve the same operators $Y_{x,y}$.

\newpage 

\appendix

\section{Hermite polynomials and quadratic functions\label{sec:hermite}}

In the whole section we assume $\beta=1$. Every result can be restated for the general case after replacing the variable $\omega$ by $\beta^{-1/2}\omega$.


Let $\chi$ be the set of positive integer-valued functions $\xi: \Z \to \N$, such that $\xi_x$ \ccl{vanishes} for all but a finite number of $x \in \Z$. The \emph{length} of $\xi$, denoted by $\vert \xi \vert$, is defined as \[\vert \xi \vert=\sum_{x\in\Z} \xi_x.\] For $\xi \in \chi$, we define the polynomial function $H_\xi$ on $\Omega$ as \[ H_\xi(\omega)=\prod_{x \in \Z} h_{\xi_x}(\omega_x),\] where $\{h_n\}_{n\in\N}$ are the normalized Hermite polynomials w.r.t.~the one-dimensional standard Gaussian probability law (with density $(2\pi)^{-1/2}\exp(-x^2/2)$ on $\R$). The sequence $\{H_\xi\}_{\xi\in\chi}$ forms an orthonormal basis of the Hilbert space $\mathbf L^2(\mu_1)$, where $\mu_1$ is the infinite product Gibbs measure on $\R^{\Z}$, defined by \eqref{eq:mesinva} with $\beta=1$. As a result, every function $f \in \mathbf L^2(\mu_1)$ can be decomposed in the form \[ f(\omega)=\sum_{\xi \in \chi} F(\xi) H_\xi(\omega).\]
Moreover, we can compute the scalar product $\ps{f,g}_1$ for $f=\sum_\xi F(\xi) H_\xi$ and $g=\sum_\xi G(\xi) H_\xi$ as \[ \ps{f,g}_1=\sum_{\xi \in \chi}F(\xi)G(\xi).\] 

\bde{\label{def:degre}\ccl{Define $\chi_n :=\left\{\xi\in\chi \ ; \ \vert \xi \vert=n\right\},$ a} function $f\in\mathbf{L}^2(\mu_1)$ is \emph{of degree} $n$ if its decomposition \[f=\sum_{\xi\in\chi} F(\xi) H_\xi\] satisfies: $F(\xi)=0$ for all $\xi \notin \chi_n$.}
\begin{rem} \label{rem:quadr} In this paper, we mainly focus on degree 2 functions, which are by Definition \ref{def:degre} of the form 
\begin{equation} \sum_{x \in \Z} \varphi(x,x) (\omega_x^2-1) + \sum_{x\neq y} \varphi(x,y) \omega_x\omega_y \label{eq:degree2}\end{equation}
where $\varphi:\Z^2 \to \R$ is a square summable symmetric function. Note that they all have zero mean w.r.t.~$\mu_1$, and they can also be rewritten as
\[\sum_{x\in\Z} \psi(x,x)(\omega_x^2-\omega_{x+1}^2) + \sum_{x\neq y}\psi(x,y)\omega_x\omega_y,\]
for some square summable symmetric function $\psi:\Z^2\to\R$.
\end{rem}
\subsection{Local functions} On the set of $n$-tuples ${\bf x}:=(x_1, \ldots,x_n)$ of $\Z^n$, we introduce the equivalence relation ${\bf x} \sim {\bf y}$ if there exists a permutation $\sigma$ on $\{1, \ldots,n\}$ such that $x_{\sigma(i)} =y_i$ for all $i \in \{1, \ldots,n\}$. The class of ${\bf x}$ for the relation $\sim$ is denoted by $[{\bf x}]$ and its cardinal by $c({\bf x})$.  Then the set of configurations of $\chi_n$ can be identified with the set of $n$-tuples classes for $\sim$ by the one-to-one application:
\begin{align*}
\Z^n/\sim \quad & \to \, \chi_n\\
[{\bf x}]=[(x_1,\ldots,x_n)]  &\mapsto \quad \xi^{[{\bf x}]}  
\end{align*}
where for any $y \in \Z$, $(\xi^{[{\bf x}]})_y= \sum_{i=1}^n {\bf 1}_{y=x_i}$. We identify $\xi \in \chi_n$ with the occupation numbers of  a configuration with $n$ particles, and  $[\bf x]$ corresponds to  the positions of those $n$ particles. A function  $F: \chi_n \to \R$ is nothing but a symmetric function $F:\Z^n \to \R$ through the identification of $\xi$ with $[{\bf x}]$. We denote (with some abuse of notations) by $\langle \cdot, \cdot \rangle$ the scalar product on $\oplus {\mathbf L}^2 (\chi_n)$, each $\chi_n$ being equipped with the counting measure. Hence, for two functions $F,G:\chi \to \R$, we have
\begin{equation*}
\langle F, G \rangle = \sum_{n\ge 0} \sum_{\xi \in \chi_n} F_n (\xi) G_n (\xi) = \sum_{n \ge 0} \sum_{{\bf x} \in \Z^n} \frac{1}{c({\bf x})} \,  F_n ({\bf x}) G_n ({\bf x}),
\end{equation*}
with $F_n, G_n$ the restrictions of $F,G$ to $\chi_n$.

\subsection{Dirichlet form} It is not hard to check the following proposition, which is a direct consequence of the fact that $h_n$ has the same parity of the integer $n$.

\bprop{If a local function $f \in \mathbf L^2(\mu_1)$ is written on the form $f=\sum_{\xi\in\chi} F(\xi)H_\xi,$ then \[\S f(\omega)= \sum_{\xi \in \chi} (\mathfrak{S} F)(\xi) H_\xi(\omega),\] where $\mathfrak{S}$ is the operator acting on functions $F:\chi\to\R$ as \[\mathfrak SF(\xi)=\lambda\sum_{x\in\Z}\big[F(\xi^{x,x+1})-F(\xi)\big] + \gamma \sum_{x\in\Z} \big((-1)^{\xi_x}-1\big)F(\xi).\] Above $\xi^{x,y}$ is obtained from $\xi$ by exchanging $\xi_x$ and $\xi_y$.}
From this result we deduce:

 \bcor{\label{cor:dirichlet}For any $f=\sum_{\xi\in\chi} F(\xi)H_\xi \in \mathbf L^2(\mu_1)$, we have \[\cD(\mu_1;f)=\bps{f,-\S f}_1=\sum_{\xi\in\chi}\bigg\{ \frac{\lambda}{2}\sum_{x\in\Z} \big(F(\xi^{x,x+1})-F(\xi)\big)^2+\gamma\sum_{x\in\Z} \big((-1)^{\xi_x}-1\big)F^2(\xi)\bigg\}\]}

 \subsection{Quadratic functions} \label{sub:quad-herm}

Recall Definition \ref{de:quadratic}. In other words, we are mostly interested in \emph{quadratic} functions $f$ in $\mathbf{L}^2(\mu_1)$, which have  zero average with respect to $\mu_1$ and compact support. They correspond exactly to \emph{degree 2 functions} as we already noticed in Remark \ref{rem:quadr}, but with the additional assumption that their support is compact. 

The next propositions give some useful properties:
 
 \bprop{\label{prop:variational_restriction}If $f \in \mathbf L^2(\mu_1)$ is of degree 2, then the following variational formula \[\sup_{g \in \mathbf L^2(\mu_1)} \left\{ 2\bps{f,g}_1 - \cD( \mu_1 ; g)\right\}\] can be restricted over degree 2 functions $g$.
 
Moreover, if the support of $f$ is finite and included in $\Lambda$, then the supremum can be restricted to functions with support included in $\Lambda$. 
 }
 
 \bprf{This fact follows after decomposing $g$ as $\sum_{\xi \in \chi} G(\xi)H_\xi$.  Corollary \ref{cor:dirichlet} and the orthogonality of Hermite polynomials imply that we can restrict the supremum over functions $g$ of degree two \eqref{eq:degree2}. 
%
Moreover, if $x\neq y$, then $\langle (\omega_x^2-1)(\omega_y^2-1) \rangle_1 = 0$, and if $x,y,z,t$ are all distinct, then $\langle \omega_x\omega_y\omega_z\omega_t \rangle_1=0$. This implies that the support of $g$ can be restricted to the one of $f$, otherwise it would only increase the Dirichlet form.
  }
 
 \bprop{\label{prop:convergence}Let $\{f_n\}_{n\in\N}$ be a sequence of degree 2 functions in $\mathbf L^2(\mu_1)$. Suppose that $\{f_n\}$ weakly converges to $f$ in $ \mathbf L^2(\mu_1)$. Then, $f$ is of degree 2.
 
 Moreover, if any $f_n$ has support included in some finite subset $\Lambda$, then the support of $f$ is included in $\Lambda$.}
\bprf{For all $n \in \N$, and $\xi \notin \chi_2$, the scalar product $\bps{f_n,H_\xi}_1$ vanishes (by definition). From the weak convergence, we know that \[\bps{f_n,H_\xi}_1 \to \bps{f,H_\xi}_1,\] as $n$ goes to infinity, for all $\xi \in \chi$. This implies: $\bps{f,H_\xi}_1 =0$ for all $\xi \notin \chi_2$. 
%
 }
 
 Note that the set denoted by $\cQ$ and defined in Definition \ref{de:quadratic} contains cylinder quadratic functions in $\mathbf{L}^2(\P_1^\star)$. The conclusions of Propositions \ref{prop:variational_restriction} and \ref{prop:convergence} can be restated for our purpose as: 
 
 \begin{cor}
If $f \in \cQ$, then the following variational formula \[\sup_{g \in \mathbf L^2(\P_1^\star)} \left\{ 2 \E_1^\star[f,g] - \cD( \P^\star_1 ; g)\right\}\] can be restricted over functions $g$ in $\cQ$.
Moreover, if $\{f_n\}_n$ is a sequence of  functions in $\cQ$ such that $\{f_n\}$ weakly converges to $f$ in $\mathbf L^2(\P^\star_1)$, then $f$ belongs to $\cQ$.
 \end{cor}

\section{Proof of the weak sector condition \label{app:sector}}

In this section we prove Proposition \ref{prop:sector} that we recall here for the sake of clarity.

\bprop{[Weak Sector condition]
\label{prop:seccond}\begin{enumerate}[(i)] \item \ccl{There exist} two constants $C_0(\gamma,\lambda)$ and $C_1(\gamma,\lambda)$ 
such that the following \ccl{inequalities} hold for all $f,g \in \Q$: 
\begin{align*}\left| \ll \A^\m f, \S g \gg_\beta\right| &\leqslant C_0 \vertiii{ \S f }_\beta\, \vertiii{ \S g }_\beta. \\
\left| \ll \A^\m f, \S g\gg_\beta\right| &\leqslant C_1 \vertiii{ \S f }_\beta^2+ \frac{1}{2}\vertiii{ \S g}_\beta^2.
 \end{align*}
\label{sector_item11}
\item There exists a positive constant $C(\beta)$ such that, for all $g \in \Q$, \begin{equation*} \vertiii{ \A^\m g }_\beta\ \leqslant C(\beta)\vertiii{\S g }_\beta.\end{equation*} \label{sector_item22}\end{enumerate} }

\bprf{We prove \eqref{sector_item11}. We assume that \begin{align*}
g(\m,\omega)&=\sum_{x\in\Z}\psi_{x,0}(\m)(\omega_{x+1}^2-\omega_x^2) +\sum_{\substack{x \in \Z\\k \geqslant 1}} \psi_{x,k}(\m)\omega_x\omega_{x+k} \\
f(\m,\omega)&=\sum_{x \in \Z} \varphi_{x,0}(\m) (\omega_{x+1}^2-\omega_x^2) + \sum_{\substack{x \in \Z\\k \geqslant 1}} \varphi_{x,k}(\m) \omega_x \omega_{x+k}.
\end{align*}
We denote by $\Delta^\m \psi$ the discrete Laplacian in the variable $\m$, that is \begin{equation*}
\Delta^\m \psi(\m)= 2 \psi(\m)-\psi(\tau_1 \m)-\psi(\tau_{-1}\m),
\end{equation*} and $\tau_x \Delta^\m$ is the operator defined as \begin{equation*} (\tau_x\Delta^\m) \psi(\m):=\Delta^\m \psi(\tau_x \m).\end{equation*}
Straightforward computations show that \begin{align}
\vertiii{ \S g }^2_\beta&=\frac{\gamma}{2}\E_\beta^\star \left[ \big(\nabla_0 \Gamma_g\big)^2\right] + \frac{\lambda}{2}\E_\beta^\star \left[\big(\nabla_{0,1} \Gamma_g\big)^2\right] \notag\\
&= \frac{4\gamma}{\beta^2} \sum_{\substack{x \in \Z\\ k\geqslant 1}} \E[\psi_{x,k}^2] + \frac{2\lambda}{\beta^2} \sum_{x\in\Z}  \E\bigg[\bigg(\sum_{x \in \Z} \tau_x \big(\Delta^\m \psi_{x,0}\big) \bigg)^2\bigg]   \notag \\
&  \qquad +\frac{\lambda}{\beta^2}\sum_{k \geqslant 2}\E\bigg[ \bigg(\sum_{x\in\Z} \big[\tau_{-x} (\psi_{x,k}) -\tau_{1-x}(\psi_{x,k})\big]\bigg)^2\bigg],\notag \\
~ \notag\\
\vertiii{\S f }^2_\beta &\geqslant \vertiii{ \S^{\text{flip}} f }^2_\beta = \frac{\gamma}{2} \E_\beta^\star\bigg[ \bigg(2\sum_{\substack{z \in \Z\\k\geqslant 1}} \varphi_{z,k}(\m) \omega_0\omega_k \bigg)^2\bigg]= \frac{2\gamma}{\beta^2} \sum_{k\geqslant 1} \E\bigg[  \bigg(\sum_{z\in\Z} \varphi_{z,k}(\m) \bigg)^2\bigg].\label{eq:straight}
\end{align}
Now we deal with  $\ll \A^\m g,\S f \gg_\beta$. From Proposition \ref{prop:identities}, and by definition, 
\begin{align*}
\ll \A^\m g, \S f\gg_\beta =&-\sum_{z\in\Z} \E_\beta^\star \left[f, \tau_z(\A^\m g)\right] \\
=& -\sum_{x,z \in \Z} \E\left[ \varphi_{x,0}(\m) \langle \omega_{x+1}^2-\omega_x^2, \tau_z(\A^\m g)\rangle_\beta \right]\\
& - \sum_{\substack{x,z\in\Z\\k\geqslant 1}}\E\left[\varphi_{x,k}(\m) \langle \omega_x\omega_{x+k}, \tau_z(\A^\m g)\rangle_\beta\right] \\
 = &\frac{2}{\beta^2} \sum_{x\in\Z}\E\bigg[ \frac{\tau_x(\Delta^\m\psi_{x,0})}{\sqrt{m_xm_{x+1}}} \sum_{z\in\Z}\tau_{-z}(\varphi_{z,1})\bigg]  \\
 &+ \frac{1}{\beta^2} \sum_{x\in\Z}\E\bigg[\bigg(\frac{\tau_1\psi_{x,1}}{\sqrt{m_xm_{x+1}}}-\frac{\psi_{x,1}}{\sqrt{m_{x+1}m_{x+2}}}\bigg)\sum_{z\in\Z} \tau_{-z}(\varphi_{z,2}) \bigg]\\
  &+ \frac{1}{\beta^2}\sum_{k\geqslant 2} \sum_{x\in\Z}\E\bigg[\bigg(\frac{\tau_1\psi_{x,k}}{\sqrt{m_xm_{x+1}}}-\frac{\psi_{x,k}}{\sqrt{m_{x+k}m_{x+k+1}}}\bigg)\sum_{z\in\Z} \tau_{-z}(\varphi_{z,k+1}) \bigg] \\
 &+ \frac{1}{\beta^2} \sum_{k\geqslant 2}\sum_{x\in\Z}\E\bigg[\bigg(\frac{\tau_{-1}\psi_{x,k}}{\sqrt{m_xm_{x+1}}}-\frac{\psi_{x,k}}{\sqrt{m_{x+k}m_{x+k-1}}}\bigg)\sum_{z\in\Z} \tau_{-z}(\varphi_{z,k-1}) \bigg].
\end{align*}
From Cauchy-Schwarz inequality, and recalling $1/\sqrt{m_0m_1} \leqslant C$ ($\P$-a.s.), we obtain the following bound: 
\begin{align}
\Vert \ll \A^\m g, \S f\gg_\beta\Vert & \leqslant \frac{2C}{\beta^2}{\E\bigg[\bigg(\sum_{x\in\Z} \tau_x(\Delta^\m \psi_{x,0}) \bigg)^2\bigg]^{1/2}\; \E\bigg[\bigg(\sum_{z\in\Z} \tau_{-z}\varphi_{z,1}\bigg)^2\bigg]^{1/2}} \label{eq:term1} \\
& \qquad + \frac{3C}{\beta^2} {\E\bigg[\bigg(\sum_{x\in\Z}\tau_1 \psi_{x,1}-\psi_{x,1}\bigg)^2\bigg]^{1/2}\;\E\bigg[\bigg(\sum_{z\in\Z} \tau_{-z}\varphi_{z,2}\bigg)^2\bigg]^{1/2}} \label{eq:term2}\\
& \qquad + \frac{3C}{\beta^2} \sum_{k\geqslant 2} \E \bigg[ \bigg(\sum_{x\in\Z}  \tau_1 \psi_{x,k} - \psi_{x,k}\bigg)^2  \bigg]^{1/2}\; \E\bigg[\bigg(\sum_{z\in\Z} \tau_{-z}\varphi_{z,k+1}\bigg)^2\bigg]^{1/2} \label{eq:term3} \\
& \qquad + \frac{3C}{\beta^2} \sum_{k\geqslant 2} \E \bigg[ \bigg(\sum_{x\in\Z}  \tau_{-1} \psi_{x,k} - \psi_{x,k}\bigg)^2  \bigg]^{1/2} \;\E\bigg[\bigg(\sum_{z\in\Z} \tau_{-z}\varphi_{z,k-1}\bigg)^2\bigg]^{1/2} . \label{eq:term4} 
\end{align}
Now we are going to use twice the trivial inequality $\sqrt{ab}\leqslant a/\varepsilon+\varepsilon b/2$ for a particular choice of $\varepsilon>0$: in \eqref{eq:term1} we take $\varepsilon=\gamma/C$ and in \eqref{eq:term2} we take $\varepsilon=2\gamma/(3C)$. This trick gives the final bound
\begin{align*}
\Vert \ll \A^\m g, \S f\gg_\beta\Vert & \leqslant \frac{2C^2}{\gamma\beta^2}\E\bigg[\bigg(\sum_{x \in \Z} \tau_x \big(\Delta^\m \psi_{x,0}\big) \bigg)^2\bigg] +\frac{2\gamma}{\beta^2}\sum_{k\geqslant 1}\E\bigg[\bigg(\sum_{z\in\Z} \varphi_{z,k}(\tau_{-z}\m)\bigg)^2\bigg]\\
& \qquad + \frac{9C^2}{\gamma\beta^2} \sum_{k \geqslant 2}\E\bigg[\bigg(\sum_{x\in\Z} \tau_1\psi_{x,k}-\psi_{x,k}\bigg)^2\bigg]. 
\end{align*}
Recalling \eqref{eq:straight}, we obtain \begin{equation*}
\Vert \ll \A^\m g, \S f\gg_\beta\Vert \leqslant \frac{9C^2}{\gamma\lambda} \vertiii{ \S g }^2_\beta+ \frac{1}{2}\vertiii{ \S f }_\beta^2.
\end{equation*}
If we use the Cauchy-Schwarz inequality, we get:
\begin{equation*}
\ll \A^\m g,\S f\gg_\beta^2\ \leqslant \frac{18C^2}{\gamma\lambda} \vertiii{\S g}^2_\beta\  \vertiii{ \S f}^2_\beta.
\end{equation*} 
We have proved \eqref{sector_item11} with $C_0=\sqrt{18C^2/(\gamma\lambda)}$ and $C_1=9C^2/(\gamma\lambda)$. Now we turn to \eqref{sector_item22}. 
From Lemma \ref{lem:AandS} and Statement \eqref{sector_item11}, \begin{equation*}
\ll \A^\m g,j_{0,1}^S \gg_\beta\ =\ \ll \S g,j_{0,1}^A \gg_\beta\ \leqslant  \vertiii{ \S g }_\beta \vertiii{ j_{0,1}^A }_\beta.
\end{equation*} 
Moreover, from Statement \eqref{sector_item11}, we also get, for all $f \in \Q_0$, \begin{equation*} 
-2\ll \A^\m g, \S f \gg_\beta \ \leqslant \ \vertiii{ \S f }^2_\beta + \frac{2C}{\gamma\lambda}\vertiii{ \S g }^2_\beta.
\end{equation*} 
As a result, the variational formula \eqref{eq:variational_new} for $\vertiii{ \A^\m g}^2_\beta$ gives: 
\begin{equation*}
\vertiii{\A^\m g}^2_\beta\ \leqslant \ \frac{1}{\lambda\chi(\beta)}\ll \A^\m g, j_{0,1}^S \gg_\beta^2 + \frac{9C^2}{\gamma\lambda}\vertiii{ \S g }^2_\beta \leqslant \Bigg(\frac{ \vertiii{ j_{0,1}^A }^2_\beta}{\lambda \chi(\beta)} +\frac{9C^2}{\gamma\lambda}\Bigg) \vertiii{ \S g }^2_\beta.
\end{equation*}
The result is proved. }

\section{Tightness \label{sec:tightness}}

In this section we prove the tightness of the sequence $\{\mathfrak{Y}^N\}_{N\geqslant 1}$, by using standard arguments, \ccl{which can be found for instance in \cite{MR1707314} and in several recent related works}. 
First, let us recall  \ccl{(cf. Section \ref{sec:fluctuation})}  that $k > 5/2$ and that the space $\mathfrak H_{-k}$ is equipped with the norm defined as \[
\Vert\Y \Vert^2_{-k}=\sum_{n\geqslant 1} (\pi n)^{-2k} \big| \Y(\mathbf{e}_n) \big|^2.
\]

\bprop{\label{prop:tight} The sequence  $\{\mathfrak Y^N\}_{N\geqslant 1}$ is tight in $\cC([0,T],\mathfrak H_{-k})$.} 

\bprf{The tightness of the sequence $\{\mathfrak Y^N\}$ follows from two conditions (see \cite[page 299]{MR1707314}): \begin{align}
\lim_{A \to \infty} \limsup_{N\to \infty} \ccl{\mathbb Q_{\mu_\beta^N}}\left[ \sup_{0 \leqslant t \leqslant T} \Vert \mathcal{Y}_{t,\m}^N \Vert_{-k} > A \right] &=0 \vphantom{\Bigg\{}  \label{eq:tight1}\\
\lim_{\delta \to 0} \limsup_{N\to\infty}\ccl{\mathbb Q_{\mu_\beta^N}}\Big[ w(\Y^N_\m,\delta) > \varepsilon \Big] & = 0, \qquad \text{ for all } \varepsilon >0, \vphantom{\Bigg\{} \label{eq:tight2}
\end{align} where the modulus of continuity $w(\Y,\delta)$ is defined by 
\[w(\Y,\delta)=\sup_{\substack{\Vert t-s \Vert < \delta \\ 0 \leqslant s \leqslant t \leqslant T}} \Vert \Y_t - \Y_s \Vert_{-k}.\]
\ccl{Let us remind the decomposition of $\Y_{t,\m}^N$ given in \eqref{eq:dec}: for any $H \in C^2(\T)$, \begin{equation}\label{eq:decbis}
\Y_{t,\m}^N(H)= \Y_{0,\m}^N(H)+\int_0^t \sqrt N \sum_{x\in\T_N} \nabla_N H\Big(\frac x N\Big) j_{x,x+1}(\mathbf{m},s)\text ds + \mathcal{M}^{N}_{t,\m}(H), \end{equation}
where $\mathcal{M}^{N}_{t,\m}(H)$ is the martingale defined in \eqref{eq:defmart}. We first note that there is a constant $\bar C(H)>0$ such that, for $N$ large enough, 
\[ 
\frac{1}{N} \sum_{x\in\T_N}\left[ H^2\Big(\frac x N\Big) + (\nabla_N H)^2\Big(\frac x N\Big)\right] \leqslant \bar C(H).
\] 

\paragraph{{Proof of \eqref{eq:tight1}}.} 
Thanks to \eqref{eq:decbis} the bound
\begin{equation} 
\label{eq:L2Yt}
\mathbb Q_{\mu_\beta^N}\left[ \sup_{0 \leqslant t \leqslant T} \big(\mathcal{Y}_{t,\m}^N(H)\big)^2 \right]\leq C(\beta) \bar C(H) T
\end{equation}
follows from the following three estimates: 
\begin{enumerate}[(i)]
\item first, by assumption, $\mathbb Q_{\mu_\beta^N} \big[\big(\mathcal{Y}_{0,\m}^N(H)\big)^2 \big] \leqslant C_1(\beta) \bar C(H)$ for some constant $C_1(\beta)>0$\; ; 
\item second, by Doob's inequality and an elementary computation,  \[\mathbb Q_{\mu_\beta^N} \Big[\sup_{0 \leqslant t \leqslant T}\big(\mathcal{M}_{t,\m}^N(H)\big)^2 \Big] \leqslant  T C_2(\beta)  \bar C(H), \] for some constant $ C_2(\beta)>0$ ;
\item third, we use Proposition \ref{prop:relation_variance}, to obtain
\begin{multline}
\limsup_{N\to\infty}\mathbb Q_{\mu_\beta^N} \Bigg[ \sup_{0 \leqslant t \leqslant T}\bigg( \int_0^t  \sqrt N \sum_{x\in\T_N} \nabla_N H\Big(\frac x N\Big) j_{x,x+1}(\mathbf{m},s)\text ds \bigg)^2 \Bigg] \\\leqslant CT \vertiii{ j_{0,1} }_{\beta}^2 \int_\T (\nabla H) ^2(u) {\rm d}u\leq T C_3(\beta)\bar C (H). \label{eq:lefthand}
\end{multline}
\end{enumerate}
This proves \eqref{eq:L2Yt}, we now show that \eqref{eq:tight1} follows:
\begin{align*}
\limsup_{N \to\infty} \mathbb{Q}_{\mu_\beta^N} \bigg[ \sup_{0 \leqslant t \leqslant T} \big\| \mathcal{Y}_{t,\m}^N \big\|_{-k}^2 \bigg] & \leqslant \sum_{n \geqslant 1} (\pi n)^{-2k} \limsup_{N \to\infty} \mathbb{Q}_{\mu_\beta^N} \bigg[ \sup_{0 \leqslant t \leqslant T} \big|  \mathcal{Y}_{t,\m}^N(e_n) \big|^2 \bigg] \\
& \leqslant C(\beta) \sum_{n\geqslant 1} (\pi n)^{-2k}(1+(\pi n)^2),
\end{align*} 
where we used that we can choose $\bar{C}(e_n)\leq C (1+(\pi n)^2)$. The result then follows from Chebychev's inequality and the fact that $k > 5/2$.

\paragraph{Proof of  \eqref{eq:tight2}.} As before, we use the decomposition \eqref{eq:decbis} and we prove that the same estimate \eqref{eq:tight2} holds for each term of the decomposition. The most demanding one is the martingale term: we need to show that, for any $n \geqslant 1$ and any $\varepsilon >0$ the following limit holds: 
\begin{equation}
\label{eq:mart-est1}
\lim_{\delta \to 0} \limsup_{N\to\infty} \mathbb{Q}_{\mu_\beta^N} \bigg[ \sup_{\substack{|s-t| < \delta \\ 0 \leqslant s,t \leqslant T}} \big| \mathcal{M}_{t,\m}^N(e_n) - \mathcal{M}_{s,\m}^N(e_n) \big| > \varepsilon \bigg] = 0,
\end{equation}
which is a consequence of the fact that 
\begin{equation}
\label{eq:mart-est2}
\lim_{\delta \to 0} \lim_{N\to\infty}  \mathbb{Q}_{\mu_\beta^N} \Big[ \tilde{w} \big(  \mathcal{M}_{\cdot,\m}^N(e_n),\delta\big) > \varepsilon \Big] = 0, 
\end{equation}
where $\tilde{w}$ is the \emph{modified} modulus of continuity defined by
\[ 
\tilde{w} \big(  \mathcal{M}_{\cdot,\m}^N(e_n),\delta\big) := \inf_{\{t_i\}} \max_{0\leqslant i \leqslant r} \sup_{t_i \leqslant s < t \leqslant t_{i+1}} \big|  \mathcal{M}_{t,\m}^N(e_n) -  \mathcal{M}_{s,\m}^N(e_n) \big|,
\]
with the infimum being over all partitions of $[0,T]$ such that $0=t_0 < t_1 < \dots < t_r = T$ with $t_{i+1}-t_i > \delta$. Using Aldous' criterion \cite[Proposition 4.1.6]{MR1707314} we are reduced to show that for any $n$, any $\varepsilon >0,$
\begin{equation} \lim_{\delta \to 0} \limsup_{N \to\infty} \sup_{\substack{\tau \in \mathcal{T}_T \\ 0\leqslant \sigma \leqslant \delta}}   \mathbb{Q}_{\mu_\beta^N} \Big[ \big| \mathcal{M}_{\tau + \sigma,\m}^N(e_n) -  \mathcal{M}_{\tau,\m}^N(e_n) \big| > \varepsilon \Big]=0, \label{eq:mart-est3}\end{equation} where $\mathcal{T}_T$ is the set of all stopping times bounded by $T$. The probability inside \eqref{eq:mart-est3} is bounded by 
\[ \frac{1}{\varepsilon^2}    \mathbb{Q}_{\mu_\beta^N} \Big[ \big( \mathcal{M}_{\tau + \sigma,\m}^N(e_n) -  \mathcal{M}_{\tau,\m}^N(e_n) \big)^2 \Big],\]
so that
 \begin{multline*}
\mathbb{Q}_{\mu_\beta^N} \Big[ \big( \mathcal{M}_{\tau + \sigma,\m}^N(e_n)  -   \mathcal{M}_{\tau,\m}^N(e_n) \big)^2 \Big] \\
 = \lambda \;  \mathbb{Q}_{\mu_\beta^N} \bigg[  \int_{\tau}^{\tau +\sigma} \frac{1}{N} \sum_{x\in\mathbb{T}_N} \Big(\nabla_N e_n\Big(\frac x N \Big)\Big)^2 \big(\omega_x^2 - \omega_{x+1}^2\big)^2 (s)  \mathrm{d}s\bigg]
  \leqslant C(\beta)\lambda \sigma,
\end{multline*}
which proves \eqref{eq:mart-est1}.

In the next step, in order to treat the integral term which appears in \eqref{eq:decbis}, we need to use the \emph{fluctuation-dissipation} relation and  the Boltzmann-Gibbs principle given in Proposition \ref{prop:macro2}.
More precisely, we need to estimate 
\[ 
\mathbb{Q}_{\mu_\beta^N} \bigg[  
\sup_{\substack{|s-t| < \delta \\ 0 \leqslant s,t \leqslant T}} \bigg( \int_s^t \sqrt N \sum_{x\in\T_N} \nabla_N H\Big(\frac x N \Big) j_{x,x+1}(\m,u) \mathrm{d}u \bigg)^2
\bigg] \]
and to do this we sum and subtract $D(\omega_{x+1}^2 - \omega_x^2)(u) + \mathcal{L}^\m(\tau_x f_k)(\m,u)$ to $j_{x,x+1}(\m,u)$ inside the sum. From the elementary inequality $(a+b+c)^2 \leqslant 3a^2 + 3b^2+3c^2$, we are thus left to estimate \begin{enumerate}[(i)]
\item first, 
\begin{multline*}
 \mathbb{Q}_{\mu_\beta^N} \bigg[  
\sup_{\substack{|s-t| < \delta \\ 0 \leqslant s,t \leqslant T}} \bigg( \int_s^t \sqrt N \sum_{x\in\T_N} \nabla_N H\Big(\frac x N \Big) \Big[j_{x,x+1}(\m,u) \\
+ D\big(\omega_{x+1}^2-\omega_x^2\big)(u) + \mathcal{L}^\m(\tau_x f_k)(\m, u)\Big] \mathrm{d}u \bigg)^2
\bigg]
\end{multline*}
which is exactly 
\[ \mathbb{Q}_{\mu_\beta^N} \bigg[  
\sup_{\substack{|s-t| < \delta \\ 0 \leqslant s,t \leqslant T}} \big( \mathfrak{I}_{t,\m,f_k}^{1,N}(H) -  \mathfrak{I}_{s,\m,f_k}^{1,N}(H)\big)^2 \bigg],\]
where $\mathfrak{I}_{t,\m,f}^{1,N}(H)$ was defined in \eqref{eq:DefI1}.
From the Boltzmann-Gibbs principle (see \eqref{eq:BG_principle}) this quantity  vanishes as $N\to\infty$ and then $k \to\infty$;
\item second, 
\[
 D^2 \; \mathbb{Q}_{\mu_\beta^N} \bigg[  
\sup_{\substack{|s-t| < \delta \\ 0 \leqslant s,t \leqslant T}} \bigg( \int_s^t \sqrt N \sum_{x\in\T_N} \nabla_N H\Big(\frac x N \Big) \big(\omega_{x+1}^2-\omega_x^2\big)(u) \mathrm{d}u \bigg)^2
\bigg]
\] which, after recentering $\omega_{x+1}^2$ and $\omega_x^2$ around their average value $\beta^{-1}$, after   an integration by parts and after applying the Cauchy-Schwartz inequality, can be bounded by 
\[ D^2\; \delta T\; \mathbb{E}_{\beta}^\star \bigg[ \bigg( \frac{1}{\sqrt N} \sum_{x\in\T_N} \Delta_N H \Big(\frac x N \Big) \big(\omega_x^2 - \beta^{-1}\big) \bigg)^2 \bigg] \leqslant   D^2 \; \delta T \; \tilde C(\beta) \bar C(H), \] where $\bar C(H)$ has been previously defined and $\tilde C(\beta)$ is some positive constant;
\item  third, 
\[ \mathbb{Q}_{\mu_\beta^N} \bigg[  
\sup_{\substack{|s-t| < \delta \\ 0 \leqslant s,t \leqslant T}} \bigg( \int_s^t \sqrt N \sum_{x\in\T_N} \nabla_N H\Big(\frac x N \Big)  \mathcal{L}^\m(\tau_x f_k)(\m, u) \mathrm{d}u \bigg)^2.
\bigg] \] which is exactly \[  \mathbb{Q}_{\mu_\beta^N} \bigg[  
\sup_{\substack{|s-t| < \delta \\ 0 \leqslant s,t \leqslant T}} \Big( \mathfrak{I}_{t,\m,f_k}^{2,N}(H) -  \mathfrak{I}_{s,\m,f_k}^{2,N}(H) \Big)^2\bigg].\]
where $\mathfrak{I}_{t,\m,f}^{2,N}(H)$ was defined in \eqref{eq:DefI2}.
We add and subtract $\mathfrak{M}_{t,\m,f_k}^{2,N}(H) - \mathfrak{M}_{s,\m,f_k}^{2,N}(H)$ (defined in \eqref{eq:DefM2}) inside the parenthesis, and using $(a+b)^2 \leqslant 2a^2 + 2b^2$ we are reduced to bound 
\begin{equation}
\label{tight1fin}
\mathbb{Q}_{\mu_\beta^N} \bigg[  
\sup_{\substack{|s-t| < \delta \\ 0 \leqslant s,t \leqslant T}} \Big( \mathfrak{I}_{t,\m,f_k}^{2,N}(H) + \mathfrak{M}_{t,\m,f_k}^{2,N}(H) -  \big(\mathfrak{I}_{s,\m,f_k}^{2,N}(H) + \mathfrak{M}_{s,\m,f_k}^{2,N}(H)\big) \Big)^2\bigg], 
\end{equation}
as well as 
\begin{equation}
\label{tight2fin}
\mathbb{Q}_{\mu_\beta^N} \bigg[  
\sup_{\substack{|s-t| < \delta \\ 0 \leqslant s,t \leqslant T}} \Big(  \mathfrak{M}_{t,\m,f_k}^{2,N}(H) -   \mathfrak{M}_{s,\m,f_k}^{2,N}(H)\big) \Big)^2\bigg].
\end{equation}
The first term \eqref{tight1fin} vanishes as $N\to\infty$ from Lemma \ref{lem:macro1}.

Since the functions $f_k$ belong to $\cQ$, in particular they are cylinder and therefore the  argument that we used for $\mathcal{M}_{t,\m}^N(e_n)$ to prove \eqref{eq:mart-est1} can be repeated, to obtain that \eqref{tight2fin} is less than $C(\beta,H,f_k,\lambda,\gamma)\; \delta$, thus concluding the proof.\end{enumerate}}

}

\ccl{Now, let us go back to the martingales $\mathfrak{M}_{t,\mathbf{m},f_k}^{1,N}$ defined right after \eqref{eq:decomp}. We prove the following lemma (whose statement was used in the conclusion of Section \ref{ssec:martingale}).

\begin{lem}\label{lem:proofapp}
The sequence $\{\mathfrak{M}_{t,\m,f_k}^{1,N}\}_{N,k}$ is tight. Up to extraction, the martingales $\{\mathfrak{M}_{t,\mathbf{m},f_k}^{1,N}(\cdot)\}_t$ converge as $N \to \infty$ and afterwards $k \to \infty$ to a Gaussian process characterized by 
\[ 
\lim_{k\to\infty}\lim_{N\to\infty} \mathbb{Q}_{\mu_\beta^N} \Big[ \mathfrak{M}_{t,\mathbf{m},f_k}^{1,N}(G)\mathfrak{M}_{t,\mathbf{m},f_k}^{1,N}(H) \Big] = 2 t D \chi(\beta) \int_\mathbb{T} G'(u) H'(u)  \mathrm{d}u.
\]
\end{lem}
\begin{proof} The proof of tightness can be straightforwardly adapted from the previous proof. 
Moreover, standard computations on the stochastic integrals (using the stationarity and translation invariance of $\mu_\beta^N$, and the fact that the Poisson processes are all independent) give 
\begin{align*}
&\mathbb{Q}_{\mu_\beta^N} \Big[ \mathfrak{M}_{t,\mathbf{m},f_k}^{1,N}(G) \mathfrak{M}_{t,\mathbf{m},f_k}^{1,N}(H) \Big]  \\ & = \frac{1}{N}\int_0^t \sum_{x\in\mathbb{T}_N} \nabla_N G\Big(\frac x N \Big) \nabla_N H\Big(\frac x N\Big) \mathbb{Q}_{\mu_\beta^N} \Big[ \lambda \big( \nabla_{x,x+1}(\omega_x^2 - \Gamma_{f_k})(s)\big)^2 + \gamma \big( \nabla_x(\Gamma_{f_k})(s) \big)^2  \Big] \mathrm{d}s\\
 & = \frac{t}{N}\sum_{x\in\mathbb{T}_N} \nabla_N G\Big(\frac x N \Big) \nabla_N H\Big(\frac x N\Big) \mathbb{E}_\beta^\star \Big[ \lambda \big( \nabla_{0,1}(\omega_x^2 - \Gamma_{f_k})\big)^2 + \gamma \big( \nabla_0(\Gamma_{f_k}) \big)^2  \Big] \\
 & \xrightarrow[N\to\infty]{} t \int_\T G'(u)H'(u) \mathrm{d} u  \; \mathbb{E}_\beta^\star \Big[ \lambda \big( \nabla_{0,1}(\omega_x^2 - \Gamma_{f_k})\big)^2 + \gamma \big( \nabla_0(\Gamma_{f_k}) \big)^2  \Big] \\
 & \xrightarrow[k \to\infty]{}  2 t D \chi(\beta) \int_\mathbb{T} G'(u) H'(u) \mathrm{d}u,
\end{align*}
where in the last step we used Proposition \ref{prop:limit}.
\end{proof}
}

%
%

\bibliographystyle{amsplain}


\begin{thebibliography}{99}

\bibitem{MR2784282}
O.~Ajanki and F.~Huveneers, \emph{Rigorous scaling law for the heat current in
  disordered harmonic chain}, Comm. Math. Phys. \textbf{301} (2011), no.~3,
  841--883. 
  
\bibitem{MR2480742}
G.~Basile, C.~Bernardin, and S.~Olla, \emph{Thermal conductivity for a momentum conservative model},
	Communications in Mathematical Physics, \textbf{287} (2009), no.~1, 67--98.
	

\bibitem{MR2448630}
C.~Bernardin, \emph{Thermal conductivity for a noisy disordered harmonic chain}, J.
  Stat. Phys. \textbf{133} (2008), no.~3, 417--433. 



\bibitem{MR3101849}
C.~Bernardin and F.~Huveneers, \emph{Small perturbation of a disordered
  harmonic chain by a noise and an anharmonic potential}, Probab. Theory
  Related Fields \textbf{157} (2013), no.~1-2, 301--331. 

\bibitem{MR2185330}
C.~Bernardin, and S.~Olla, \emph{{Fourier's law for a microscopic model of heat conduction}}, Journal of Statistical Physics, \textbf{121} (2005), no.~3-4, 271--289.
	

\bibitem{MR2904271}
C.~Bernardin and G.~Stoltz, \emph{Anomalous diffusion for a class of systems
  with two conserved quantities}, Nonlinearity \textbf{25} (2012), no.~4,
  1099--1133. 

\bibitem{MR2082192}
F.~Bonetto, J.~L.~ Lebowitz, and J.~Lukkarinen, \emph{Fourier's law for a harmonic crystal with self-consistent stochastic reservoirs},
	{Journal of Statistical Physics}, \textbf{116} (2004), no.~1-4, 783--813.
	
\bibitem{Brezis} H. Brezis, \emph{Functional Analysis, Sobolev Spaces and Partial Differential Equations},
Springer (2010)


	
\bibitem{CL} 
A.~Casher and J.~L.~Lebowitz, \emph{Heat flow in regular and disordered harmonic chains}, Journal of Mathematical Physics, \textbf{12} (1971), no.~8, 1701--1711.

\bibitem{DLK}
A.~Dhar, V.~Kannan, and J.~L. Lebowitz, \emph{Heat conduction in disordered
  harmonic lattices with energy conserving noise}, Phys. Rev. E \textbf{83}
  (2011), no.~021108.


\bibitem{MR838085}
S.~N. Ethier and T.~G. Kurtz, \emph{Markov processes}, Wiley Series in
  Probability and Mathematical Statistics: Probability and Mathematical
  Statistics, John Wiley \& Sons Inc., New York, 1986, Characterization and
  convergence. 

\bibitem{MR2021195}
A.~Faggionato and F.~Martinelli, \emph{Hydrodynamic limit of a disordered
  lattice gas}, Probab. Theory Related Fields \textbf{127} (2003), no.~4,
  535--608. 

\bibitem{F}
J.~Fritz, \emph{Hydrodynamics in a symmetric random medium}, Communications in Mathematical Physics, 
\textbf{125} (1989), no.~1, 13--25.

\bibitem{FFL}
J.~Fritz, T.~Funaki and J.~L.~Lebowitz, \emph{Stationary states of random {H}amiltonian systems}, Probab. Theory Related Fields, \textbf{99} (1994), no.~2, 211--236.

\bibitem{GPV} M.~Z.~Guo, G.~C.~Papanicolau and S.~R.~S.~Varadhan, \emph{Nonlinear diffusion limit for a system with nearest neighbor interactions}, Comm. Math.
Phys. {\bf 118} (1988), 31--59.

\bibitem{HS}
R. Holley and D. Stroock, \emph{Generalized Ornstein --- Uhlenbeck processes as limits of interacting systems},
Stochastic Integrals, Springer (1981), 152--168.


\bibitem{MR2446327}
M.~Jara and C.~Landim,  \emph{Quenched non-equilibrium central limit theorem for a tagged
  particle in the exclusion process with bond disorder}, Ann. Inst. Henri
  Poincar{\'e} Probab. Stat. \textbf{44} (2008), no.~2, 341--361. 

\bibitem{MR1707314}
C.~Kipnis and C.~Landim, \emph{Scaling limits of interacting particle systems},
  Grundlehren der Mathematischen Wissenschaften [Fundamental Principles of
  Mathematical Sciences], vol. 320, Springer-Verlag, Berlin, 1999.

\bibitem{MR1653397}
K.~Komoriya, \emph{Hydrodynamic limit for asymmetric mean zero exclusion
  processes with speed change}, Ann. Inst. H. Poincar{\'e} Probab. Statist.
  \textbf{34} (1998), no.~6, 767--797. 

\bibitem{MR2952852}
T.~Komorowski, C.~Landim and S.~Olla, \emph{Fluctuations in {M}arkov
  processes}, Grundlehren der Mathematischen Wissenschaften [Fundamental
  Principles of Mathematical Sciences], vol. 345, Springer, Heidelberg, 2012,
  Time symmetry and martingale approximation.

\bibitem{MR3810840}
T.~Komorowski, S.~Olla and M.~Simon, \emph{Macroscopic evolution of mechanical and thermal energy in a
              harmonic chain with random flip of velocities}, Kinet. Relat. Models \textbf{11} (2018), no.~3, 615--645.

\bibitem{MR1465163}
C.~Landim and H.~T.~Yau, \emph{Fluctuation-dissipation equation of asymmetric
  simple exclusion processes}, Probab. Theory Related Fields \textbf{108}
  (1997), no.~3, 321--356. 
  
\bibitem{LLR}
J.~L.~Lebowitz, E.~Lieb and Z.~Rieder, \emph{Properties of harmonic crystal in a stationary non-equilibrium state}, J. Math. Phys.,
\textbf{8} (1967), 1073--1078.

\bibitem{MR2540160}
M.~Mourragui and E.~Orlandi, \emph{Lattice gas model in random medium and open
  boundaries: hydrodynamic and relaxation to the steady state}, J. Stat. Phys.
  \textbf{136} (2009), no.~4, 685--714. 

\bibitem{sasolla}
S.~Olla and M.~Sasada, \emph{Macroscopic energy diffusion for a chain of
  anharmonic oscillators},  Probability Theory and Related Fields, {\bf 157} (2013), no.~3-4, 721--775.

\bibitem{P1}
R.~E.~Peierls, \emph{Zur kinetischen Theorie der Wärmeleitung in Kristallen}, Annalen des Physik, \textbf{395} (1929), no.~8, 1055--1101.

\bibitem{P2}
R.~E.~Peierls, \emph{Quantum Theory of Solids}, Oxford University Press, London, 1955.

\bibitem{MR2271489}
J.~Quastel, \emph{Bulk diffusion in a system with site disorder}, Ann. Probab.
  \textbf{34} (2006), no.~5, 1990--2036.


\bibitem{MR0493419}
M.~Reed and B.~Simon, \emph{Methods of modern mathematical physics. {I}.
  {F}unctional analysis}, Academic Press, New York, 1972. 

\bibitem{MR2895558}
M.~Sasada, \emph{Hydrodynamic limit for exclusion processes with velocity},
  Markov Process. Related Fields \textbf{17} (2011), no.~3, 391--428.




\bibitem{Sim}
M.~Simon, \emph{Hydrodynamic limit for the velocity-flip model}, Stochastic
  Processes and their Applications \textbf{123} (2013), 3623--3662.



\bibitem{MR1354152}
S.~R.~S. Varadhan, \emph{Nonlinear diffusion limit for a system with nearest
  neighbor interactions. {II}}, Asymptotic problems in probability theory:
  stochastic models and diffusions on fractals ({S}anda/{K}yoto, 1990), Pitman
  Res. Notes Math. Ser., vol. 283, Longman Sci. Tech., Harlow, 1993,
  pp.~75--128. 



\end{thebibliography}

\end{document}